\documentclass[10pt,a4paper, reqno]{article}
\usepackage{graphicx}
\usepackage{amsmath, color, verbatim}
\usepackage[colorlinks=true]{hyperref}

\usepackage[numbers,sort&compress]{natbib}
\usepackage{amsfonts, mathtools, enumitem} 
\usepackage{enumitem}
\usepackage{titlesec}
\usepackage[hyperref, dvipsnames]{xcolor}
\hypersetup{
  colorlinks=true,
}

\usepackage{amssymb}
\usepackage{amsthm}
\usepackage{thmtools, thm-restate}
\usepackage[left=2cm,right=2cm,top=2cm,bottom=2cm]{geometry}
\DeclareMathOperator*{\esssup}{ess\,sup}

\renewcommand{\d}{\mathrm{d}}

\newcommand{\md}{\partial^\bullet}

\newtheorem{theorem}{Theorem}[section]

\newtheorem{lem}[theorem]{Lemma}
\newtheorem{prop}[theorem]{Proposition}

\newtheorem{remark}[theorem]{Remark}

\newcommand{\R}{\mathbb{R}}

\renewcommand{\bar}{\overline}

\newcommand{\norm}[1]{{\left\Vert #1
		\right\Vert}}

\makeatletter
\newcommand*{\T}{%
  {\mathpalette\@T{}}%
}
\newcommand*{\@T}[2]{%
  \raisebox{\depth}{$\m@th#1\intercal$}%
}
\makeatother

\titleformat{\part}
  {\centering\Large\scshape}{}{1em}{Part \thepart: }

\newcommand{\W}{\mathbb W} %maybe mathbb? W(W^{1,p}...) doesn't look so good 
\newcommand{\N}{\mathbb N}
\theoremstyle{definition}

\allowdisplaybreaks

  \numberwithin{figure}{section}
\numberwithin{table}{section}

\numberwithin{equation}{section}
\numberwithin{theorem}{section}

\title{Regularisation and separation for evolving surface Cahn-Hilliard equations}
\author{Diogo Caetano\thanks{Mathematics Institute, University of Warwick, Coventry, CV4 7AL, UK (\tt{D.Caetano@warwick.ac.uk}, \, \tt{C.M.Elliott@warwick.ac.uk}.)}
\and Charles M. Elliott\,\footnotemark[1]
\\ 
\and Andrea Poiatti\footnotemark[2]}

\author{Diogo Caetano\thanks{Mathematics Institute, University of Warwick, Coventry CV4 7AL, United Kingdom ({\tt Diogo.Caetano@warwick.ac.uk}, {\tt C.M.Elliott@warwick.ac.uk})}  \and 
Charles M. Elliott\footnotemark[1] \and
Maurizio Grasselli\thanks{Dipartimento di Matematica, Politecnico di Milano, Milano, 20133, Italy (\tt{mauri\-zio.grasselli@polimi.it}, \tt{andrea.poiatti@polimi.it})} \and 
Andrea Poiatti\footnotemark[2]}

\renewcommand{\d}{\mathrm d}
\newcommand{\gint}{\int_{\Gamma(t)}}
\newcommand{\tgrad}{\nabla_{\Gamma}}
\newcommand{\tzerograd}{\nabla_{\Gamma_0}}
\newcommand{\invlaprho}{(\Delta_{\Gamma, \rho}^{-1})}

\begin{document}

\hypersetup{
  urlcolor     = blue, 
  linkcolor    = Bittersweet, 
  citecolor   = Cerulean
}

\maketitle
% REQUIRED

\begin{abstract}
  We consider the Cahn-Hilliard equation with constant mobility and
logarithmic potential on a two-dimensional evolving closed surface
embedded in $\R^3$, as well as a related weighted model. The
well-posedness of weak solutions for the corresponding initial value
problems on a given time interval $[0,T]$ have already been
established by the first two authors. Here we first prove some
regularisation properties of weak solutions in finite time. Then, we
show the validity of the strict separation property for both the
problems. This means that the solutions stay uniformly away from the
pure phases $\pm1$ from any positive time on. This property plays an
essential role to achieve higher-order regularity for the solutions.
Also, it is a rigorous validation of the standard double-well
approximation. The present results are a twofold extension of the
well-known ones for the classical equation in planar domains.
\end{abstract}

\section{Introduction}

The Cahn-Hilliard equation in an evolving surface setting has recently been studied in \cite{CaeEll21}. More precisely, having fixed $T>0$, considering a family of closed, connected, oriented surfaces $\Gamma(t) \subset \R^3$ such that its evolution is given a priori as a flow determined by the (sufficiently smooth) velocity field $\mathbf V$,
the evolving surface Cahn-Hilliard equation reads
\begin{equation}\label{intro:model1}
\begin{cases}
\dot u + u \nabla_\Gamma \cdot \mathbf V - \nabla_\Gamma \cdot( u \mathbf V_a^\tau) + \Delta_\Gamma^2 u - \Delta_\Gamma F'(u) = 0, \quad \text{ in } \,\,  \mathcal G_T,\\
u(0)=u_0, \quad \text{ in } \,\, \Gamma(0),
\end{cases}
\end{equation}
where $\mathcal G_T := \bigcup_{t\in [0,T]} \{t\}\times \Gamma(t)$. Here $\mathbf V_a^\tau$ corresponds to the difference between the tangential component of $\mathbf V$ and an advective velocity $\textbf{V}_a$ on the surface. {The quantity $u$ can be interpreted as the difference between the concentration of two immiscible substances which are present on the surface.}
In the same contribution, the following related weighted Cahn-Hilliard system has also been analysed
\begin{align}\label{intro:model2}
\begin{cases}
\dot \rho +\rho \nabla_\Gamma\cdot \mathbf V = 0, \quad \text{ in } \,\,  \mathcal G_T, \\
\rho\dot c - \nabla_\Gamma \cdot \left(\rho\nabla_\Gamma\left( -\dfrac{1}{\rho} \Delta_\Gamma c + F'(c) \right) \right)  &= 0, \quad \text{ in } \,\,  \mathcal G_T,\\
c(0)=c_0,\,\,\rho(0)\equiv 1, \quad \text{ in } \,\, \Gamma(0),
\end{cases}
\end{align}
where $\rho$ is a suitable weight function transported by the surface evolution (in particular, it can be interpreted as the total mass density) {\color{black}and $c$ can be seen as the relative (i.e., dimensionless) concentration difference between the two substances on the surface}.
 In the equations above, equipped with suitable initial conditions, $\tgrad$ denotes the tangential gradient on the surface $\Gamma(\cdot)$, $\Delta_\Gamma$ is the Laplace-Beltrami operator and $\dot u$ denotes the material time derivative of $u$ (see Section \ref{additional} for more details). The functional framework in this work is the same as in \cite{CaeEll21}.
 {\color{black}Model \eqref{intro:model2} is a simplified version of the one presented in \cite{Sauer,ZimTosLan19}, which also includes a coupling with an equation for the surface deformation, i.e., the Kirchhoff-Love thin shell equation. We decided to concentrate on the single Cahn-Hilliard equation in the same formulation arising from the model in \cite{ZimTosLan19}, so that this analysis, although being interesting \textit{per se}, could also be exploited (and extended) to take the evolution of the surface into account as in \cite{Sauer,ZimTosLan19}; in our case, this evolution is given \textit{a priori}. In both models, the main physical property is the conservation of total mass. In particular (see also Remark \ref{conservation} for more details), in the case of model \eqref{intro:model1} this implies, $u$ denoting the difference density, that we have 
 \begin{align}
 \int_{\Gamma(t)}u\equiv \int_{\Gamma_0} u_0,
 \label{mass1}
 \end{align}
 whereas, for model \eqref{intro:model2}, it holds
 \begin{align}
 \int_{\Gamma(t)}\rho c\equiv \int_{\Gamma_0}c_0,
 \label{mass2}
 \end{align}
 $\rho c$ being this time the difference density between the two substances and having assumed $\rho(0)\equiv 1$.
} 

This work aims at proving that weak solutions to equations \eqref{intro:model1} and \eqref{intro:model2} regularise in finite time and enjoy the instantaneous strict separation property from the pure phases, i.e., for each $\tau > 0$, there exists $\delta>0$ depending on $\tau,\ T$ and the data such that 
\begin{align}\label{intro:sep}
    \|u(t)\|_{L^\infty(\Gamma(t))} \leq 1 - \delta, \quad \text{ for a.a. } t\in [\tau,T].
\end{align}
The present results answer to some regularity issues left open (see \cite[5.1.3]{CaeEll21}). 

The classical Cahn-Hilliard equation was introduced in \cite{CahHil58} (see also \cite{Cah61}) for the study of spinodal decomposition in binary alloys. More precisely, it models the phenomenon of phase separation of an alloy of two components when the temperature of the system is quenched to a critical temperature, resulting on a spatially separated two-phase structure. On a smooth bounded domain $\Omega\subset \R^d$, where $d=1,2,3$, it reads as 
\begin{align}\label{intro:ch}
\begin{cases}
    \dot u =\Delta w, \quad \text{ in } \Omega \times (0,T),\\
    w=-\Delta u + F'(u), \quad \text{ in } \Omega \times (0,T),
    \end{cases}
\end{align}
where $w$ is the so called chemical potential
and is complemented with the (no-flux) boundary and initial conditions
\begin{align*}
    \partial_n u = \partial_n w = 0 \,\, \text{ on } \partial \Omega \times (0,T) \quad \text{ and } \quad u(0) = u_0  \,\, \text{ in } \Omega.
\end{align*}
Note that we have assumed constant mobility and set it equal to $1$.
 It has since then inspired numerous works in other areas of science; to name a few, versions of the Cahn-Hilliard equation have been used to study population dynamics \cite{CohMur81}, tumour growth models \cite{KhaSan08} and they have been exploited in image processing analysis \cite{BerEseGil07b,BerEseGil07a} (see also the recent book \cite{Mir19} for other examples). It is also worth recalling that phase separation has recently become a paradigm in cell biology (see, for instance, \cite{Dol18, dolgin2022shape}).
 {\color{black}For example, Cahn-Hilliard equations have been
used to model solid tumour growth (see, e.g., \cite{chatelain2011morphological,lowengrub1,xu2016mathematical} and the book \cite[Part I, Chap. 5]{cristini2010multiscale}), dynamics of plasma membranes
and multicomponent vesicles \cite{Garcke1,Garcke2,Webb,case2022membranes,lowengrub2, Mercker,snead2019control,xue2022phase}.
In some of these cases, such as sorting in biological membranes, phase separation and
coarsening take place in a thin, evolving layer of self-organising molecules, which in
continuum-based approach can be modelled as a material surface. This justifies the
recent interest in the Cahn-Hilliard equation posed on evolving surfaces.
In particular, one of the most interesting phenomenon which has been modelled by evolving-surface Cahn-Hilliard equations (see, e.g., \cite{Embar,ZimTosLan19}) is the lipid rafts formation on cell membranes, which are composed of lipids,
proteins, and cholesterol. Whereas proteins mediate traffic
and serve as signalling devices, lipids provide a fluid matrix
within which transmembrane proteins are free to move. The separation of lipids into two immiscible liquid phases is often linked to the formation of rafts in cell membranes. These rafts are heterogeneous, highly dynamic, sterol, and
sphingolipid-enriched domains that compartmentalise cellular processes and they are thought to be in the liquid-ordered
phase. Rafts are believed to play an important
role in regulating protein activity (\cite{Hess}) that
may in turn affect biological processes such as trafficking
and signalling and are
known to be central to the replication of viruses (\cite{Ono}). All of the above phenomena involve elastic bending of cell membranes being fully coupled with the irreversible
processes of lipid flow, the diffusion of lipids and proteins, and the surface binding of proteins.
Comprehensive membrane models which include these effects are needed to fully understand the complex physical behavior of biological membranes. In our work we make a first stride in the direction of the analysis of these highly complex models. In particular, as already noticed, taking inspiration from such models (see, e.g., \cite{Sauer, ZimTosLan19}), we do not consider the elasticity of the surface, that would determine an evolution equation for the surface itself, which we assume to be given, but we only study the Cahn-Hilliard equation arising from such models. The natural direction for future work is then to consider the fully coupled system, where the evolution of the surface is itself part of problem, see for instance \cite{AbeBurGar22b,AbeBurGar22a, Sauer,ZimTosLan19}.
}
 In equation \eqref{intro:ch}, $u$ is again to be thought of as the difference between the concentrations of the two components in the mixture. The function $F$ is the homogeneous free energy (potential) of the system, and is defined as follows
\begin{align}
    F(r) = \dfrac{\theta}{2} \left( (1+r)\ln(1+r) + (1-r)\ln(1-r) \right) - \dfrac{\theta_0}{2} r^2, \quad r\in [-1,1],
    \label{sing}
\end{align}
where $\theta$, $\theta_0$ are absolute temperatures and satisfy $0<\theta<\theta_0$. This ensures that $F$ has a double-well shape. From the modelling point of view, the logarithmic terms are related to the entropy of the system, while the quadratic term accounts for demixing effects. It is the competition between the two terms that gives rise to the spatially distributed phase separation. The potential defined in this way is called \textit{singular}, whereas many authors  considered a proper approximation, which avoids the fact that $F^\prime$ is unbounded at the pure phases $\pm1$. The most common choice is a polynomial of fourth degree, typically of the form $F(r)=\frac{1}{4}(r^2-u_\theta^2)^2$, where $u_\theta>0$ and $-u_\theta<0$ are the minima of \eqref{sing}; this guarantees that $F$ has a symmetric double-well shape also with minima at $\pm u_\theta$. This is usually referred as a regular potential. This case was also taken into account in \cite{CaeEll21}. However, the polynomial approximation does not ensure the existence of physical solutions, that is, solutions whose values are in $[-1,1]$,
due to the lack of comparison principles for the Cahn-Hilliard equation. We refer to the original proof of well-posedness with \eqref{sing} in \cite{EllLuc91}, as well as the survey articles \cite{Ell89, NovSeg84} and the recent book \cite{Mir19} for an overview of the mathematical results for \eqref{intro:ch}.

As far as the evolving surface version is concerned, both systems \eqref{intro:model1} and \eqref{intro:model2} are treated in \cite{CaeEll21}, where the authors establish existence, uniqueness and stability of solutions for the different cases where $F$ is a smooth, logarithmic or double obstacle potential, respectively. Some regularity results are also proved. We refer also to \cite{EllRan15, OlsXuYus21,YusQuaOls20} for different derivations of the equation and some numerical results, and to \cite{ZimTosLan19} for a rigorous modelling source for the weighted system \eqref{intro:model2}. {\color{black} Nevertheless, for the sake of completeness we will give a short derivation of each model in Sections \ref{derivation1} and \ref{derivation2}, respectively.} Interest in these equations is part of the more general problem of considering partial differential equations on domains or surfaces that evolve in time which is presently being vastly studied, since these have been seen to provide more realistic models for physical and biological phenomena. Some examples are \cite{BarEllMad11, EilEll08, EllStiVen12, ErlAziKar01, GarLamSti14, VenSekGaf11}. Not only are these relevant for applications, but they also raise interesting modelling, numerical and computational questions, as well as challenging problems from the point of view of mathematical analysis. We refer in particular to \cite{EllRan21} for a detailed exposition of the numerical analysis of such problems and to 
\cite{AlpCaeDjuEll21, AlpEllSti15a, AlpEllSti15b} for an abstract functional framework well-suited for the treatment of such problems. 

The importance of establishing the strict separation property is twofold. 
\begin{itemize}
    \item First it is essential when one considers higher-order regularity of weak solutions, due to the behaviour of $F$ and its derivatives close to $\pm 1$. Indeed, note that 
\begin{align*}
    F'(r) = \dfrac{\theta}{2} \ln \dfrac{1+r}{1-r} - \theta_0 r \quad \text{ and } \quad F''(r) = \dfrac{\theta}{2} \left( \dfrac{1}{1+r} + \dfrac{1}{1-r}\right) - \theta_0
\end{align*}
are both singular when $r\to \pm 1$, and even though the structure of $F'$ can be exploited in order to obtain some estimates, it is more challenging to do so for $F''$. As noted in, e.g., \cite{GioGraMir17}, we have 
\begin{align}\label{intro:growth}
    F''(r) \leq C e^{C |F'(r)|},
\end{align}
which precludes us from controlling $F''$ in $L^p$-spaces in terms of the $L^p$-norms of $F'$. It is then important to study conditions that ensure the integrability of $F''$. It is clear that establishing strict separation from the pure phases as in \eqref{intro:sep} is crucial to achieve this. Furthermore, the strict separation property is a fundamental ingredient in the study of longtime behavior of solutions (see \cite{GioGraMir17,MirZel04}).
\item Secondly, if the strict separation property holds then, being the solution away from $\pm 1$, the logarithmic potential $F$ is smooth and can be dominated by a polynomial. Therefore, this result can be viewed as a rigorous justification of the usual aforementioned polynomial double-well approximation.     
\end{itemize}   

Separation from the pure phases for dimensions $d\geq 3$ is unknown even in the fixed domain setting, in the case of constant mobility. This apparently technical restriction is related to the growth condition \eqref{intro:growth} (see \cite{GalGioGra22} for a detailed analysis). As a consequence, our results are restricted to two-dimensional surfaces. The strict separation for $d=2$ was first established for \eqref{intro:model1} in \cite{MirZel04}. Then, a more general argument was introduced in \cite{GioGraMir17}. For an up-to-date picture of the state-of-the-art, the reader is referred to \cite{GalGioGra22}, where new proofs and further generalizations are given.

Here, not only we extend the result for two-dimensional planar domains for equation \eqref{intro:model1}, but we also prove the first separation result for the weighted model \eqref{intro:model2}. 

This article on one hand complements \cite{CaeEll21} by proving instantaneous regularisation of the weak solutions and on the other hand extends the validity of the strict separation property for the local Cahn-Hilliard equation with constant mobility to the setting of evolving surfaces in $\R^3$ in two cases. It is also worth observing that our approach to the estimates for solutions to the approximate problems allows us to forgo some of the assumptions made in \cite{CaeEll21} (see, e.g., \cite[Assumption $A_P$]{CaeEll21}), generalising the results therein. For the sake of completeness, here we also include the proof of continuous dependence on the initial data, entailing uniqueness for the general system \eqref{intro:model2}, which was omitted in \cite{CaeEll21}. 

The paper is structured as follows. In Section \ref{additional} we briefly recall the setting from \cite{CaeEll21} on the evolution of the surfaces and state some additional regularity assumptions. Section \ref{firstmodel} is devoted to the analysis of the first system \eqref{intro:model1} and we establish higher-order regularity for the solution $u$ and prove that it satisfies the strict separation property. Finally, in Section \ref{secondmodel} we introduce a better suited Galerkin approximation for the alternative weighted model \eqref{intro:model2} and prove analogous regularity results as well as the strict separation property. We also include three appendices for the sake of completeness: in Appendix \ref{app:geom} we present the proofs of some propositions which are only stated in the main body of the paper. Appendix \ref{app:prelimtools} collects some preliminary results essential to obtain the control of higher-order derivatives of the logarithmic potential (in particular, we report a Moser-Trudinger type inequality valid on compact Riemannian manifolds and a generalized Young's inequality), whereas in Appendix \ref{app:embedding} we show the validity of a well-known embedding result for Bochner spaces also in the evolving space setting. 

\section{Surface motion: assumptions}
\label{additional}
We refer to the setting and the notation of \cite{DziEll13-a}. To be more precise, we consider $T>0$ and a $C^2$-evolving surface $\{\Gamma(t)\}_{t\in [0,T]}$ in $\R^3$, i.e. a closed, connected, orientable $C^2$-surface $\Gamma_0$ in $\R^3$ together with a smooth \textit{flow map} 
\begin{align*}
\Phi\colon [0,T]\times \Gamma_0 \to \R^3
\end{align*}
such that 
\begin{itemize}
\item[(i)] denoting $\Gamma(t) := \Phi_t^0(\Gamma_0)$, the map $$\Phi_t^0 := \Phi(t, \cdot)\colon \Gamma_0\to \Gamma(t)$$ is a $C^2$-diffeomorphism, with inverse map $$ \Phi_0^t \colon \Gamma(t)\to \Gamma_0;$$
\item[(ii)] $\Phi_0^0 = \text{id}_{\Gamma_0}$. 
\end{itemize}
%
%a family of closed, connected and orientable surfaces $\{\Gamma(t)\}_{t\in [0,T]}$ in $\R^{3}$ which evolve in time with a prescribed normal $C^2$ velocity vector field $\mathbf{V_\nu}$. We call such a family $\{\Gamma(t)\}_{t\in [0,T]}$ an \textit{evolving $C^2$-surface}. 
It follows from the definition above that, for each $t\in [0,T]$, $\Gamma(t)$ is also a closed, connected, orientable $C^2$-surface.
In addition to the assumptions on the surface motion, we {\color{black} first recall the same hypothesis as \cite[$\mathbf{A}_\Phi$]{CaeEll21}:

\vskip4mm
\textbf{Assumption $\mathbf{A}_\Phi$}: There exists a velocity field $\textbf{V}:[0, T ] \times \R^3 \to \R ^3$ with regularity
$$
\textbf{V}\in C^0([0,T];C^2(\R^3;\R^3)),
$$
such that, for any $t\in [0,T]$ and every $x\in \Gamma_0$, 
\begin{align}
\dfrac{d}{dt} \, \Phi_t^0(x) &= \mathbf{V}\left(t, \Phi_t^0(x)\right), \quad \text{ in } [0,T] \\
\Phi_0^0(x) &= x.
\end{align}

In addition to $\mathbf{A}_\Phi$ we will also suppose, where necessary, that the velocity field $\mathbf{V}\colon [0, T]\times \R^3\to \R^3$ satisfies

\vskip4mm
\textbf{Assumptions $\textbf{B}_\Phi$:}
\label{add}
\begin{enumerate}
	
	\item  $\mathbf{V}$ is such that
    \begin{align}
	\mathbf{V}\in C^1([0,T];C^2(\R^3,\R^3));
	\label{V}
	\end{align}
	\item The advective tangential velocity $\mathbf{V}_a$ is such that
     \begin{align}
     \label{VA}
     \mathbf{V}_a\in C^1([0,T];C^1(\R^3,\R^3)).
     \end{align}
	 \end{enumerate}
	 
	Note that the above assumptions $\textbf{B}_\Phi$ require more regularity in time for the map $\Phi_0^{(\cdot)}$ than \textbf{A}$_{\Phi}$.} Denoting by $\mathbf V_\tau$ and $\mathbf V_\nu$ the tangential and normal components of $\mathbf V$, respectively, {\color{black} Assumptions $\textbf{B}_\Phi$} imply, in particular, for $t\in[0,T],$
		\begin{align}
	&\Vert \mathbf{V}_\tau(t)\Vert_{C^2(\Gamma(t))},\Vert \mathbf{V}_\nu(t)\Vert_{C^2(\Gamma(t))}\leq \Vert \mathbf{V}(t)\Vert_{C^2(\Gamma(t))}\leq C_\mathbf{V},\nonumber\\
    &\Vert \mathbf{V}_a(t)\Vert_{C^0(\Gamma(t))}, \Vert \partial^\bullet\mathbf{V}_a^\tau(t)\Vert_{C^0(\Gamma(t))}\leq C_\mathbf{V},
    \label{es1}
    \end{align}
    for $t\in [0,T]$ and for some $C_\mathbf{V}>0$ independent of time, where (see \cite{CaeEll21}) $\mathbf{V}_a^\tau:=\mathbf{V}_\tau-\mathbf{V}_a$.
    {\color{black}whereas the single extra Assumption $\textbf{B}_\Phi.2.$ implies 
    \begin{align}
	\Vert \partial^\bullet\mathbf{V}_a(t)\Vert_{C^0(\Gamma(t))}\leq C_\mathbf{V},
	\label{es22}
	\end{align}
	 for $t\in [0,T]$ and for some $C_\mathbf{V}>0$ independent of time.}
	  
	  In this setting we can define the normal material time derivative of a scalar quantity $u$ on $\Gamma(t)$ by
	 \begin{align*}
	     \partial^\circ u := {\color{black}\hat {u}}_t + \nabla {\color{black}\hat{u}} \cdot \mathbf V_\nu,
	 \end{align*}
	 where ${\color{black}\hat {u}}$ denotes any extension of $u$ to a (space-time) neighbourhood of $\Gamma(t)$, and its full material time derivative, or strong time derivative, as
\begin{align}\label{eq:timederiv}
\partial^\bullet u := \partial^\circ u + \nabla_\Gamma u \cdot \mathbf V_\tau = {\color{black}\hat {u}}_t + \nabla {\color{black}\hat {u}} \cdot \mathbf{V},
\end{align}
 These definitions take into account not only the evolution of the quantity $u$ but also the movement of points in the surface. The definition in \eqref{eq:timederiv} can be abstracted to a weaker sense as follows. Let $u\in L^2_{H^1}$. A functional $v\in L^2_{H^{-1}}$ is said to be the weak time derivative of $u$, and we write $v=\partial^\bullet u$, if, for any $\eta\in \mathcal{D}_{H^1}(0,T)$, we have
\begin{align}
\int_0^T \langle v(t), \eta(t)\rangle_{H^{-1}\times H^1}  = -\int_0^T (u(t),\md\eta(t))_{L^2} - \int_0^T \int_{\Gamma(t)} u(t)\eta(t) \nabla_{\Gamma(t)}\cdot \mathbf{V}(t).
\end{align}
Observe that $\md \eta$ is the strong material derivative of $\eta$. We will use $\md u$ for both the strong and weak material derivatives.

	 Thanks to \cite[Lemma 2.6]{DziKroMul13}, we  observe that, for a sufficiently regular vector field $\mathbf{f}$, there holds
	 $$
	 \partial^\bullet (\tgrad \cdot \mathbf{f})=\partial^\bullet \underline{D}^{\Gamma(t)}_l\mathbf{f}_l=\underline{D}^{\Gamma(t)}_l\partial^\bullet\mathbf{f}_l-A_{lr}(\mathbf{V})\underline{D}^{\Gamma(t)}_r\mathbf{f}_l,\quad l,r=1,2,3.
	 $$
	 Here $\underline{D}^{\Gamma(t)}_i$, $i=1,2,3$, is the $i^{th}$ component of the tangential gradient and
	 $$
	 A_{lr}(\mathbf{V})=\underline{D}^{\Gamma(t)}_l\mathbf{V}_r-\nu_s\nu_l\underline{D}^{\Gamma(t)}_r\mathbf{V}_s,\quad s=1,2,3,
	 $$
	 with $\nu$ standing for the normal vector field to the surface.
	 Therefore,  by \eqref{V} we can set $\mathbf{f}=\mathbf{V}$ in the previous equality and deduce that, due to the regularity of $\Gamma(t)$, we also have
	 		\begin{align}
	 \Vert \partial^\bullet\tgrad \cdot \mathbf{V}(t)\Vert_{C^0(\Gamma(t))}\leq C_\mathbf{V}, \quad \text{ for any }t\in[0,T].
	 \label{es2}
	 \end{align}
	  Moreover, denoting by $J_t^0$ and $J_0^t$, the change of area element from $\Gamma_0$ to $\Gamma(t)$ and the one from $\Gamma(t)$ to $\Gamma_0$, respectively, we have, for any $\eta:\Gamma(t)\to \R$,
	 $$
	 \int_{\Gamma(t)}\eta d\Gamma=\int_{\Gamma_0}\widetilde{\eta}J_t^0 d\Gamma_0,
	 $$
	 with $\widetilde{\eta}(p)=\eta(\Phi_t^0(p)), \text{ for any }p\in \Gamma_0$. Note that {\color{black}$J_t^0(\cdot)=\left\vert\text{det} D_{\Gamma_0}\Phi_t^0(\cdot)\right\vert$}. %{\color{purple}(should be absolute value of determinant?)} 
	 Then, due to the previous assumptions, we deduce
	 \begin{align}
	 \frac{1}{C_J}\leq \Vert J_t^0\Vert_{C^0(\Gamma(t))}\leq C_J, \quad\text{ for any }t\in[0,T],
	 \label{equiv}
	 \end{align}
	 where $C_J$ can be chosen to be independent of $t$. This also implies that there exists a constant $C_\Gamma>0$, independent of $t$, such that
	 \begin{align}\label{eq:surfacerelation}
	 \frac{\vert\Gamma_0\vert}{C_\Gamma}\leq \vert\Gamma(t)\vert\leq C_\Gamma\vert\Gamma_0\vert,
	 \end{align}
	 for any $t\in[0,T]$. Here $\vert \Gamma_0\vert$ stands for the Lebesgue surface measure. For any integrable function $v$ over a surface $\Gamma$ of positive measure, we set
$$
(v)_{\Gamma}=\dfrac{1}{\vert\Gamma\vert}\int_{\Gamma}v.
$$

\textbf{General agreement.} The symbol $C>0$ will denote a generic constant, depending only on the structural parameters of the problem and $T$, but independent of time $t$ and of the approximating indices $\delta,M$ (unless otherwise
specified).

\section{The first model}
\label{firstmodel}
	In this section we consider system \eqref{intro:model1}, which is a model proposed, e.g., in \cite{OcoSti16} and \cite[Problem 4.1]{CaeEll21}. We start by briefly recalling its derivation.
\subsection{Derivation}
\label{derivation1}{\color{black}
Fix $t\in [0,T]$ and consider a scalar quantity $u=u(t)\colon \Gamma(t) \to \R$, to be thought of, in our context, as the concentration difference between two immiscible substances present in a mixture on the surface. The main property for $u$ is that its total mass is conserved, so that \eqref{mass1} holds. Starting from the balance law 
\begin{align*}
    \dfrac{d}{dt} \int_{P(t)} u = -\int_{\partial P(t)} q \cdot \mu,
\end{align*}
on every portion $P(t)\subset \Gamma(t)$ evolving under the purely normal velocity $\mathbf V_\nu$, where $q$ is a flux to be defined later on, we are led to
\begin{align*}
    \int_{P(t)} \partial^\circ u + u \, \text{div}_\Gamma\mathbf V_\nu + \text{div}_\Gamma q = 0.
\end{align*}
Since this holds for every region we obtain the equation
\begin{align*}
    \partial^\circ u + u \, \text{div}_\Gamma\mathbf V_\nu + \text{div}_\Gamma q = 0 \quad \text{ on } \Gamma(t).
\end{align*}
Now, the flux $q$ consists of an advective term $q_a = u \mathbf V_a$, where $\mathbf V_a$ is tangential and represents the particle velocity in the fluid, and a diffusive part $q_d = -\tgrad( -\Delta_\Gamma u + F'(u) )$. Substituting in the equation above leads to
\begin{align}\label{eq:CH_model1_normal}
    \partial^\circ u + u \, \text{div}_\Gamma \mathbf V_\nu + \text{div}_\Gamma (u\mathbf V_a)  = \Delta_\Gamma \left(  -\Delta_\Gamma u + F'(u)\right).
\end{align}
This is the \textit{physical} equation that we wish to consider. In order to apply the functional framework developed in \cite{AlpCaeDjuEll21}, we work under Assumption $\mathbf{A_\Phi}$. To incorporate the tangential component $\mathbf V_\tau$, which is to be interpreted as an \textit{arbitrary parametrisation} of the surface $\Gamma(t)$, we add and subtract a term $\text{div}_\Gamma (u \mathbf V_\tau)$ to \eqref{eq:CH_model1_normal}, resulting in
\begin{align}\label{eq:CH_model1}
    \partial^\bullet u + u \, \text{div}_\Gamma \mathbf V + \text{div}_\Gamma (u(\mathbf V_a-\mathbf V_\tau))  = \Delta_\Gamma \left(  -\Delta_\Gamma u + F'(u)\right),
\end{align}
or equivalently
\begin{align*}
    \partial^\bullet u + u \, \text{div}_\Gamma \mathbf V + \text{div}_\Gamma (u(\mathbf V_a-\mathbf V_\tau))  &= \Delta_\Gamma w, \\
    -\Delta_\Gamma u + F'(u) &= w.
\end{align*}}
\subsection{Weak formulation}
\label{weakk}
In order to introduce the variational formulation of the Cahn-Hilliard systems we need to introduce the abstract framework of \cite{AlpCaeDjuEll21, AlpEllSti15a} for the definition of the time-dependent function spaces. We do it in a summarised way and refer the reader to \cite{AlpCaeDjuEll21, AlpEllSti15a} for a more rigorous and detailed explanation of how these are constructed and how they can be abstracted to a more generalised setting.

We use the flow map $\Phi_t^0$ to define pullback and pushforward operators
\begin{align*}
\phi_{-t} u = u \circ \Phi_t^0 \colon \Gamma_0 \to \R \,\, &\text{ for } u\colon \Gamma(t)\to \R, \\
\phi_{t} v = v \circ \Phi_0^t \colon \Gamma(t) \to \R \,\, &\text{ for } v\colon \Gamma_0\to \R
\end{align*}
and suppose, for $t\in [0,T]$, $X(t)$ to be a Banach space of functions over $\Gamma(t)$; typically $L^2(\Gamma(t))$ or some higher-order Sobolev space $H^k(\Gamma(t))$, but also other $L^p(\Gamma(t))$ or even the dual space $H^{-1}(\Gamma(t)):=(H^1(\Gamma(t))^*$. We consider functions of the form 
\begin{align*}
u\colon [0, T]&\to \bigcup_{t\in[0,T]} X(t)\times \{t\}, \quad t\mapsto (\bar{u}(t), t)
\end{align*}
and identify $u(t)\equiv \overline u(t)$; in practice, we want to see $u(t)$ as an element of the space $X(t)$. The function spaces are defined as follows. 
\begin{itemize}
\item[(i)] For $p\in [1,\infty]$, $u\in L^p_{X}$ if  $t\mapsto \phi_{-(\cdot)}\bar{u}(\cdot)\in L^p(0, T; X_0)$ with the norm
 \begin{align*}
     \|u\|_{L^p_X} := \begin{cases}
         \left( \int_0^T \|u(t)\|_{X(t)}^p \right)^{1/p} &\text{ if } p<\infty \\
         \esssup_{t\in [0,T]} \|u(t)\|_{X(t)} &\text{ if } p=\infty
     \end{cases};
 \end{align*}
 if $X(t)=H(t)$ are Hilbert spaces, then so is $L^2_H$ with the inner product
\begin{align*}
(u,v)_{L^2_{H}} &:= \int_0^T (u(t), v(t))_{H(t)}. 
\end{align*} 

\item[(ii)] For $k\in \N\cup\{0\}$, $u\in C^k_{X}$ if $t\mapsto \phi_{-t}u(t)\in C^k([0,T]; X_0)$ and we define its time derivatives as 
\begin{align*}
    \partial^{\bullet, s} u(t) = \phi_t \dfrac{d^s}{dt^s} \phi_{-t} u(t), \quad s=1, \dots, k.
\end{align*}
We will denote $\partial^{\bullet, 1} = \partial^\bullet$; it is easy to see that the latter produces the same definition as in \eqref{eq:timederiv}. These spaces are endowed with the norm
\begin{align*}
\|u\|_{C^k_{X}} := \sup_{t\in [0,T]} \|u(t)\|_{X(t)} + \sum_{s=1}^k \sup_{t\in [0,T]} \|\partial^{\bullet, s} u(t)\|_{X(t)}.
\end{align*} 

\item[(iii)] A test function $u\in \mathcal{D}_{X}(0,T)$ if $t\mapsto \phi_{-t}u(t)\in C_c^\infty((0,T); X_0).$ \vskip 2mm

\item[(iv)]% We define the weak version of the time derivative \eqref{eq:timederiv} for functions $u\in L^2_{H^1}$; for the general case see \cite{AlpCaeDjuEll21}. An element $v\in L^2_{H^{-1}}$ is said to be the \textit{weak time derivative} of $u$, and we write $v=\partial^\bullet u$, if, for any $\eta\in \mathcal{D}_{H^1}(0,T)$, we have
%\begin{align}
%\int_0^T \langle v(t), \eta(t)\rangle_{H^{-1},\, H^1}  = -\int_0^T (u(t),\md\eta(t))_{L^2} - \int_0^T \int_{\Gamma(t)} u(t)\eta(t) \nabla_{\Gamma}\cdot \mathbf{V}(t).
%\end{align}
We define the Banach space
\begin{align*}
H^1_{H^{-1}} := \left\{ u\in L^2_{H^1} \colon \md u\in L^2_{H^{-1}}\right\} \, \text{ with } \,\, \|u\|_{H^1_{H^{-1}}} := \|u\|_{L^2_{H^1}} + \| \partial^\bullet u \|_{L^2_{H^{-1}}}.
\end{align*}

\item[(v)] Similarly, $H^1_{L^2}$ denotes the space of those $u\in L^2_{H^1}$ which have a more regular weak time derivative $\md u\in L^2_{L^2}$. \vskip 2mm
\end{itemize}

\begin{remark}
We point out that in order to define the spaces $L^p_X$ we are implicitly assuming that the family $(\phi_t, X(t))_{t\in [0,T]}$ is compatible in the sense of \cite[Assumption 3.1]{CaeEll21}. Moreover, in the cases where $X(t)$ is a space of functions with some regularity, it is also assumed that the differentiation makes sense on $\Gamma(t)$. In our setting, these assumptions require appropriate regularity of the flow map $\Phi$ and the velocity field $\mathbf V$. More precisely, if for some $k\in\N$ we have
\begin{align*}
        \Phi_{(\cdot)}^0, \Phi_0^{(\cdot)}\in C^1([0,T]; C^k(\R^3, \R^3)) \quad \text{ and } \quad \mathbf V\in C^0([0,T];C^k(\R^3, \R^3)),
    \end{align*}  
then:
\begin{itemize}
    \item[(i)] $\Gamma(t)$ is a $C^k$-surface, and it makes sense to define strong and weak derivatives up to order $k$ for functions on $\Gamma(t)$ (in our case we have $k=2$);
    \item[(ii)] the pair $(\phi_t, X(t))_t$ is compatible where $X=L^p, \, C^0, \, H^{-1}, \, W^{r,p}, \, C^{r}$ for $p\in [1,\infty]$ and $1\leq r\leq k$ (see e.g. \cite[Section 7]{CaeEll21} for examples).
\end{itemize}
Notice that the assumptions in Sec.\ref{additional} are enough to guarantee that all the function spaces involved are well defined, as well as to ensure compatibility with the evolution.
\end{remark}

Let us now explicitly recall some notation introduced in \cite[Section 3.2]{CaeEll21}. In particular, for $t\in[0,T]$, we define the following bilinear forms:
\begin{enumerate}
	\item for $\eta,\phi\in L^2(\Gamma(t))$, the zero order terms
	$$
	m(t;\eta,\phi):=\int_{\Gamma(t)}\eta\phi, \quad
	g(t;\eta,\phi):=\int_{\Gamma(t)}\eta\phi\tgrad \cdot \mathbf{V}(t);
	$$
\item for $\eta\in L^2(\Gamma(t))$, $\phi\in H^1(\Gamma(t))$, the first order term
$$
a_N(t;\eta,\phi):=\int_{\Gamma(t)}\eta \mathbf{V}_a^\tau(t)\cdot \tgrad  \phi;
$$
\item for $\eta,\phi\in H^1(\Gamma(t))$, the second order terms
$$
a_S(t;\eta,\phi):=\int_{\Gamma(t)}\tgrad \eta\cdot \tgrad \phi, \quad b(t;\eta,\phi):=\int_{\Gamma(t)}B(\mathbf{V}(t))\tgrad \eta\cdot \tgrad \phi,
$$
with ${B}(\mathbf{V})=(\tgrad \cdot \mathbf{V})\mathbf{I}-2\mathbf{D}(\mathbf{V})$, where $\mathbf{D}(\mathbf{V})$ stands for the symmetrized rate tensor;
\item for $\eta\in H^{-1}(\Gamma(t))$, $\phi\in H^1(\Gamma(t))$, the duality pairing
$$
m_\star(t;\eta,\phi):=\langle\eta,\phi\rangle_{H^{-1}(\Gamma(t)),H^1(\Gamma(t))}.
$$
\end{enumerate}
For the sake of simplicity, from now on we will omit the explicit dependence on time and we will denote $\Vert\cdot\Vert_{L^2(\Gamma(t))}$ only by $\Vert\cdot\Vert$.
In conclusion, %{\color{black}[TO CHECK THE PROOF]},
we recall \cite[Prop.2.8]{CaeEll21}, see also \cite[Sec.8.2]{DziEll13-a}) in
\begin{prop}\label{prop:transport}
	\label{propder}
	The following transport formulas hold:
	\begin{itemize}
	    \item[(a)] For $\eta, \phi\in H^1_{H^{-1}}$, 
	    \begin{align*}
	        \dfrac{d}{dt} m(\eta,\phi) = m_\ast(\partial^\bullet \eta, \phi) + m_\ast(\partial^\bullet \phi, \eta) + g(\eta,\phi).
	    \end{align*}
	    \item[(b)] If additionally $\tgrad \partial^\bullet \eta, \tgrad\partial^\bullet \phi\in L^2_{L^2}$, then
	    \begin{align*}
	        \dfrac{d}{dt} a_S (\eta, \phi) =a_S(\partial^\bullet \eta, \phi) + a_S (\eta, \partial^\bullet \phi) + b(\eta, \phi).
	    \end{align*}
	    \item[(c)] For $\eta,\phi\in H^1_{L^2}$, with $\tgrad \partial^\bullet \phi\in L^2_{L^2}$, the following identity holds
\begin{align}
\nonumber
\frac{d}{dt}a_N(\eta,\phi)&=a_N(\partial^\bullet\eta,\phi)+a_N(\eta,\partial^\bullet\phi)\\
&+\int_{\Gamma(t)}B_{adv}(\mathbf{V}_a^\tau(t),\mathbf{V}(t))\eta\cdot \tgrad \phi,
\label{der}
\end{align}
for almost any $t\in [0,T]$, where $B_{adv}$ is a vector field given by
$$
B_{adv}(\mathbf{V}_a^\tau,\mathbf{V})_i:=(\partial^\bullet\mathbf{V}_a^\tau)_i+(\tgrad \cdot \mathbf{V})(\mathbf{V}_a^\tau)_i-\sum_{j=1}^3(\mathbf{V}_a^\tau)_j\underline{D}^{\Gamma(t)}_j\mathbf{V}_i,\quad i=1,2,3.
$$
	\end{itemize}

\end{prop}
In the above we use the same notation as in \cite{EllRan21}. We postpone the proof of this result to Appendix \ref{app:geom}.

{\color{black}Let us now recall the compatibility condition on the initial datum $u_0$ in order to obtain the existence of a (weak) solution to the Cahn-Hilliard equation with logarithmic potential. As noted in \cite{CaeEll21}, there is an interplay between the evolution of the surfaces $\Gamma(t)$ and the admissible initial conditions. In effect, setting $$m_\eta(t):=\frac{1}{\vert\Gamma(t)\vert}\left\vert\int_{\Gamma_0}\eta\right\vert,$$ we require that the initial datum $u_0$ satisfies 
\begin{align}\label{eq:compcondition}
\max_{t\in[0,T]}m_{u_0}(t)<1.    
\end{align}
Although it might seem unnatural to prescribe a condition involving the initial data and the area of the surfaces at all future times, this assumption has an interesting physical meaning. Let us introduce the maximum shrinkage ratio $S_R$ on $[0,T]$, which is \textit{a priori} prescribed by the evolution of the surfaces $\Gamma(t)$, as:
$$
S_R:=\max_{t\in[0,T]}\frac{\vert\Gamma_0\vert}{\vert\Gamma(t)\vert}.
$$
Condition \eqref{eq:compcondition} then means that
\begin{align}
|(u_0)_{\Gamma_0}|S_R<1.
\label{cond}
\end{align}
Since at time $t=0$ we already assume $|(u_0)_{\Gamma_0}|<1$ for the initial mass (meaning that we do not start with a one-phase mixture), condition \eqref{cond} shows that the absolute value of the initial mass, $|(u_0)_{\Gamma_0}|$, compensates for the maximum shrinkage ratio $S_R$ of the evolving surface $\Gamma(t)$ on $[0,T]$. Therefore, the higher the value of $S_R$ (i.e., the smaller the surface $\Gamma(t)$ becomes over $[0,T]$), the further $u_0$ must be from the pure phases $\pm1$ in $\Gamma_0$.

Moreover, recalling that the Cahn-Hilliard dynamics implies that the total mass of the solution is preserved, namely, $$\int_{\Gamma(t)}u(t)\equiv \int_{\Gamma_0}u_0, \quad \text{for all }t\in[0,T],$$ we can write
$$
\vert(u(t))_{\Gamma(t)}\vert=m_{u_0}(t)
$$
and thus the condition $m_{u_0}(t)<1$ for every $t\in[0,T]$ is the dynamic counterpart of
$$
\vert(u(t))_{\Omega}\vert= \left\vert \dfrac{1}{|\Omega|}\int_\Omega u(t) \,dx\right\vert = \left\vert\dfrac{1}{|\Omega|}\int_\Omega u_0 \, dx\right\vert=\vert(u_0)_\Omega\vert<1
$$
which holds in the case of a static bounded domain $\Omega$.

For future use, let us then define the set $\mathcal{I}_0$ of admissible initial conditions as
$$
\mathcal{I}_0:=\left\{\eta\in H^1(\Gamma_0):  \vert\eta\vert\leq 1\text{ a.e. on }\Gamma_0,\ \text{and } \left\vert(\eta)_{\Gamma_0}\right\vert S_R<1\right\}. 
$$}
We recall that the free energy is given by
$$
E^{CH}[u]=\int_{\Gamma(t)}\left(\frac{\vert\nabla u\vert^2}{2}+F(u)\right),
$$
with $F$ logarithmic potential. Note that, if $u_0\in \mathcal{I}_0$, then $u_0$ has finite energy. Indeed
$\Vert   u_0\Vert_{L^\infty(\Gamma_0)}\leq 1$ implies $F(u_0)\in L^1(\Gamma_0)$.
We can now recall the result obtained in \cite[Theorems 5.14, 5.15]{CaeEll21} for
\begin{align}
\label{logpot}
F(s)=\frac{\theta}{2}((1+r)\ln(1+r)+(1-r)\ln(1-r))+\frac{1-r^2}{2},\quad r\in(-1,1),
\end{align}
where $0<\theta<1$ ($F$ is extended by continuity at $r=\pm1$). We set $F_{ln}(r):=(1+r)\ln(1+r)+(1-r)\ln(1-r)$ for the sake
of simplicity and denote $\varphi(s)=F'_{ln}(s)$. Note that this choice corresponds to the singular potential \eqref{sing} with $\theta_0=1$, up to a translation to get directly {\color{black}$F\geq 0$} in $[-1,1]$.

\begin{remark}
 Concerning the logarithmic potential, there exists a constant $C>0$ such that
	\begin{align}
	\label{explog}
	\frac{\theta}{2}\varphi^\prime(s)\leq e^{C\left\vert\frac{\theta}{2}\varphi(s)\right\vert+C},\quad\forall s\in(-1,1).
	\end{align}
\end{remark}
{\color{black} We now let Assumption $\mathbf{A}_\Phi$} hold: we have (see \cite{CaeEll21})
\begin{theorem}
	\label{existence}
	Let $u_0\in \mathcal{I}_0$ and $F:[-1,1]\to \R$ be given by \eqref{logpot}. Then there exists a unique pair $(u,w)$ with
	$$u\in L^\infty_{H^1}\cap H^1_{H^{-1}}\quad \text{ and } \quad w\in L^2_{H^1},$$ such that, for almost any $t\in(0,T]$, $\vert u(t)\vert<1$ almost everywhere on $\Gamma(t)$ and $u$ satisfies, for all $\eta\in L^2_{H^1}$ and almost any $t\in[0,T]$,
	\begin{align}
	&m_\star(\partial^\bullet u, \eta)+g(u,\eta)+a_N(u,\eta)+a_S(w,\eta)=0,\label{CH}\\&
	a_S(u,\eta)+\frac{\theta}{2}m(\varphi(u),\eta)-m(u,\eta)-m(w,\eta)=0,\label{mu}
	\end{align}
	where $u(0)=u_0$ almost everywhere in $\Gamma_0$. The solution $u$ also satisfies the additional regularity
	$$u\in C^0_{L^2}\cap L^\infty_{L^p}\cap L^2_{H^2},$$
	for all $p\in[1,+\infty)$. Furthermore, if $u_0,v_0\in \mathcal{I}_0$ are such that $(u_0)_{\Gamma_0}=(v_0)_{\Gamma_0}$, and $u,v$ are the solutions of the system with $u(0)=u_0$ and $v(0)=v_0$, then there exists a constant $C>0$ independent of $t$, such that, for almost any $t\in[0,T]$,
	$$\Vert u(t)-v(t)\Vert_{H^{-1}(\Gamma(t))}\leq e^{Ct}\Vert u_0-v_0\Vert_{H^{-1}(\Gamma(t))}.$$
\end{theorem}
\begin{remark}
\label{LfourHtworeg}
We notice that the regularity stated in Theorem \ref{existence} can be slightly improved. In particular, since $u\in L^2_{H^2}$ solves the problem, for almost all $t\in [0,T]$,
$$
-\Delta_{\Gamma}u(t)=w(t)-F^{\prime}(u(t))\in L^2({\Gamma(t)}),
$$
we are allowed to multiply by $-\Delta_{\Gamma}u\in L^2(\Gamma(t))$ for almost any $t\in[0,T]$. Recalling that $\varphi^\prime>0$, after an integration by parts, being $\Gamma(t)$ closed and $u\in L^\infty_{H^1}$, we obtain
\begin{align*}
\norm{\Delta_{\Gamma}u}^2&\leq \norm{\Delta_{\Gamma}u}^2+\frac{\theta}{2}m(\varphi^\prime(u),\vert\tgrad u\vert^2)\leq \norm{\tgrad  w}\norm{\tgrad  u}+\Vert\tgrad u\Vert^2\leq C(1+\norm{\tgrad  w}),
\end{align*}
and knowing that $w\in L^2_{H^1}$, we infer $u\in L^4_{H^2}$.
\end{remark}

\subsection{Regularisation and strict separation property}
In this subsection we need {\color{black}Assumptions $\textbf{B}_\Phi$.}
We first show that the approximating solution given by the Galerkin scheme is consequently more regular. This allows us to perform higher-order regularisation estimates and to establish the strict separation property.

\subsubsection{Galerkin approximation}
Motivated by the approach in \cite{CaeEll21}, we start by considering the Galerkin approximation with the approximated potential $F^\delta\in C^2(\R)$ {\color{black}defined by 
$$
F^\delta(r):=F_{ln}^\delta(r)+\frac{1-r^2}{2},\quad \forall r\in\R,
$$
where}
\begin{align*}
    F^\delta_{ln}(r) = \dfrac{\theta}{2}
    \begin{cases}
    (1-r)\ln(\delta)+(1+r)\ln(2-\delta) + \frac{(1-r)^2}{2\delta} + \frac{(1+r)^2}{2(2-\delta)} - 1, & \, r\geq 1-\delta \\
    (1+r)\ln(1+r) + (1-r)\ln(1-r), & \, |r|\leq 1-\delta \\
    (1+r)\ln(\delta)+(1-r)\ln(2-\delta) + \frac{(1+r)^2}{2\delta} + \frac{(1-r)^2}{2(2-\delta)} -1, & \, r\leq -1+\delta
    \end{cases}
\end{align*}
and denote $\varphi^\delta=(F_{ln}^\delta)^\prime$. Let us recall \cite[Sec.4.1]{CaeEll21} and consider a basis $\{\chi_j^0: j\in \N\}$ orthonormal in $L^2(\Gamma_0)$ and orthogonal in $H^1(\Gamma_0)$ consisting of smooth functions such that $\chi_1^0$ is constant (for example consider the eigenfunctions of the Laplace-Beltrami operator). We then transport this basis using the flow map. This gives $\{\chi_j^t:=\phi_t(\chi_j^0): j\in \N\}\subset H^1(\Gamma(t))$ and we define the finite-dimensional spaces 
\begin{align}
V_M(t) := \{\chi_j^t \colon j=1,\dots,M\}.
\label{VM}
\end{align}
The goal in this section is to find an approximating solution pair $(u^M, w^M)$ in these spaces; more precisely, for each $M\in \N$, we aim to find functions $u^M,w^M\in L^2_{V_M}$  with $\partial^\bullet u^M\in L^2_{V_M}$ such that, for any $\eta\in L^2_{V_M}$ and all $t\in[0,T]$, 
	\begin{align}
&m(\partial^\bullet u^M, \eta)+g({\color{black}u^M},\eta)+a_N(u^M,\eta)+a_S(w^M,\eta)=0,\label{CH2}\\&
a_S(u^M,\eta)+m((F^\delta)^\prime(u^M),\eta)-m(w^M,\eta)=0,\label{mu2}
\end{align}
and $u^M(0)=P_M^0u_0$ almost everywhere in $\Gamma_0$ ($P_M^0$ is the $L^2$ orthogonal projector operator at $t=0$ (see \cite[Sec.4.1]{CaeEll21}). 
We first prove a refinement of \cite[Prop.4.4]{CaeEll21}, namely,
\begin{prop}
	{\color{black}Let Assumptions $\textbf{B}_\Phi$ hold.} Then there exists a unique local solution pair to \eqref{CH2}-\eqref{mu2}. In particular there exist functions $(u^M,w^M)$ satisfying \eqref{CH2}-\eqref{mu2} on an interval $[0,t^\star)$, $0\leq t^\star\leq T$, together with the initial condition $u^M(0)=P_M^0u_0$. The functions are of the form
	$$
	u^M(t)=\sum_{i=1}^M u_i^M(t)\chi_i^t,\qquad w^M(t)=\sum_{i=1}^M w_i^M(t)\chi_i^t, \qquad t\in[0,t^\star),
	$$
	with $u_i^M\in C^2([0,t^\star))$ and $w_i^M\in C^2([0,t^\star))$, for every $i\in\{1,\ldots,M\}$.
	\label{proposition}
\end{prop}
\begin{proof}
	We consider the matrix form of the equations, where we denote $\mathbf{u}^M(t)=(u_1^M(t),\ldots,u_M^M(t))$ and $\mathbf{w}^M(t)=(w_1^M(t),\ldots,w_M^M(t))$,  \begin{align*}
	&M(t)\dot{\mathbf{u}}^M(t)+{G}(t)\mathbf{{u}}^M(t)+A_N(t)\mathbf{u}^M(t)+A_S(t)\mathbf{w}^M(t)=0,\\&
	A_S(t)\mathbf{u}^M(t)+(\mathbf{F}^\delta)^\prime(\mathbf{{u}}^M(t))-M(t)\mathbf{w}^M(t)=0.
	\end{align*}
	Here
    \begin{align*}
    &(M(t))_{ij}=m(t;\chi_i^t,\chi_j^t),\qquad ({G}(t))_{ij}=g(t;{\color{black}\chi_j^t},\chi_i^t),\\&
    (A_S(t))_{ij}=a_S(t;\chi^t_i,\chi^t_j),\qquad (A_N(t))_{ij}=a_N(t;\chi_j^t,\chi_i^t),
    \end{align*}
    and
    $$
    (\mathbf{F}^\delta)^\prime(\mathbf{u}^M(t))_j=m(t;(F^\delta)^\prime(u^M(t)),\chi_j^t).
    $$
    We now observe that actually these matrices enjoy more regularity than noted in \cite{CaeEll21}. Indeed, let us consider for example $M_{ij}$. By \cite[Prop. 2.8]{CaeEll21}, recalling that $\partial^\bullet \chi_i^t\equiv0$ for any $i=1,\ldots,M$, we have
    $$
    \frac{d}{dt}M_{ij}=g(\chi^t_i,\chi^t_j)\in C^0([0,T]),
    $$
    due to the regularity assumptions on $\Gamma(t)$, on the corresponding flow map and on the velocity field. We also have
    $$
    \frac{d}{dt}(M^{-1})_{ij}=-\left(M^{-1}\left(\frac{d}{dt}M\right)M^{-1}\right)_{ij}\in C^0([0,T]).
    $$
    Similarly, we get
        $$
    \frac{d}{dt}(A_S(t))_{ij}=b(\chi^t_i,\chi^t_j)\in C^0([0,T]),
    $$
    and, by \eqref{der},
      $$
    \frac{d}{dt}(A_N(t))_{ij}=\int_{\Gamma(t)}B_{adv}(\mathbf{V}_a^\tau(t),\mathbf{V}(t))\chi^t_j\cdot \tgrad \chi_i^t\in C^0([0,T]).
    $$
   {\color{black} Moreover, by Proposition \ref{prop:transport} point (a),
    \begin{align*}
    \frac{d}{dt}({G}(t))_{ij}=m(\chi_i^t&\partial^\bullet(\tgrad \cdot \mathbf{V}),\chi_j^t )+g\left(\chi_i^t\tgrad \cdot \mathbf{V},\chi_j^t\right)\in C^0([0,T]),
    \end{align*}
    thanks to \eqref{es2}.}
    Furthermore, we get
     $$
    \frac{d}{dt}(\mathbf{F}^\delta)^\prime(\mathbf{y})_j=g\left((F^\delta)^\prime\left(\sum_{i=1}^M y_i\chi_i^t\right),\chi_j^t\right)\in C^0([0,T]).
    $$
    Recalling then that $(F^\delta)^\prime$ is $C^{1,1}(\R)$, ensuring the spatial $C^{1,1}$-continuity of the $\mathbf{y}$-Jacobian of $\mathbf{y} \mapsto (\textbf{F}^\delta)^\prime(\mathbf{y})$, the result follows from the general ODE theory.
\end{proof}

We now obtain the energy estimate for the Galerkin approximation, without letting $M\rightarrow\infty$ (cf. \cite{CaeEll21}). This is essential in order to maintain the necessary regularity to perform higher-order estimates. 

\begin{prop}
\label{energ}
    Denoting by $(u_\delta^M, w_\delta^M)$ the solution pair to  \eqref{CH2}-\eqref{mu2}, we have the energy estimate 
    \begin{align}
\sup_{t\in [0,T]}\tilde{E}^{CH}[t;u_\delta^M(t)]+
%{\color{black}\sup_{t\in [0,T]}\Vert u_\delta^M(t)\Vert^2}+
\int_0^T\Vert\tgrad w_\delta^M\Vert^2dt \leq C_\delta(T),
\end{align}
{\color{black}for some $C_\delta(T)$ depending on $\delta$ and $T$, but not on $M$,} where $\tilde{E}^{CH}:=E^{CH}+\tilde C$ for some $\tilde C>0$ chosen so that $\tilde{E}^{CH}\geq 0$. {\color{black}Moreover, it holds 
\begin{align}
\sup_{t\in[0,T]}\tilde{E}^{CH}[u_\delta^M(t)]&+\frac{1}{4}\int_0^T\Vert\tgrad w_\delta^M\Vert^2\leq C(T)\left(1+\int_0^T\vert m(w_\delta^M,1)\vert^2 dt\right),
    \label{energylimit}
\end{align}
for some $C(T)>0$ independent of $M,\delta$ but possibly depending on $T$.}
\end{prop}

\begin{remark}
We observe that, even though the energy $E^{CH}$ is not necessarily non-negative, it is bounded from below. Indeed, there exists $C=C(\theta)$ such that $F \geq C(\theta)$, and therefore 
\begin{align}
    E^{CH}[u] \geq C(\theta) |\Gamma(t)| \geq \dfrac{C(\theta)}{C_\Gamma} |\Gamma_0|,
\end{align}
{\color{black}where $C_\Gamma$ is defined in \eqref{eq:surfacerelation}.} Note that $\tilde C$ (see the previous statement) depends on $T$, $\Gamma_0$ and $\theta$, as well as on an upper bound for the $C^1$-norm of the flow map $\Phi$. 
\end{remark}

\begin{proof}
		Keep $M,\delta$ fixed and denote by $C$ a generic positive constant independent of $M$ and $\delta$. We consider the Galerkin approximation and the results given by Proposition \ref{proposition}. Here we denote the solution by $(u_\delta^M,w_\delta^M)$ to emphasize the dependence on $\delta$.
We then note that the total mass of $u_\delta^M$ is preserved in time, since $\eta\equiv1\in V_M$ for any $M\in\N$.
We then come back to the energy estimate as in the proof of \cite[Prop.5.7, a)]{CaeEll21}:
\begin{align}
\frac{d}{dt}\tilde{E}^{CH}[u_\delta^M]+\frac{1}{2}\Vert\tgrad w_\delta^M\Vert^2\leq -g({\color{black}u_\delta^M},w_\delta^M)+C_0+C_1\tilde{E}^{CH}[u_\delta^M],
\label{enel}
\end{align}
where $\tilde{E}_{CH}:={E}_{CH}+\tilde{C}$, with $\tilde{C}>0$ a suitable constant so that $\tilde{E}_{CH}\geq 0$.
Thanks to the bound on $\mathbf V$, by Young and Poincaré's inequalities we have
\begin{align*}
-g({\color{black}u_\delta^M},w_\delta^M)&\leq C\Vert {\color{black}u_\delta^M}\Vert\Vert w_\delta^M\Vert\leq C\Vert {\color{black}u_\delta^M}\Vert(\vert m(w_\delta^M,1)\vert+\Vert\tgrad w_\delta^M\Vert)\\&\leq \frac 1 4\Vert\tgrad w_\delta^M\Vert^2+C(\Vert {\color{black}u_\delta^M}\Vert^2+\vert m(w_\delta^M,1)\vert^2)\\&\leq \frac 1 4\Vert\tgrad w_\delta^M\Vert^2+C(1+\Vert {\color{black}\tgrad u_\delta^M}\Vert^2+\vert m(w_\delta^M,1)\vert^2),
\end{align*}
{\color{black}where we exploited the fact that, by mass conservation, Poincaré's inequality and \eqref{eq:surfacerelation},  
\begin{align}
\Vert u_\delta^M\Vert&\leq \frac {1}{\vert \Gamma(t)\vert}\vert  m(u_\delta^M,1)\vert+\Vert \tgrad u_\delta^M\Vert\leq C(\vert  m(P_M^0u_0,1)\vert+\Vert \tgrad u_\delta^M\Vert)\leq C(1+\Vert \tgrad u_\delta^M\Vert),
\label{ud}
\end{align}
recalling that $\Vert P_M^0u_0\Vert\leq \Vert u_0\Vert\leq C$.
For further use, this gives
\begin{align}
\frac{d}{dt}\tilde{E}^{CH}[u_\delta^M]+\frac{1}{4}\Vert\tgrad w_\delta^M\Vert^2\leq C_0+C_1\tilde{E}^{CH}[u_\delta^M]+C_2\vert m(w_\delta^M,1)\vert^2,
    \label{energylimit2}
\end{align}
with $C_0,C_1,C_2>0$ independent of $M,\delta$.
Since estimate \eqref{e} cannot be performed in the Galerkin context (note that, differently from the static domain case, here it is not ensured that $\Delta_{\Gamma}u_\delta^M\in L^2_{V_M}$), we now first obtain an estimate depending on $\delta$ but not on $M$, and in a second time, after passing to limit as $M\to\infty$, we obtain estimates uniform in $\delta$, thanks to \eqref{e}. Therefore, we now exploit $\vert F^\prime_\delta(u_\delta^M)\vert\leq C_\delta\vert u_\delta^M\vert$, being ${\color{black}\vert F^{\prime\prime}_\delta\vert}\leq C_\delta$ and $F_\delta^\prime(0)=F^\prime(0)=0$, to deduce
\begin{align}
\vert m(w_\delta^M,1)\vert\leq \vert m(F_\delta^\prime(u_\delta^M),1)\vert \leq C_\delta\Vert u_\delta^M\Vert_{L^1(\Gamma(t))}\leq C_\delta\Vert u_\delta^M\Vert
\label{meanw}
\end{align}
with $C_\delta>0$ depending on $\delta$, so that,} by Young and Poincaré's inequalities, we find 
$$
-g({\color{black}u_\delta^M},w_\delta^M)\leq C\Vert u_\delta^M\Vert\Vert w_\delta^M\Vert\leq C_\delta\left(1+\Vert u_\delta^M\Vert^2\right)+\frac{3}{8}\Vert\tgrad w_\delta^M\Vert^2,
$$
and thus, from \eqref{enel} and \eqref{ud},
\begin{align*}
\frac{d}{dt}\tilde{E}^{CH}[u_\delta^M]+\frac{1}{8}\Vert\tgrad w_\delta^M\Vert^2\leq {\color{black}C_\delta}\tilde{E}^{CH}[u_\delta^M]+C_\delta.
\end{align*}
Therefore Gronwall's Lemma, together with \eqref{ud} and \cite[Lemma 5.6]{CaeEll21} give
\begin{align}
\sup_{t\in [0,T]}\tilde{E}^{CH}[t;u_\delta^M(t)]+\sup_{t\in [0,T]}\Vert u_\delta^M(t)\Vert^2+\int_0^T\Vert\tgrad w_\delta^M\Vert^2dt \leq C_\delta,
\label{energy1}
\end{align}
independently of $M$. {\color{black}Moreover, by Gronwall's Lemma and \cite[Lemma 5.6]{CaeEll21}, applied to \eqref{energylimit2}, we also get \eqref{energylimit}.}
\end{proof}

\subsubsection{Regularisation and strict separation property}
\label{sep}
The main result of this section is the following{\color{black}, letting Assumptions $\textbf{B}_\Phi$ hold.}
\begin{theorem}
	\label{main1}
{\color{black}	Let the assumptions of Theorem \ref{existence} and Assumptions $\textbf{B}_\Phi$ hold.} Denote by $(u,w)$ the (unique) weak solution to \eqref{CH}-\eqref{mu} with $u(0)=u_0$.
	\begin{itemize}
	    \item[(i)] There exists a constant $C=C(T,{E}^{CH}(u_0))>0$ such that, for almost any $t\in[0,T]$,
	\begin{align}
	t\Vert w\Vert_{H^1(\Gamma(t))}^2+\int_0^t s\Vert\partial^\bullet u\Vert^2_{H^1(\Gamma(s))}ds\leq C(T,{E}^{CH}(u_0)).
	\label{mu5}
	\end{align}
	\item[(ii)] For any $0<\tau\leq T$, there exist constants $C=C(T,\tau,{E}^{CH}(u_0))>0$ and $C_p=C_p(T,\tau,p, {E}^{CH}(u_0))>0$ such that, for almost any $t\in[\tau,T]$,
	\begin{align}
	\Vert w\Vert_{H^1(\Gamma(t))}\leq C(T,\tau, {E}^{CH}(u_0)),
	\label{mu4}
	\end{align}
	$$
	\Vert\varphi(u)\Vert_{L^p(\Gamma(t))}+\Vert\varphi^\prime(u)\Vert_{L^p(\Gamma(t))}\leq C_p(T,\tau,p,{E}^{CH}(u_0)),\quad\forall p\in[2,\infty),
	$$
	$$
	\Vert u\Vert_{H^2(\Gamma(t))}\leq C(T,\tau,{E}^{CH}(u_0)).
	$$
	\item[(iii)] There exists $\xi=\xi(T,\tau, {E}^{CH}(u_0))>0$ such that
	\begin{align*}
	\Vert u \Vert_{L^\infty(\Gamma(t))}\leq 1-{\xi}, \quad \text{ for a.a. }t\in[\tau,T].
	\end{align*}
	\end{itemize}
\end{theorem}
\begin{remark}
	\label{sepp}
	Notice that actually we could say more about the separation property, if we assume the regularity on the flow map and the velocity field as in Lemma \ref{C1}. Indeed, for any $\tau>0$, we have $u\in L^\infty_{H^2}(\tau,T)$ and $\partial^\bullet u\in L^2_{H^1}(\tau,T)$ (we use the symbol $L^q_{X}(\tau,T)$ to mean that the set of times is $[\tau,T]$ instead of $[0,T]$, typical of $L^q_{X}$). By the embedding result shown in Appendix \ref{app:embedding}, we infer that $u\in C^0_{H^{3/2}}(\tau,T)$, thus by the embedding $H^{3/2}(\Gamma(t))\hookrightarrow C^0(\Gamma(t))$, $u\in C^0_{C^0}(\tau,T)$, implying 
		\begin{align*}
	\sup_{t\in[\tau,T]}\Vert u \Vert_{C^0(\Gamma(t))}\leq 1-{\xi}.
	\end{align*}
\end{remark}
\begin{proof}
From now on, we will omit the dependence on ${E}^{CH}(u_0)$, since it is considered understood due to Prop. \ref{energ}. 

\underline{Part (i).} {\color{black}\textbf{Step 1. Limit as $M\to\infty$.}} We test \eqref{CH2} with $\eta=\partial^\bullet w^M_\delta\in L^2_{V_M}$, obtaining
	\begin{align}
	m(\partial^\bullet u^M_\delta, \partial^\bullet w^M_\delta)+g({\color{black}u^M_\delta},\partial^\bullet w^M_\delta)+a_N(u^M_\delta,\partial^\bullet w^M_\delta)+a_S(w^M_\delta,\partial^\bullet w^M_\delta)=0.
	\label{base}
	\end{align}
	Observe now that, by Proposition \ref{prop:transport}, for $\eta\in L^2_{V_M}$ such that $\partial^\bullet\eta\in L^2_{V_M}$, we have
	\begin{align}
\label{m1}
	\frac{d}{dt}m(w^M_\delta,\eta)=m(\partial^\bullet w^M_\delta,\eta)+m(w^M_\delta,\partial^\bullet\eta)+g(w^M_\delta,\eta),
	\end{align}
	but also (see \eqref{mu2})
	\begin{align}
\label{m2}
	\frac{d}{dt}m(w^M_\delta,\eta)=\frac{d}{dt}\left(a_S(u^M_\delta,\eta)+m((F^\delta)^\prime(u^M_\delta),\eta)\right)
	\end{align}
	Moreover, again by Proposition \ref{prop:transport}, we infer
	$$
	\frac{d}{dt}a_S(u^M_\delta,\eta)=a_S(\partial^\bullet u^M_\delta,\eta)+a_S(u^M_\delta,\partial^\bullet\eta)+b(u^M_\delta,\eta).
	$$
	On the other hand, exploiting the chain rule, we get
	\begin{align*}
	\frac{d}{dt}m((F^\delta)^\prime(u^M_\delta),\eta)=m((F^\delta)^{\prime\prime}(u^M_\delta)\partial^\bullet u^M_\delta,\eta)&+m((F^\delta)^\prime(u^M_\delta),\partial^\bullet\eta)+g((F^\delta)^\prime(u^M_\delta),\eta).
	\end{align*}
	Therefore, comparing \eqref{m1} with \eqref{m2}, we obtain
	\begin{align*}
	m(&\partial^\bullet w^M_\delta,\eta)+m(w^M_\delta,\partial^\bullet\eta)+g(w^M_\delta,\eta) 
	=
	m((F^\delta)^{\prime\prime}(u^M_\delta)\partial^\bullet u^M_\delta,\eta)\\
	&+m((F^\delta)^\prime(u^M_\delta),\partial^\bullet\eta)+g((F^\delta)^\prime(u^M_\delta),\eta)+a_S(\partial^\bullet u^M_\delta,\eta)+a_S(u^M_\delta,\partial^\bullet\eta)+b(u^M_\delta,\eta),
	\end{align*}
	but since $\partial^\bullet\eta$ is still an admissible function in \eqref{mu2}, we infer
	\begin{align*}
	m(\partial^\bullet w^M_\delta,\eta)&=-g(w^M_\delta,\eta)+m((F^\delta)^{\prime\prime}(u^M_\delta)\partial^\bullet u^M_\delta,\eta)+g((F^\delta)^{\prime}(u^M_\delta),\eta)+a_S(\partial^\bullet u^M_\delta,\eta)+b(u^M_\delta,\eta).
	\end{align*}
	Thus, setting $\eta=\partial^\bullet u_\delta^M\in L^2_{V_M}$ (note that $\partial^{\bullet,2} u_\delta^M\in L^2_{V_M}$ by Proposition \ref{proposition}), we get
		\begin{align}
	m(\partial^\bullet w^M_\delta,\partial^\bullet u^M_\delta)&=-g(w^M_\delta,\partial^\bullet u^M_\delta)+m((F^\delta)^{\prime\prime}(u^M_\delta)\partial^\bullet u^M_\delta,\partial^\bullet u^M_\delta)\nonumber\\&+g((F^\delta)^{\prime}(u^M_\delta),\partial^\bullet u^M_\delta)+a_S(\partial^\bullet u^M_\delta,\partial^\bullet u^M_\delta)+b(u^M_\delta,\partial^\bullet u^M_\delta).
	\label{e1}
	\end{align}
Using once more Proposition \ref{prop:transport}, on account of the regularity given by Proposition \ref{proposition}, we obtain
	\begin{align}
	\frac{1}{2}\frac{d}{dt}\Vert\tgrad w^M_\delta\Vert^2=a_S(w^M_\delta,\partial^\bullet w^M_\delta)+\frac{1}{2}b(w^M_\delta,w^M_\delta).
	\label{aS}
	\end{align}
	Furthermore, consider the term $g$ in \eqref{base}. By Proposition \ref{prop:transport}, we have
	\begin{align*}
	\frac{d}{dt}g({\color{black}u^M_\delta},w^M_\delta)&=g({\color{black}u^M_\delta},\partial^\bullet w^M_\delta)+m(\partial^\bullet({\color{black}u^M_\delta} \tgrad \cdot \mathbf{V}), w^M_\delta) +g({\color{black}u^M_\delta}\tgrad \cdot \mathbf{V}, w^M_\delta)\\
	&=g({\color{black}u^M_\delta},\partial^\bullet w^M_\delta)+m(\partial^\bullet u^M_\delta \tgrad \cdot \mathbf{V}, w^M_\delta)+m({\color{black}u^M_\delta} \partial^\bullet\tgrad \cdot \mathbf{V}, w^M_\delta)+g({\color{black}u^M_\delta}\tgrad \cdot \mathbf{V}, w^M_\delta).
	\end{align*}
	Therefore we get
	\begin{align}
	g({\color{black}u^M_\delta},\partial^\bullet w^M_\delta)=\frac{d}{dt}&g({\color{black}u^M_\delta},w^M_\delta)-m(\partial^\bullet u_\delta^M \tgrad \cdot \mathbf{V}, w^M_\delta)-m( {\color{black}u^M_\delta} \partial^\bullet\tgrad \cdot \mathbf{V}, w^M_\delta)-g({\color{black}u^M_\delta}\tgrad \cdot \mathbf{V}, w^M_\delta).
	\label{g}
	\end{align}
	Recalling \eqref{der}, we have %{\color{purple}(check if want to keep $V_a^\tau$) }
	\begin{align*}
\frac{d}{dt}a_N(u^M_\delta,w^M_\delta) =a_N(\partial^\bullet u^M_\delta,&w^M_\delta)+a_N(u^M_\delta,\partial^\bullet w^M_\delta)+\int_{\Gamma(t)}B_{adv}(\mathbf{V}_a^\tau(t),\mathbf{V}(t))u^M_\delta\cdot \tgrad w^M_\delta,
	\end{align*}
	that is,
	\begin{align}
	a_N(u^M_\delta,\partial^\bullet w^M_\delta) = \frac{d}{dt}a_N(u^M_\delta,&w^M_\delta)-a_N(\partial^\bullet u^M_\delta,w^M_\delta)-\int_{\Gamma(t)}B_{adv}(\mathbf{V}_a^\tau(t),\mathbf{V}(t))u^M_\delta\cdot \tgrad w^M_\delta.\label{aN}
	\end{align}
	Thanks to Proposition \ref{prop:transport}, we deduce
	\begin{align}
	\frac{d}{dt} \frac{\theta}{2}g(F_{ln}^\delta(u^M_\delta),1)=\frac{\theta}{2}g(\varphi_\delta(u^M_\delta),\partial^\bullet u^M_\delta)+\frac{\theta}{2}\int_{\Gamma(t)}&F_{ln}^\delta(u^M_\delta)\partial^\bullet \tgrad \cdot \mathbf{V}+\frac{\theta}{2}g(F_{ln}^\delta(u^M_\delta), \tgrad \cdot \mathbf{V}).
	\label{Fdelta}
	\end{align}
	This last equality is necessary, since in the Galerkin setting we are not able to retrieve a uniform estimate for $\Vert\varphi_\delta(u_\delta^M)\Vert_{L^2_{L^2}}$ but only for its $L^2$-orthogonal projection over $V_M$(t), whereas, thanks to \eqref{energy1}, we are able to uniformly estimate $\sup_{t\in[0,T]}\Vert F^\delta_{ln}(u_\delta^M)\Vert_{L^1(\Gamma(t))}$.
	Collecting \eqref{base}-\eqref{aN}, we obtain
	\begin{align*} \frac{d}{dt}\Bigg(\frac{1}{2}&\Vert\tgrad w^M_\delta\Vert^2+a_N(u^M_\delta,w^M_\delta)+g({\color{black}u_\delta^M},w^M_\delta)\Bigg)
+\Vert\tgrad \partial^\bullet u^M_\delta\Vert^2\nonumber\\
&\hskip 6mm +m((F^\delta)^{\prime\prime}(u^M_\delta)\partial^\bullet u^M_\delta,\partial^\bullet u^M_\delta)\nonumber\\
&=g(w^M_\delta,\partial^\bullet u^M_\delta)-g((F^\delta)^{\prime}(u^M_\delta),\partial^\bullet u^M_\delta)-b(u^M_\delta,\partial^\bullet u^M_\delta)+\frac{1}{2}b(w^M_\delta,w^M_\delta) \nonumber\\
&\hskip 6mm +m(\partial^\bullet u^M_\delta \tgrad \cdot \mathbf{V}, w^M_\delta)+m( {\color{black}u_\delta^M} \partial^\bullet\tgrad \cdot \mathbf{V}, w^M_\delta)\nonumber\\
&\hskip 6mm  +g({\color{black}u_\delta^M}\tgrad \cdot \mathbf{V}, w^M_\delta)+a_N(\partial^\bullet u^M_\delta,w^M_\delta)\nonumber\\
&\hskip 6mm +\int_{\Gamma(t)}B_{adv}(\mathbf{V}_a^\tau(t),\mathbf{V}(t))u^M_\delta\cdot \tgrad w^M_\delta.
	%\label{eq0}
	\end{align*}
	In particular, substituting the explicit value of $F^\delta$, setting $\varphi_\delta:=(F^\delta_{ln})^\prime$, and exploiting \eqref{Fdelta}, we find
	 \begin{align*}
	&\frac{d}{dt}\left(\frac{1}{2}\Vert\tgrad w^M_\delta\Vert^2+a_N(u^M_\delta,w^M_\delta)+g(u^M_\delta,w^M_\delta)+\frac{\theta}{2}g(F_{ln}^\delta(u^M_\delta),1)\right)\\&+\Vert\tgrad \partial^\bullet u^M_\delta\Vert^2+\frac{\theta}{2}m(\varphi_\delta^\prime(u^M_\delta)\partial^\bullet u^M_\delta,\partial^\bullet u^M_\delta)\\&=g(w^M_\delta,\partial^\bullet u^M_\delta)+m(\partial^\bullet u^M_\delta,\partial^\bullet u^M_\delta)+g(u^M_\delta,\partial^\bullet u^M_\delta)-b(u^M_\delta,\partial^\bullet u^M_\delta)\\&+\frac{1}{2}b(w^M_\delta,w^M_\delta)+m(\partial^\bullet u^M_\delta \tgrad \cdot \mathbf{V}, w^M_\delta)+m({\color{black}u_\delta^M} \partial^\bullet\tgrad \cdot \mathbf{V}, w^M_\delta)\\&+g({\color{black}u_\delta^M}\tgrad \cdot \mathbf{V}, w^M_\delta)+a_N(\partial^\bullet u^M_\delta,w^M_\delta)+\int_{\Gamma(t)}B_{adv}(\mathbf{V}_a^\tau(t),\mathbf{V}(t))u^M_\delta\cdot \tgrad w^M_\delta\\&+\frac{\theta}{2}\int_{\Gamma(t)}F_{ln}^\delta(u^M_\delta)\partial^\bullet \tgrad \cdot \mathbf{V}+\frac{\theta}{2}g(F_{ln}^\delta(u^M_\delta), \tgrad \cdot \mathbf{V}).
	\end{align*}
	Recalling that $\varphi_\delta^\prime=(F^\delta_{ln})''\geq 0$, we have $\frac{\theta}{2}m(\varphi_\delta^\prime(u^M_\delta)\partial^\bullet u^M_\delta,\partial^\bullet u^M_\delta)\geq 0$.
	%Then, by the regularity of the velocity fields, the interpolation estimate $\Vert \partial^\bullet u^M_\delta\Vert	\leq C\Vert \partial^\bullet u^M_\delta\Vert_{-1}^{1/2}\Vert\tgrad \partial^\bullet u^M_\delta\Vert^{1/2}$, H\"older and Young's inequalities,
	%\begin{align*}
	%&m(\partial^\bullet u^M_\delta,\partial^\bullet u^M_\delta)\leq \frac{1}{6}\Vert\tgrad \partial^\bullet u^M_\delta\Vert^2+C\Vert\partial^\bullet u^M_\delta\Vert_{-1}^2.
	%\end{align*}
	Therefore, setting
\begin{align}
\mathcal{Q}^M_\delta:=\frac{1}{2}\Vert\tgrad w^M_\delta\Vert^2+a_N(u^M_\delta,w^M_\delta)+g({\color{black}u_\delta^M},w^M_\delta)
+\frac{\theta}{2}g(F_{ln}^\delta(u^M_\delta),1),
\label{ees}
\end{align}
and multiplying by $t\in (0,T]$ the above inequality, we get
			 \begin{align*}
	&{\color{black}\frac{d}{dt}\left(t\mathcal{Q}^M_\delta(t)\right)}+t\Vert\tgrad \partial^\bullet u^M_\delta\Vert^2+t\frac{\theta}{2}m(\varphi_\delta^\prime(u^M_\delta)\partial^\bullet u^M_\delta,\partial^\bullet u^M_\delta)\\&=\mathcal{Q}^M_\delta(t)+g(w^M_\delta,t\partial^\bullet u^M_\delta)+tm(\partial^\bullet u^M_\delta,\partial^\bullet u^M_\delta)\\&+g(u^M_\delta,t\partial^\bullet u^M_\delta)-b(u^M_\delta,t\partial^\bullet u^M_\delta)+\frac{t}{2}b(w^M_\delta,w^M_\delta)\\&+m(\partial^\bullet u^M_\delta \tgrad \cdot \mathbf{V}, tw^M_\delta)+m( {\color{black}u_\delta^M} \partial^\bullet\tgrad \cdot \mathbf{V}, tw^M_\delta)\\&+g({\color{black}u_\delta^M}\tgrad \cdot \mathbf{V}, tw^M_\delta)+a_N(\partial^\bullet u^M_\delta,tw^M_\delta)+t\int_{\Gamma(t)}B_{adv}(\mathbf{V}_a^\tau(t),\mathbf{V}(t))u^M_\delta\cdot \tgrad w^M_\delta\\&+\frac{\theta t}{2}\int_{\Gamma(t)}F_{ln}^\delta(u^M_\delta)\partial^\bullet \tgrad \cdot \mathbf{V}+\frac{\theta t}{2}g(F_{ln}^\delta(u^M_\delta), \tgrad \cdot \mathbf{V}).
	\end{align*}
	{\color{black}We now recall that, by Poincaré's inequality,
		\begin{align}
	\label{w_bis}
	\Vert w^M_\delta\Vert\leq C(\Vert \tgrad  w^M_\delta\Vert+\vert m(w_\delta^M,1)\vert).
	\end{align}}
	Then, on account of \eqref{V} and using H\"older's and Young's inequalities, we deduce
	\begin{align*}
	&g(w^M_\delta,t\partial^\bullet u^M_\delta)+m(\partial^\bullet u^M_\delta,t\partial^\bullet u^M_\delta)\\&+g(u^M_\delta,t\partial^\bullet u^M_\delta)-b(u^M_\delta,t\partial^\bullet u^M_\delta)+\frac{t}{2}b(w^M_\delta,w^M_\delta)\\&{\color{black}+}\frac{\theta t}{2}\int_{\Gamma(t)}F_{ln}^\delta(u^M_\delta)\partial^\bullet \tgrad \cdot \mathbf{V}+\frac{\theta t}{2}g(F_{ln}^\delta(u^M_\delta), \tgrad \cdot \mathbf{V})\\&\leq Ct\Vert w^M_\delta\Vert\Vert\partial^\bullet u^M_\delta\Vert
	+t\Vert\partial^\bullet u^M_\delta\Vert^2\\&+t{\color{black}\Vert u^M_\delta\Vert}\Vert \partial^\bullet u^M_\delta\Vert+Ct\Vert\tgrad  u_\delta^M\Vert\Vert\tgrad \partial^\bullet u^M_\delta\Vert+Ct\Vert\tgrad w^M_\delta\Vert^2\\&+Ct\Vert F_{ln}^\delta(u_\delta^M)\Vert_{L^1(\Gamma(t))}+Ct\Vert F_{ln}^\delta(u_\delta^M)\Vert_{L^1(\Gamma(t))}\\&\leq Ct\Vert \partial^\bullet u^M_\delta\Vert^2+Ct\Vert\tgrad w^M_\delta\Vert^2+\frac{t}{2}\Vert\tgrad \partial^\bullet u^M_\delta\Vert^2\\&+Ct(\Vert u^M_\delta\Vert^2+\vert m(w_\delta^M,1)\vert^2+\Vert F_{ln}^\delta(u_\delta^M)\Vert_{L^1(\Gamma(t))}),
	\end{align*}
	where we exploited \eqref{es2} and \eqref{w_bis}.
Analogously, recalling \eqref{es1}, \eqref{es2} and \eqref{w_bis}, thanks to Cauchy-Schwarz and Young's inequalities, we find
	\begin{align*}
	&m(\partial^\bullet u^M_\delta \tgrad \cdot \mathbf{V}, tw^M_\delta)+m( {\color{black}u_\delta^M} \partial^\bullet\tgrad \cdot \mathbf{V}, tw^M_\delta)\\&+g({\color{black}u_\delta^M}\tgrad \cdot \mathbf{V}, tw^M_\delta)+a_N(\partial^\bullet u^M_\delta,tw^M_\delta)+t\int_{\Gamma(t)}B_{adv}(\mathbf{V}_a^\tau(t),\mathbf{V}(t))u^M_\delta\cdot \tgrad w^M_\delta\\&\leq Ct\Vert\partial^\bullet u^M_\delta\Vert({\color{black}\vert m(w_\delta^M,1)\vert}+\Vert \tgrad  w^M_\delta\Vert)+Ct{\color{black}\Vert u^M_\delta\Vert}({\color{black}\vert m(w_\delta^M,1)\vert}+\Vert \tgrad  w^M_\delta\Vert)\\&+Ct{\color{black}\Vert u^M_\delta\Vert}({\color{black}\vert m(w_\delta^M,1)\vert}+\Vert \tgrad  w^M_\delta\Vert)+Ct\Vert\partial^\bullet u^M_\delta\Vert\Vert\tgrad  w^M_\delta\Vert+Ct\Vert u^M_\delta\Vert\Vert \tgrad  w^M_\delta\Vert\\&\leq Ct\Vert\tgrad w^M_\delta\Vert^2+Ct\Vert\partial^\bullet u^M_\delta\Vert^2+Ct(1+{\color{black}\vert m(w_\delta^M,1)\vert^2}+\Vert u^M_\delta\Vert^2)
	\end{align*}
	Collecting the above inequalities, we eventually obtain
	\begin{align}
    \nonumber
	{\color{black}\frac{d}{dt}\left(t\mathcal{Q}^M_\delta(t)\right)}+\frac{t}{2}\Vert\tgrad \partial^\bullet u^M_\delta\Vert^2&\leq \mathcal{Q}^M_\delta(t)+Ct\Vert\tgrad w^M_\delta\Vert^2+Kt\Vert\partial^\bullet u^M_\delta\Vert^2\\&{\color{black}+Ct(\Vert u^M_\delta\Vert^2+\vert m(w_\delta^M,1)\vert^2+\Vert F_{ln}^\delta(u_\delta^M)\Vert_{L^1(\Gamma(t))})},
	\label{passage}
	\end{align}
	where we explicitly denoted by $K>0$, for further use, the constant in front of $\Vert\partial^\bullet u^M_\delta\Vert$, which has to be controlled in a uniform way. Since in the Galerkin setting we are not able to find a uniform estimate for $\Vert \partial^\bullet u^M_\delta\Vert_{L^2_{H^{-1}}}$, we cannot exploit the classical interpolation inequality  $\Vert \partial^\bullet u_\delta^M\Vert	\leq C\Vert \partial^\bullet u^M_\delta\Vert_{H^{-1}(\Gamma(t))}^{1/2}\Vert\partial^\bullet u^M_\delta\Vert_{H^1(\Gamma(t))}^{1/2}$ as in the case of nonevolving surfaces.
	Therefore, we test \eqref{CH2} with $\eta=t\partial^\bullet u_\delta^M\in L^2_{V_M}$, obtaining
	$$
	t\Vert \partial^\bullet u^M_\delta\Vert^2+tg({\color{black}u^M_\delta},\partial^\bullet u^M_\delta)+ta_N(u^M_\delta,\partial^\bullet u^M_\delta)+ta_S(w^M_\delta,\partial^\bullet u^M_\delta)=0,
	$$
	from which, by standard estimates, for some $\kappa$ sufficiently small to be chosen later on, we obtain
	$$t\Vert \partial^\bullet u^M_\delta\Vert^2\leq \frac{t}{2}\Vert \partial^\bullet u^M_\delta\Vert^2+t\kappa\Vert\tgrad  \partial^\bullet u^M_\delta\Vert^2+Ct\Vert\tgrad w_\delta^M\Vert^2+Ct{\color{black}\Vert u^M_\delta\Vert^2}, $$	
that is,
$$
\frac{t}{2}\Vert \partial^\bullet u^M_\delta\Vert^2\leq t\kappa\Vert\tgrad  \partial^\bullet u^M_\delta\Vert^2+Ct\Vert\tgrad w_\delta^M\Vert^2+Ct{\color{black}\Vert u^M_\delta\Vert^2}.
$$
We can now add this inequality multiplied by $\omega=4K$  to \eqref{passage}, choosing $\kappa=\frac{1}{16K}$. {\color{black}This gives, for all $t\in [0,T]$,
	\begin{align}
	{\frac{d}{dt}\left(t\mathcal{Q}^M_\delta(t)\right)}+\frac{t}{4}\Vert\tgrad \partial^\bullet u^M_\delta\Vert^2+{K t}\Vert \partial^\bullet u^M_\delta\Vert^2&\leq \mathcal{Q}^M_\delta(t)+Ct\Vert\tgrad w^M_\delta\Vert^2\nonumber+Ct(\Vert u^M_\delta\Vert^2\\&+\vert m(w_\delta^M,1)\vert^2+\Vert F_{ln}^\delta(u_\delta^M)\Vert_{L^1(\Gamma(t))}).
	\label{passage2}
	\end{align}
Recalling Proposition \ref{energ} and \eqref{meanw}, from \eqref{passage2} we get, for all $t\in [0,T]$
	\begin{align}
	{\frac{d}{dt}\left(t\mathcal{Q}^M_\delta(t)\right)}+\frac{t}{4}\Vert\tgrad \partial^\bullet u^M_\delta\Vert^2+{K t}\Vert \partial^\bullet u^M_\delta\Vert^2&\leq \mathcal{Q}^M_\delta(t)+Ct\Vert\tgrad w^M_\delta\Vert^2+C_\delta t,
	\label{pp1}
	\end{align}
 where $C_\delta>0$ depends on $\delta$.
}
Observe now that %being by Proposition \ref{energ}, $\Vert u^M_\delta\Vert_{L^2(\Gamma(t))}\leq C_\delta$,
	\begin{align}
	    \label{e1a}
	\vert a_N(u^M_\delta,w^M_\delta)\vert\leq C\Vert u_\delta^M\Vert \Vert\tgrad  w^M_\delta\Vert\leq C\Vert u_\delta^M\Vert^2+\frac{1}{8}\Vert\tgrad w^M_\delta\Vert^2,
	\end{align}
	and, by Poincaré's inequality,
	\begin{align}
	\vert g(u^M_\delta,w^M_\delta)\vert\leq C\Vert u_\delta^M\Vert \Vert w_\delta^M\Vert\leq C(\Vert u_\delta^M\Vert^2+\vert m(w_\delta^M,1)\vert^2)+\frac{1}{8}\Vert\tgrad w^M_\delta\Vert^2,
	\label{e2a}
	\end{align}
	so that (see Proposition \ref{energ}, \eqref{meanw} and \eqref{w_bis})
	$$
	\vert g(u^M_\delta,w^M_\delta)\vert\leq C_\delta+\frac{1}{8}\Vert\tgrad w^M_\delta\Vert^2,
	$$
	Moreover, owing to \eqref{energy1}, we have
	$$
	\frac{\theta}{2}\vert g(F_{ln}^\delta(u^M_\delta),1)\vert\leq C\Vert F_{ln}^\delta(u^M_\delta)\Vert_{L^1(\Gamma(t))}\leq C_\delta.
	$$
	implying that there exists $\widehat{C}{\color{black}_\delta}>0$, depending on $\delta$, such that
	$$
	\frac{1}{4}\Vert \tgrad  w^M_\delta\Vert^2-\widehat{C}{\color{black}_\delta}\leq   \mathcal{Q}^M_\delta(t)\leq C{\color{black}_\delta}(1+\Vert \tgrad  w^M_\delta\Vert^2),\quad\forall\,t\in [0,T].
	$$
	Note now that, due to the energy estimate \eqref{energy1}, we have that $\mathcal{Q}^M_\delta$ is uniformly bounded in $L^1(0,T)$.
	Thus, defining $\widetilde{\mathcal{Q}}^M_\delta:=\mathcal{Q}_\delta^M+\widehat{C}{\color{black}_\delta}\geq 0$, we obtain from \eqref{pp1}
	\begin{align}
{\color{black}\frac{d}{dt}\left(t\widetilde{\mathcal{Q}}^M_\delta(t)\right)}+{K t}\Vert \partial^\bullet u^M_\delta\Vert^2+\frac{t}{4}\Vert\tgrad \partial^\bullet u^M_\delta\Vert^2 \leq \widetilde{\mathcal{Q}}^M_\delta(t)+C{\color{black}_\delta}t\widetilde{\mathcal{Q}}^M_\delta(t)+{C}{\color{black}_\delta}t.
\label{2}
\end{align}
Therefore, by Gronwall's Lemma applied to $y(t)=t\widetilde{\mathcal{Q}}^M_\delta(t)$, recalling that $\widetilde{\mathcal{Q}}^M_\delta$ is uniformly {\color{black}(in $M$)} bounded in $L^1(0,T)$, we infer
	$$
	t\widetilde{\mathcal{Q}}^M_\delta(t)\leq C{\color{black}_\delta}(T),\quad \forall\,t\in[0,T],
	$$
	entailing, recalling \eqref{w_bis},
	\begin{align}
	t\Vert w^M_\delta\Vert_{H^1(\Gamma(t))}^2+\int_0^t s\Vert\partial^\bullet u^M_\delta\Vert^2_{H^1(\Gamma(s))}ds\leq C{\color{black}_\delta}(T), \quad
    \forall\,t\in[0,T],
    \label{pp2}
	\end{align}
 {\color{black}where $C_\delta(T)>0$ depends on $\delta$ but is independent of $M$.}
{\color{black}
	For further use we also integrate in time \eqref{passage2}, to get,
	for all $t\in [0,T]$,
	\begin{align}
	&\nonumber{t\mathcal{Q}^M_\delta(t)}+\int_0^t\frac{s}{4}\Vert\tgrad \partial^\bullet u^M_\delta\Vert^2ds+ \int_0^t{K s}\Vert \partial^\bullet u^M_\delta\Vert^2ds\\&\leq \int_0^t\mathcal{Q}^M_\delta(s)ds+C\int_0^t s\Vert\tgrad w^M_\delta\Vert^2ds\nonumber\\&+C\int_0^ts(\Vert u^M_\delta\Vert^2+\vert m(w_\delta^M,1)\vert^2+\Vert F_{ln}^\delta(u_\delta^M)\Vert_{L^1(\Gamma(t))})ds.
	\label{passage3}
	\end{align}
Then by the definition \eqref{ees} of $\mathcal{Q}_\delta^M$, \eqref{e1a} and \eqref{e2a}, we also deduce that, for any $t\in[0,T]$, 
	\begin{align*}
	\int_0^t \mathcal{Q}_\delta^M(s)ds&\leq C(T)\int_0^T\left(\tilde{E}^{CH}[t;u_\delta^M(t)]+\Vert u_\delta^M\Vert^2\right)dt+C(T)\int_0^T\Vert \tgrad w_\delta^M\Vert^2 dt\\&\leq C(T)\left(1+\int_0^T\vert m(w_\delta^M,1)\vert^2dt\right),
	\end{align*}
	having applied \eqref{energylimit}, \eqref{ud} and the conservation of mass. This, together with \eqref{passage3}, gives 	\begin{align}
	{t\mathcal{Q}^M_\delta(t)}+\int_0^t\frac{s}{4}\Vert\tgrad \partial^\bullet u^M_\delta\Vert^2ds+ \int_0^t{K s}\Vert \partial^\bullet u^M_\delta\Vert^2ds\leq C(T)\left(1+\int_0^T\vert m(w_\delta^M,1)\vert^2dt\right),
	\label{passage4}
	\end{align}
	where again we applied \eqref{energylimit} and \eqref{ud} for the last summands in the right-hand side of \eqref{passage3}. Note also that, by the definition \eqref{ees} of $\mathcal{Q}_\delta^M$, from \eqref{passage4} we infer	\begin{align*}&\nonumber\frac t 2\Vert \tgrad w_\delta^M\Vert^2+\int_0^t\frac{s}{4}\Vert\tgrad \partial^\bullet u^M_\delta\Vert^2ds+ \int_0^t{K s}\Vert \partial^\bullet u^M_\delta\Vert^2ds\\&\leq C(T)\left(1+\int_0^T\vert m(w_\delta^M,1)\vert^2dt\right)\nonumber\\&\nonumber-ta_N(u^M_\delta,w^M_\delta)-tg(u_\delta^M,w^M_\delta)
-\frac{\theta t}{2}g(F_{ln}^\delta(u^M_\delta),1)\nonumber\\&\leq
C(T)\left(1+\int_0^T\vert m(w_\delta^M,1)\vert^2dt\right)+\frac{t}{4}\Vert \tgrad w_\delta^M\Vert^2+C(T)(1+\tilde{E}^{CH}[t;u_\delta^M]),	\end{align*}
due to \eqref{ud}, \eqref{e1a} and \eqref{e2a}. This implies, by \eqref{energylimit},
\begin{align}	
\frac t 4\Vert \tgrad w_\delta^M\Vert^2+\int_0^t\frac{s}{4}\Vert\tgrad \partial^\bullet u^M_\delta\Vert^2ds+ \int_0^t{K s}\Vert \partial^\bullet u^M_\delta\Vert^2ds\leq C(T)\left(1+\int_0^T\vert m(w_\delta^M,1)\vert^2dt\right),	\label{passage5}	\end{align}
for any $t\in[0,T]$, where $C(T)$ is independent of $M,\delta$.
By standard compactness arguments, exploiting the regularity we have just obtained, arguing as in \cite[Sec.4.2]{CaeEll21}, we can pass to the limit as $M\to\infty$ and obtain the existence of a couple $(u_\delta,w_\delta)$ satisfying for any $\eta\in L^2_{H^1}$ and all $t\in[0,T]$,
	\begin{align}
&m_\star(\partial^\bullet u_\delta, \eta)+g(u_\delta,\eta)+a_N(u_\delta,\eta)+a_S(w_\delta,\eta)=0,\label{CH2_b}\\&
a_S(u_\delta,\eta)+m((F^\delta)^\prime(u_\delta),\eta)-m(w_\delta,\eta)=0,\label{mu2_b}
\end{align}
and $u_\delta(0)=u_0$ almost everywhere in $\Gamma_0$. Moreover, it holds 
\begin{align*}
\sup_{t\in [0,T]}\tilde{E}^{CH}[t;u_\delta(t)]+{\color{black}\sup_{t\in [0,T]}\Vert u_\delta(t)\Vert^2}+\int_0^T\Vert w_\delta\Vert^2_{H^1(\Gamma(t))}dt &\leq C_\delta(T),\\
	t\Vert w_\delta\Vert_{H^1(\Gamma(t))}^2+\int_0^t s\Vert\partial^\bullet u_\delta\Vert^2_{H^1(\Gamma(s))}ds&\leq C{\color{black}_\delta}(T), \quad
    \forall\,t\in[0,T],
\end{align*}
for some $C_\delta(T)>0$ possibly depending on $\delta$.
Now note that, since clearly (see \cite[Sec.4.2]{CaeEll21}) we have $u_\delta^M\to u_\delta$ in $L^2_{L^2}$ as $M\to \infty$, and being $\vert F^\prime_\delta(c_\delta^M)\vert\leq C_\delta\vert c_\delta^M\vert$ and ${\vert F^{\prime\prime}_\delta\vert}\leq C_\delta$, we have, since $m(w_\delta^M,1)=m(F^\prime(u_\delta^M),1)$ and $m(w_\delta,1)=m(F^\prime(u_\delta),1)$,
\begin{align*}
\left\vert\int_0^T\left(\vert m(w_\delta^M,1)\vert^2-\vert m(w_\delta,1)\vert^2\right)dt\right\vert&\leq \int_0^T\left(\vert m(F^\prime(u_M^\delta),1)\vert+\vert m(F^\prime(u^\delta),1)\vert\right)\left\vert m(F^\prime(u_M^\delta)-F^\prime(u^\delta),1)\right\vert dt\\&\leq C_\delta(T)\Vert u_\delta^M-u_\delta\Vert_{L^2_{L^2}}\to 0\quad \text{as }M\to\infty,\end{align*}
having also exploited the uniform controls on $u_\delta^M$ and $u_\delta$ in $L^\infty_{L^2}$. Therefore, these result, together with the bound \eqref{pp2} (which gives weak lower sequential semicontinuity of the norms on the left-hand side of \eqref{passage5}), is enough to pass to the limit as $M\to\infty$ in \eqref{passage5}, so that it holds
\begin{align}
\esssup_{t\in[0,T]}\left(\frac t 4{\Vert \tgrad w_\delta\Vert^2}\right)+\int_0^T\frac{s}{4}\Vert\tgrad \partial^\bullet u_\delta\Vert^2ds+\int_0^T{K s}\Vert \partial^\bullet u_\delta\Vert^2ds \leq C(T)\left(1+\int_0^T\vert m(w_\delta,1)\vert^2 ds\right),
\label{ultimate}
\end{align}
where this time $C(T)>0$ is a constant \textit{independent} of $\delta$.

\textbf{Step 2. Limit as $\delta\to 0$.} 
In order to pass to the limit in $\delta$ we need to find a more refined energy estimate. In particular, let us first observe that again, from \eqref{CH2_b}, it holds the conservation of total mass, by choosing $\eta\equiv 1$. Then we test \eqref{CH2_b} by $\eta=u_\delta$, inferring 
\begin{align*}
\frac 1 2\frac{d}{dt}\Vert u_\delta\Vert^2=m_\star(\partial_t u_\delta,u_\delta)+\frac 1 2 g(u_\delta,u_\delta)=-\frac 1 2 g(u_\delta,u_\delta)-a_N(u_\delta,u_\delta)-a_S(w_\delta,u_\delta).
\end{align*}
Integrating by parts, exploiting also \eqref{mu2_b}, we get, by the definition of $F^{\prime\prime}_\delta$,
\begin{align}
a_S(w_\delta,u_\delta)=-\int_{\Gamma(t)} w_\delta\Delta_\Gamma u_\delta&=\Vert\Delta_\Gamma u_\delta\Vert^2+\int_{\Gamma(t)}F_\delta^{\prime\prime}(u_\delta)\vert \tgrad u_\delta\vert^2\nonumber\\&=\Vert\Delta_\Gamma u_\delta\Vert^2+\int_{\Gamma(t)}\varphi^\prime_\delta(u_\delta)\vert \tgrad u_\delta\vert^2-\Vert \tgrad u_\delta\Vert^2.
\label{e}
\end{align}
Notice that by \cite[Prop. 4.12]{CaeEll21} the solution to \eqref{CH2_b}-\eqref{mu2_b} is unique and thus it enjoys the regularity $L^2_{H^2}$ as in \cite[Thm. 4.14]{CaeEll21}. As a consequence, the second equation (i.e. the one for $w_\delta$) holds pointwise almost everywhere, and therefore the regularity of $u_\delta$ is sufficient to perform rigorously the calculation above. Recalling now that $\varphi_\delta^\prime \geq 0$ and by the regularity of $\textbf{V}$, we end up with
\begin{align*}
\frac 1 2\frac{d}{dt}\Vert u_\delta\Vert^2+\Vert \Delta_\Gamma u_\delta\Vert^2&\leq -\frac 1 2 g(u_\delta,u_\delta)-a_N(u_\delta,u_\delta)+\Vert \tgrad u_\delta\Vert^2\leq C(\Vert u_\delta\Vert^2+\Vert \tgrad u_\delta\Vert^2).
\end{align*}
We can now apply the following basic interpolation inequality, which can be obtained after an integration by parts and an application of Cauchy-Schwarz and Young's inequalities: %{\color{purple} (should explain that $C$ can be taken indp. of $t$?)}
$$
\Vert \tgrad u_\delta\Vert^2=-\int_{\Gamma(t)} u_\delta\Delta_\Gamma u_\delta\leq \frac 1 2\Vert u_\delta\Vert^2+\frac 1 2\Vert \Delta_\Gamma u_\delta\Vert^{2},
$$
allowing to infer, by Young's inequality,
\begin{align*}
\frac 1 2\frac{d}{dt}\Vert u_\delta\Vert^2+\frac{1}{2}\Vert \Delta_\Gamma u_\delta\Vert^2\leq  C\Vert u_\delta\Vert^2,
\end{align*}
so that, in the end, by Gronwall's Lemma, we deduce
\begin{align}
\sup_{t\in[0,T]}\Vert u_\delta\Vert^2+\int_0^T\Vert \Delta_\Gamma u_\delta\Vert^2ds\leq C(T),
    \label{l2reg}
\end{align}
uniformly in $\delta$. Concerning the energy estimate, by the same arguments as for \eqref{enel} we get:
\begin{align}
\frac{d}{dt}\tilde{E}^{CH}[u_\delta]+\frac{1}{2}\Vert\tgrad w_\delta^M\Vert^2\leq -g(u_\delta,w_\delta)+C_0+C_1\tilde{E}^{CH}[u_\delta].
\label{enel2}
\end{align}
Thanks to \eqref{l2reg} and the bound on $\mathbf V$, by Poincaré's inequality we have
\begin{align}
-g(u_\delta,w_\delta)\leq C\Vert u_\delta\Vert\Vert w_\delta\Vert\leq C\Vert w_\delta\Vert\leq \widetilde{C}(\vert m(w_\delta,1)\vert+\Vert\tgrad w_\delta\Vert),
\label{energ1}
\end{align}}with $\widetilde{C}>0$ a suitable constant independent of $\delta$.
Then, using known arguments together with the conservation of total mass (see also the proof of \cite[Prop.5.7, a)]{CaeEll21}) we deduce, from \cite[(5.1.7)]{CaeEll21},
\begin{align}
\vert m(w_\delta,u_\delta)\vert\leq \frac{1-\alpha}{8\widetilde{C}}\Vert \tgrad w_\delta\Vert^2+C\Vert \tgrad u_\delta\Vert^2+\alpha\vert m(w_\delta, 1)\vert,
\label{alpha}
\end{align}
where $\alpha\in[0,1)$ is a constant such that 
\begin{align}
0\leq m_{u_0}(t)\leq \alpha<1,
\label{alpa}
\end{align}
for any $t\in[0,T]$, whose existence is guaranteed being $\sup_{[0,T]}m_{u_0}(t)<1$ by assumption. Now, from \cite[(5.1.5)]{CaeEll21},
\begin{align}
\vert m(w_\delta,1)\vert\leq C\left(1+\Vert\tgrad  u_\delta\Vert^2\right)+\vert m(w_\delta,u_\delta)\vert,
\label{muu3}
\end{align}
thus by \eqref{alpha} we infer
\begin{align}
\vert m(w_\delta,1)\vert\leq C\left(1+\Vert\tgrad  u_\delta\Vert^2\right)+\frac{1}{8\widetilde{C}}\Vert\tgrad w_\delta\Vert^2,
\label{mu3}
\end{align}
{\color{black}so that, together with \eqref{enel2} and \eqref{energ1}, we deduce 
\begin{align*}
    \frac{d}{dt}\tilde{E}^{CH}[u_\delta]+\frac{3}{8}\Vert\tgrad w_\delta^M\Vert^2\leq C(1+\tilde{E}^{CH}[u_\delta]),
\end{align*}
entailing, by Gronwall's Lemma,
\begin{align}
\esssup_{t\in[0,T]}\tilde{E}^{CH}[u_\delta]+\int_0^T\Vert\tgrad w_\delta\Vert^2dt\leq C(T),
    \label{unifcontrol}
\end{align}
uniformly in $\delta$.}
Observe now that by \eqref{unifcontrol} we can deduce a better estimate of $\vert m(w_\delta,u_\delta)\vert$. In particular, from \cite[(5.1.6)]{CaeEll21}, \eqref{unifcontrol} and the conservation of mass, by Poincaré's inequality we infer
	\begin{align*}
	\vert m(w_\delta,u_\delta)\vert&\leq
	 C\Vert\tgrad w_\delta\Vert\Vert u_\delta-(u_\delta)_{\Gamma(t)}\Vert+m_{u_0}(t)\vert m(w_\delta,1)\vert\\&\leq C\Vert\tgrad w_\delta\Vert+\alpha\vert m(w_\delta,1)\vert,
	\end{align*}
	for $\alpha$ already defined {\color{black}in \eqref{alpa}}. On account of \eqref{l2reg}, \eqref{unifcontrol}, we have $\Vert u_\delta\Vert_{H^1(\Gamma(t))}\leq C(T)$ for almost any $t\in[0,T]$, so that{\color{black}, since $\vert\varphi_\delta(r)\vert\leq \varphi_\delta(r)r+1$ and $m(w_\delta,1)=-m(u_\delta,1)+m(\varphi_\delta(u_\delta),1)$,
	\begin{align*}
\vert m(w_\delta,1)\vert &\leq \vert m(u_\delta,1)\vert+\vert m(\varphi_\delta(u_\delta),1)\vert\\&\leq \vert m(u_\delta,1)\vert+\vert m(\varphi_\delta(u_\delta),u_\delta)\vert +\vert \Gamma(t)\vert\\&\leq C+\Vert u_\delta\Vert^2+\Vert\tgrad  u_\delta\Vert^2+\vert m(w_\delta,u_\delta)\vert\\
&\leq C(1+\Vert\tgrad w_\delta\Vert)+\alpha\vert m(w_\delta,1)\vert,
	\end{align*}
	 recalling $m(\varphi_\delta(u_\delta),u_\delta)=m(w_\delta,u_\delta)-\Vert \nabla u_\delta\Vert^2$.} This implies, being $\alpha<1$, by Poincaré's inequality,
		\begin{align}
	\label{w}
	\vert m(w_\delta,1)\vert\leq C(1+\Vert \tgrad  w_\delta\Vert),
	\end{align}
{\color{black}from which we infer, together with \eqref{unifcontrol},
$$
\int_0^T\vert m(w_\delta,1)\vert^2\leq C(T),
$$
entailing, by \eqref{ultimate},
\begin{align}
\esssup_{t\in[0,T]}\left({ \frac t 4\Vert \tgrad w_\delta\Vert^2}\right)+\int_0^T\frac{s}{4}\Vert\tgrad \partial^\bullet u_\delta\Vert^2ds+\int_0^T{K s}\Vert \partial^\bullet u_\delta\Vert^2ds \leq C(T).
\label{ultimate2}
\end{align}}
Following \cite[Sec.4.2]{CaeEll21}, \cite[Propositions 5.8-5.10]{CaeEll21} and \cite[Lemma 5.11]{CaeEll21}, the uniform estimates \eqref{l2reg}, \eqref{unifcontrol} and \eqref{ultimate2} are then enough, by standard compactness arguments, to pass to the limit as $\delta\rightarrow 0$, obtaining a weak solution with the same properties as the one of Theorem \ref{existence}. Therefore, this solution coincides with the one of Theorem \ref{existence} by uniqueness, and it also enjoys \eqref{mu5} and \eqref{mu4}. This concludes the proof of Part (i).
	\vskip 2mm
	
	\underline{Part (ii).} Let us fix $\tau>0$. We clearly have that $\Vert w\Vert_{H^1(\Gamma(t))}\leq C$ for almost any $t\geq\tau$. Let then introduce the cutoff function
    \begin{align}
    h_k(r)=\begin{cases}
    1-\frac{1}{k},\qquad r>1-\frac{1}{k},\\
    r,\qquad -1+\frac{1}{k}\leq r\leq 1-\frac{1}{k},\\
    -1+\frac{1}{k},\qquad r<-1+\frac{1}{k},
    \end{cases}
    \label{hk}
    \end{align}
    which is Lipschitz continuous. Then we set $u_k=h_k(u)$. Being $u$ in $L^\infty_{H^1}$, we have that the chain rule holds giving
    $$
    \nabla_{\Gamma} u_k=\chi_{[-1+\frac{1}{k},1-\frac{1}{k}]}(u)\nabla_{\Gamma} u.
    $$
    Accordingly, for any $k>1$ and $p\geq 2$, $f_k=\left\vert \frac{\theta}{2}\varphi (u_k)\right\vert^{p-2}\frac{\theta}{2}\varphi (u_k)$ is well defined and belongs to $L^\infty_{H^1}$ and satisfies
    $$
    \nabla_{\Gamma}\left(\left\vert \frac{\theta}{2}\varphi (u_k)\right\vert^{p-2}\frac{\theta}{2}\varphi (u_k)\right)=(p-1)\left\vert \frac{\theta}{2}\varphi (u_k)\right\vert^{p-2}\frac{\theta}{2}\varphi^\prime(u_k)\nabla_{\Gamma} u_k.
    $$
    If we now set $\eta=f_k$ in \eqref{mu}, we infer that
    \begin{align*}
    &(p-1)\int_{\Gamma(t)}\left\vert \frac{\theta}{2}\varphi (u_k)\right\vert^{p-2}\frac{\theta}{2}\varphi^\prime(u_k)\nabla_{\Gamma}u\cdot \nabla_{\Gamma}u_k+\int_{\Gamma(t)}\left\vert \frac{\theta}{2}\varphi (u_k)\right\vert^{p-2}\frac{\theta}{2}\varphi (u_k)\frac{\theta}{2}\varphi (u)=\int_{\Gamma(t)} \widehat{w}\left\vert \frac{\theta}{2}\varphi (u_k)\right\vert^{p-2}\frac{\theta}{2}\varphi (u_k),
    \end{align*}
    where $ \widehat{w}=w+u$. Being $F_{ln}$ strictly convex, the first term in the left-hand side is nonnegative. Moreover, since $\varphi$ is increasing, by the definition of $u_k$, we immediately infer
    \begin{align}
    \varphi (u_k)^2\leq \varphi (u)\varphi(u_k), \quad \forall k>1
    \label{uk}
    \end{align}
    and thus for the second term we have
    \begin{align*}
        \int_{\Gamma(t)}\left\vert \frac{\theta}{2}\varphi (u_k)\right\vert^{p-2}\frac{\theta}{2}\varphi (u_k)\frac{\theta}{2}\varphi (u) \geq \int_{\Gamma(t)}\left\vert \frac{\theta}{2}\varphi (u_k)\right\vert^{p}
    \end{align*}
    Regarding the right-hand side, we easily get, by the Sobolev embedding $H^1(\Gamma(t))\hookrightarrow L^p(\Gamma(t))$,
    \begin{align*}
    \int_{\Gamma(t)} \widehat{w}\left\vert \frac{\theta}{2}\varphi (u_k)\right\vert^{p-2}\frac{\theta}{2}\varphi(u_k)&\leq \frac{1}{2}\left\Vert\frac{\theta}{2}\varphi (u_k)\right\Vert^p_{L^p(\Gamma(t))}+C\Vert \widehat{w}\Vert^p_{L^p(\Gamma(t))}\\&\leq \frac{1}{2}\left\Vert\frac{\theta}{2}\varphi (u_k)\right\Vert^p_{L^p(\Gamma(t))}+C_p\Vert \widehat{w}\Vert^p_{H^1(\Gamma(t))},
    \end{align*}
    with $C_p>0$ depending on $p$.
    Collecting the above estimates, being ${u}\in L^\infty_{H^1}$ and $\Vert w\Vert_{H^1(\Gamma(t))}\leq C$ for almost any $t\geq\tau$, we immediately deduce
  \begin{align}
  \Vert\varphi(u_k)\Vert_{L^p(\Gamma(t))}\leq C_p(T,\tau,p),\quad \forall p\in[2,\infty),
  \label{conv}
  \end{align}
  for almost any $t\in[\tau,T]$. Therefore, there exists $\zeta\in L^\infty_{L^p}(\tau,T)$ such that, up to subsequences $\varphi(u_k)\overset{\ast}{\rightharpoonup}\zeta$ in $L^\infty_{L^p}(\tau,T)$ as $k\rightarrow\infty$. Now observe that, being $\vert u\vert<1$ almost everywhere on $\Gamma(t)$ (for almost any $t\in(0,T)$), and since, for almost any $t\in(0,T)$, $u_k\to u$ as $k\to\infty$ almost everywhere on $\Gamma(t)$, we get $\varphi(u_k)\to \varphi(u)$ almost everywhere. It is now immediate to deduce by \eqref{conv} that, for $p=2$, $\Vert\varphi(u_k)\Vert_{L^2_{L^2}}\leq C(\tau,T)$, therefore we can apply \cite[Thm.B.2]{CaeEll21} on $[\tau,T]$ to infer $\varphi(u_k)\rightharpoonup \varphi(u)$ in $L^2_{L^2}(\tau,T)$.   %being $\vert u_k\vert\leq \vert u\vert$, there exists a compact set $K\subset (0,1)$ depending on $x\in \Gamma(t)$ and $t$ (for which it holds $\vert u(x,t)\vert<1$) such that $\left\{\{u_k(x,t)\}_k,u(x,t)\right\}\subset K$, 
  By uniqueness of the weak limit we therefore infer $\zeta=\varphi(u)$. Then, by weak$^\star$ sequential lower semicontinuity we get
  \begin{align}
\esssup_{t\in[\tau,T]}\Vert\varphi(u)\Vert_{L^p(\Gamma(t))}\leq C_p(T,\tau,p),\quad \forall p\in[2,\infty).
\label{phip}
\end{align}
 Concerning $\varphi^\prime$, we consider $g_k=\frac{\theta}{2}\varphi(u_k)e^{L\frac{\theta}{2}\vert \varphi(u_k)\vert}\in L^\infty_{H^1}$, for some arbitrary $L>0$, and we observe that
 \begin{align*}
 \nabla_{\Gamma}\left(\frac{\theta}{2}\varphi(u_k)e^{L\frac{\theta}{2}\vert\varphi(u_k)\vert}\right)=\frac{\theta}{2}\varphi^\prime(u_k)\left(1+L\frac{\theta}{2}\left\vert\varphi(u_k)\right\vert\right)e^{L\frac{\theta}{2}\vert\varphi(u_k)\vert}\nabla_{\Gamma} u_k.
 \end{align*}
 Therefore, considering again \eqref{mu} with $\eta=g_k$, we get
 \begin{align*}
 \int_{\Gamma(t)}\nabla_{\Gamma}u\cdot \nabla_{\Gamma}u_k \frac{\theta}{2}\varphi^\prime(u_k)\left(1+L\frac{\theta}{2}\left\vert\varphi(u_k)\right\vert\right)e^{L\frac{\theta}{2}\vert\varphi(u_k)\vert}+\int_{\Gamma(t)}\frac{\theta}{2}\varphi(u)\frac{\theta}{2}&\varphi(u_k)e^{L\frac{\theta}{2}\vert\varphi(u_k)\vert}\\
 &=\int_{\Gamma(t)} \widehat{w}\frac{\theta}{2}\varphi(u_k)e^{L\frac{\theta}{2}\vert \varphi(u_k)\vert}.
 \end{align*}

 Again the first term in the left-hand side is nonnegative, whereas, exploiting again \eqref{uk}, we obtain in the end
  %the others factors are positive,
 \begin{align*}
 \int_{\Gamma(t)}\left(\frac{\theta}{2}\varphi(u_k)\right)^2e^{L\frac{\theta}{2}\vert\varphi(u_k)\vert}\leq \int_{\Gamma(t)} \widehat{w}\frac{\theta}{2}\varphi(u_k)e^{L\frac{\theta}{2}\vert \varphi(u_k)\vert}
 \end{align*}
By Lemma \ref{gener}, we get
 \begin{align}
 \int_{\Gamma(t)}\vert  \widehat{w}\vert \left\vert \frac{\theta}{2}\varphi(u_k)\right\vert e^{L\left\vert \frac{\theta}{2}\varphi(u_k)\right\vert}\leq \frac{1}{2}\int_{\Gamma(t)}\left\vert \frac{\theta}{2}\varphi(u_k)\right\vert^2 e^{L\left\vert \frac{\theta}{2}\varphi(u_k)\right\vert}+\int_{\Gamma(t)}e^{N\vert  \widehat{w}\vert},
 \label{pp}
 \end{align}
 implying
 \begin{align}
 \label{MTineq}
 \frac{1}{2}\int_{\Gamma(t)}\left\vert \frac{\theta}{2}\varphi(u_k)\right\vert^2 e^{L\left\vert \frac{\theta}{2}\varphi(u_k)\right\vert}\leq \int_{\Gamma(t)}e^{N\vert  \widehat{w}\vert}
 \end{align}
 for any $L>0$ and some $N=N(L)$.
 Now we exploit \eqref{equiv}, and then apply Lemma \ref{Trudi} with $u=N \widetilde{\widehat{w}}$ and the manifold $\mathcal{M}=\Gamma_0$ (with the corresponding metric) to infer
 \begin{align*}
      \int_{\Gamma(t)}e^{N\vert \widehat{w}\vert}=\int_{\Gamma_0}e^{N\vert\widetilde{\widehat{w}} \vert}J_t^0d\Gamma_0&\leq C\int_{\Gamma_0}e^{N\vert\widetilde{\widehat{w}} \vert}d\Gamma_0
      \leq Ce^{CN^2\Vert \widetilde{\widehat{w}}\Vert^2_{H^1(\Gamma_0)}} 
      \leq Ce^{CN^2\Vert {\widehat{w}}\Vert^2_{H^1(\Gamma(t))}},
 \end{align*}
 where in the last estimate we exploited the property that $(H^1(\Gamma(t)),\phi_t)_{t\in[0,T]}$ is a compatible space, where $\phi_{-t} v= \widetilde{v}$.
 Recalling now property \eqref{explog}, we infer that, for some $\tilde{C}>0$ sufficiently large, %(C_1+x^2e^Cx)>=e^Cx, with $C_1$ suff large
 \begin{align}
 \left(\frac{\theta}{2}\varphi^\prime(s)\right)^p\leq e^{pC}\left(\tilde{C}+\left\vert\frac{\theta}{2}\varphi(s)\right\vert^2e^{pC\left\vert\frac{\theta}{2}\varphi(s)\right\vert}\right),\quad\forall s\in(-1,1),
 \label{phiprime}
 \end{align}
 thus, taking $L=pC$ in \eqref{MTineq} and recalling that $\Vert \widehat{w}\Vert_{H^1(\Gamma(t))}\leq C$ for almost any $t\geq\tau$, we end up with
 $$
 \Vert\varphi^\prime(u_k)\Vert_{L^p(\Gamma(t))}\leq C_p(T,\tau,p),
 $$
 implying, by the same arguments used for $\varphi(u_k)$, applied in this case to $\varphi^\prime(u_k)$, that
  \begin{align}
 \esssup_{t\in[\tau,T]}\Vert\varphi^\prime(u)\Vert_{L^p(\Gamma(t))}\leq C_p(T,\tau,p),\quad \forall p\in[2,\infty).
 \label{phip1}
 \end{align}
  Therefore, exploiting elliptic regularity and recalling that $u\in L^\infty_{H^1}$, we obtain
  $$
  \Vert u\Vert_{H^2(\Gamma(t))}\leq \left(C+\Vert\Delta_{\Gamma(t)}u\Vert\right)\leq C\left(1+\Vert w\Vert+\Vert\varphi(u)\Vert\right)\leq C(T,\tau),
  $$
  for almost any $t\in[\tau,T]$.
  
  \vskip 2mm
  
  \underline{Part (iii).} If we now apply the chain rule to $\varphi(u)$ (which is possible, e.g., by approximation with the truncated functions $u_k$ and then passing to the limit as $k\rightarrow\infty$) we obtain
  	\begin{align*}
  	\tgrad  \varphi(u)=\varphi'(u)\tgrad u, \quad\text{ for a.a. }\ t\in [\tau,T].
  	\end{align*}
  	Then, for almost any $t\in[\tau,T]$, we have that
  $$
  \Vert \tgrad  \varphi(u)\Vert_{L^p(\Gamma(t))} \leq \Vert  \varphi^\prime(u) \Vert_{L^{2p}(\Gamma(t))}\Vert \tgrad u \Vert_{L^{2p}(\Gamma(t))}\leq C_p(T,\tau,p),
  $$
  by the Sobolev embedding $H^2(\Gamma(t))\hookrightarrow W^{1,q}(\Gamma(t))$ for every $q\geq2$. Therefore we get
  $$
  \Vert \varphi(u)\Vert_{W^{1,p}(\Gamma(t))}\leq C_p(T,\tau,p),$$ for every $p\geq2$. This implies, choosing, e.g., $p=3$, that
  \begin{align}\label{eq:diogo1}
  \Vert \varphi(u)\Vert_{ L^\infty(\Gamma(t))}\leq C(T,\tau),\quad \text{ for a.a. }t\in[\tau,T],
\end{align}
  by the embedding $W^{1,3}(\Gamma(t))\hookrightarrow L^\infty(\Gamma(t))$.
  	Therefore, being $u(t)\in H^2(\Gamma(t))\hookrightarrow C^0(\Gamma(t))$ for almost any $t\in[0,T]$, it follows from the singularities of $\varphi$ at $\pm 1$ and the estimate \eqref{eq:diogo1} that we can find $\xi=\xi(T,\tau)>0$ such that
  	\begin{align*}
  	\Vert u \Vert_{L^\infty(\Gamma(t))}\leq 1-{\xi}, \quad\text{ for a.a. }t\in[\tau,T],
  	\end{align*}	
  	that is, the strict separation property holds. This concludes the proof.
\end{proof}

\section{The second model}
\label{secondmodel}
We now consider the alternative weighted model \eqref{intro:model2}, which is analysed in \cite[Sec.6]{CaeEll21} (and references therein). {\color{black} In particular, as noticed in the Introduction, this is a simplified version of the model presented in \cite{Sauer,ZimTosLan19}, in which the problem is governed by two coupled fourth-order nonlinear PDEs that live
on an evolving two-dimensional manifold. For the phase transitions, the PDE is the Cahn-Hilliard equation for curved surfaces, which can be derived from surface mass balance in the
framework of irreversible thermodynamics. For the surface deformation, the PDE is the (vector-valued) Kirchhoff--Love thin shell equation. In our work, we study only the Cahn-Hilliard equation in the same formulation arising from the model in \cite{ZimTosLan19}, so that this analysis could, in a future work, be extended to consider the complete model as \cite{Sauer,ZimTosLan19}, in which the evolution of the surface is part of the model itself. We now show a sketch of the derivation of our model \eqref{intro:model2}, taken directly from \cite{ZimTosLan19}.
}
\subsection{Derivation}
\label{derivation2}
{\color{black}
For consistency with the literature, see e.g. \cite{Sauer,ZimTosLan19}, to derive this system we start with a description of the surfaces $\{\Gamma(t)\}_t$ given by the flow $\Phi_t^0\colon \Gamma_0 \to \Gamma(t)$ as in $\mathbf{A}_\Phi$. Assume that the surface $\Gamma(t)$ consists of two species with the mass densities per unit area $\rho_1$ and
$\rho_2$. The total mass of each species is assumed to be conserved. This entails, for the total
density $\rho := \rho_1 + \rho_2$, the balance law
\begin{align*}
    \dfrac{d}{dt}\int_{P(t)} \rho \equiv 0,
\end{align*}
for any surface patch $P(t)\subset \Gamma(t)$ evolving under the full velocity $\mathbf V$. By the Reynolds transport theorem, see also Proposition \ref{prop:transport}, and the arbitrariness of $P(t)$ we easily deduce
\begin{align*}
\partial^\bullet \rho+\rho\tgrad\cdot \textbf{V}=0, \quad \rho(0)=\widehat{\rho},
\end{align*}
where $\widehat \rho$ is the initial density. This equation shows in particular that we have the identity
\begin{align}
\rho(t,\Phi_t^0(p))=\frac{\widehat{\rho}}{J_t^0(p)},\quad \forall p\in \Gamma_0,
\label{rho_0}
\end{align}
where $J_t^0$ is the area change defined in \eqref{equiv}.

For the process on the surface we now introduce the dimensionless concentrations $c_i = \rho_i / \rho$, for $i=1,2$. Since $c_1+c_2 = 1$, it is sufficient to consider $c_1$ to model the local density fractions of the two species. The mass of species 1 in the membrane patch $P(t)$, evolving under $\mathbf V$, may only change due to a diffusive mass flux $q_d$ at the boundary, so that 
\begin{align}
\int_{P(t)}\rho\partial^\bullet{c}_1=\dfrac{d}{dt}\int_{P(t)}\rho c_1=-\int_{\partial P(t)}q_d\cdot \pmb\mu=-\int_{P(t)}\tgrad\cdot q_d,
\label{pppa}
\end{align}
where $\pmb\mu$ denotes the outer unit conormal of $\partial P(t)$. The last identity comes from the fact that we can directly consider, without loss of generality, $q_d$ to be purely tangential to $\Gamma(t)$. The diffusive flux $q_d$ is related to the chemical potential in the following sense. In analogy to 3D problems (\cite{CahHil58}), the Cahn-Hillard energy per reference area takes the form
$$\Psi_{CH} = \Psi_{mix}(c_1, T) + \Psi_i (J_t^0,\tgrad c_1),$$
where $T$ is the absolute temperature and, up to setting some constants to $1$ for simplicity, 
\begin{align*}
&\Psi_{mix}(c_1, T)=T\left(c_1\ln(c_1)+(1-c_1)\ln(1-c_1)\right)+c_1(1-c_1),
\\& \Psi_i(J_t^0,\tgrad c_1)=\frac{\lambda}2 J_t^0\vert \tgrad c_1\vert^2,
\end{align*}
with $\lambda>0$ a coefficient related to the width of the phase interface. As noted in \cite[(35)]{ZimTosLan19}, the area change $J^0_t$ is included
in $\Psi_i$, since $\Psi_{CH}$ is postulated to be an energy w.r.t. the reference configuration $\Gamma_0$, while $\tgrad c_1$ refers to the current
configuration (it can be viewed as having units of 1/(current length)). If we then want a Cahn-Hilliard energy per current area instead of one per reference area we simply need to multiply $\Psi_{CH}$ by $J^t_0=(J_t^0)^{-1}$; in effect, we have (current area)/(reference area)$=J^0_t$, thus, heuristically, $$\frac{\text{energy}}{\text{current area}}=\frac{\text{reference area}}{\text{current area}} \frac{\text{energy}}{\text{reference area}}=\frac 1 {J_t^0}\frac{\text{energy}}{\text{reference area}},$$ so that the total energy on the current configuration $\Gamma(t)$ can be defined as 
\begin{align}
E_{CH}^\rho(c_1):=\int_{\Gamma(t)}J^t_0\Psi_{CH}=\int_{\Gamma(t)}J^t_0\Psi_{mix}+\frac{\lambda}{2}\int_{\Gamma(t)}\vert \tgrad c_1\vert^2.
\label{energa}
\end{align}
Then, following the thermodynamical derivation in \cite[Appendix A]{ZimTosLan19} one can obtain that the chemical potential $w$ is defined by \footnote{As already noticed, this is performed in \cite[Appendix A]{ZimTosLan19} in the case of an elastic surface, but we keep the same derivation as a first step in the analysis of the more complex model.}
\begin{align}
w =J^0_t\frac{\delta E_{CH}^\rho}{\delta c_1} =- J^0_t\lambda\Delta_\Gamma c_1 +\Psi'_{mix}(c_1),
\end{align} 
and the diffusive flux is (see \cite[(121)]{ZimTosLan19})
$$q_d = -\frac{M}{J^0_t}\nabla_\Gamma w,$$
where $M>0$ is a mobility coefficient. The presence of the term $1/J^0_t$ in $q_d$ can again be heuristically explained by the fact that $w$ is defined per reference area, whereas $q_d$ is defined per current area. All in all, by setting for simplicity $M=1$, $\lambda=1$ and $\widehat{\rho}\equiv 1$, so that $J^0_t=1/\rho$, from \eqref{pppa} we are led to
\begin{align*}
\rho\dot c_1 - \nabla_\Gamma \cdot \left(\rho\nabla_\Gamma\left( -\dfrac{1}{\rho} \Delta_\Gamma c_1 + \Psi'_{mix}(c_1) \right) \right) = 0.    
\end{align*}
Problem \eqref{intro:model2} can be retrieved by rewriting the equation for the dimensionless concentration difference $c:=\frac{\rho_1-\rho_2}{\rho}={2c_1-1}$ and making a proper rescaling on $\Psi_{mix}$ so that it coincides with the potential $F$ defined in \eqref{sing}.
} 
\begin{remark}
\label{conservation}
{\color{black} Two observations are timely regarding the relation between models \eqref{intro:model1} and \eqref{intro:model2}. 
\begin{itemize}
    \item[1.] As observed in \cite{YusQuaOls20}, model \eqref{intro:model2} is similar to model \eqref{intro:model1}, but neither is a particular case of the other. In model \eqref{intro:model1}, the equation is written
for the conserved variable $u = \rho c$ and under the assumption that the free energy functional depends on the concentration difference $u$ rather than on the relative concentration difference $c$ as in model \eqref{intro:model2}. Formally, systems \eqref{intro:model1} and \eqref{intro:model2} are the same problem for inextensible membranes (i.e., $\tgrad \cdot\textbf{V}\equiv0$) and if one assumes $\rho = const$. Here
the situation partially resembles the coupling of a two-component compressible fluid
flow with dissipative Ginzburg–Landau interface dynamics discussed in \cite{LowTru98}, where, depending on the
choice of the variables to define the energy functional, the effects of compressibility
have to be considered in the definition of the chemical
potential. These differences between the two models explain why the conserved quantities \eqref{mass1}-\eqref{mass2} are in principle different.
    \item[2.] In the derivation of \eqref{intro:model2}, we think of the tangential component $\mathbf V_\tau$ as fixing a parametrisation of $\Gamma(t)$, as well as describing advection on the surface; note that the balance laws are considered on portions evolving under the full velocity $\mathbf V$. This is consistent with the framework of the physics literature, namely \cite{Sauer, ZimTosLan19}, in which this tangential component $\textbf{V}_\tau$ is part of the unknowns. We can also see this model in the light of \eqref{intro:model1} by heuristically taking $\mathbf V_\tau = \mathbf V_a$, so that the term involving both velocities vanishes and the notion of a material time derivative coincides with both the one in our analytical framework, i.e. involving the parametrisation, and the usual physical notion of including the advection of material points on the surface.
\end{itemize}
}
\end{remark}

\subsection{Weak formulation}
Let {\color{black}Assumption $\mathbf{A_\Phi}$} hold. Then the problem reads: find a pair $(c,w)$ such that, for all $\eta\in L^2_{H^1}$,
\begin{align}
&m_\star(\rho\partial^\bullet c,\eta) +\int_{\Gamma(t)}\rho\tgrad w\cdot\tgrad \eta =0,\label{phi}\\&
\int_{\Gamma(t)}\tgrad c\cdot\tgrad \eta+\int_{\Gamma(t)}\rho F^\prime(c)\eta=\int_{\Gamma(t)}\rho w \eta,\label{mu6}\\
&c(0)=c_0, \qquad\text{ a.e. on } \Gamma_0. \label{ini}
\end{align}
Here the weight function $\rho$, i.e., the total density, is determined by the transport equation
\begin{align}
\partial^\bullet\rho+\rho\tgrad \cdot\mathbf{V}=0,
\label{ODE}
\end{align}
with $\rho(0)\equiv 1$ on $\Gamma_0$.
%Notice that we are supposing for simplicity $\mathbf{V}_a=\mathbf{V}_\tau=\mathbf{0}$. In the more general case we should ask more regularity on the difference $\mathbf{V}_a^\tau=\mathbf{V}_\tau-\mathbf{V}_a$.
As noticed in {\color{black}Section \ref{derivation2}}, we have an interesting characterisation of $\rho$:
\begin{align}
\rho(t,\Phi(t,p))=(J_t^0(p))^{-1},\quad \forall p\in \Gamma_0.
\label{rho}
\end{align}
In particular, we also have the bounds
\begin{align}
0<\frac{1}{C_\rho}\leq \rho\leq C_\rho,\quad \vert\tgrad \rho\vert\leq C_\rho,
\label{bounds}
\end{align}
with $C_\rho>1$. Note that the total energy of the system reads {\color{black}(see also \eqref{energa}, but with respect to the relative concentration difference $c$)}
$$
E^\rho_{CH}(c):=\int_{\Gamma(t)}\left(\frac{\vert\tgrad c\vert^2}{2}+\rho F(c)\right).
$$
    {\color{black}We recall the following result (see \cite[Thm.6.2]{CaeEll21}).} The statement below and its proof involve the \textit{weighted inverse Laplacian} {\color{black}operator, which we denote as $A_\rho^{-1}$}, define to be, for any 
\begin{align*}
f\in H^{-1}(\Gamma(t)) \quad \text{ such that } \quad{\color{black} m_\star( \rho f, 1 )} =0,
\end{align*}
 the unique solution {\color{black}$\zeta:=A_\rho^{-1}f$} to the problem
\begin{align}
\gint \rho \tgrad \zeta \cdot \tgrad \eta &= {\color{black}m_\star( \rho f, \eta)}  \quad \text{ with } \quad \gint \zeta= 0.
\end{align}
We define also, for such elements $f$, the weighted norm 
\begin{align}\label{eq:invlap}
    \|f\|_{\rho, -1} := \| \sqrt{\rho} \,  \tgrad \zeta\|= \langle \rho f, \zeta\rangle.
\end{align}
Note that there exist $C_1, C_2 > 0 $ such that $C_1\| \tgrad \zeta\| \leq  \|f\|_{\rho, -1} \leq C_2 \| \tgrad \zeta\|.$
%{\color{black}[Here one should show uniqueness explicitly also in the case $\mathbf{V}_a^\tau\not=\mathbf{0}$. The energy estimates for the existence part are already present in the proof of Theorem \ref{theorem}.]}
\begin{theorem}
	\label{existence2}
Let	$c_0\in H^1(\Gamma_0)$, $\vert c_0\vert\leq 1$, $\vert(c_0)_{\Gamma_0}\vert<1$ and $F:[-1,1]\to \R$ be given by \eqref{logpot}. Then there exists a unique pair $(c,w)$ with
	$$c\in L^\infty_{H^1}\cap H^1_{H^{-1}},\quad w\in L^2_{H^1},$$ such that, for almost any $t\in[0,T]$, $\vert c(t)\vert<1$ almost everywhere in $\Gamma(t)$ and $(c,w)$ satisfies, for almost any $t\in[0,T]$, \eqref{phi}-\eqref{mu6}, with $c(0)=c_0$ almost everywhere in $\Gamma_0$. The solution $c$ also satisfies the additional regularity
	$$c\in C^0_{L^2}\cap L^\infty_{L^p}\cap L^2_{H^2},$$
	for all $p\in[1,+\infty)$. Furthermore, if $c_{0,1},c_{0,2}\in H^1(\Gamma_0)$, satisfying the existence assumptions, are such that $(c_{0,1})_{\Gamma_0}=(c_{0,2})_{\Gamma_0}$, and $c_1,c_2$ are the solutions of the system with $c_1(0)=c_{0,1}$ and $c_2(0)=c_{0,2}$, then there exists a constant $C>0$ independent of $t$, such that, for almost any $t\in[0,T]$,
	\begin{align}
	\Vert c_1(t)-c_2(t)\Vert_{\rho,-1}\leq e^{Ct}\Vert c_{0,1}-c_{0,2}\Vert_{\rho,-1}.
	\label{continuous}
	\end{align}
\end{theorem}
\begin{remark}
    {\color{black} The assumption $\vert(c_0)_{\Gamma_0}\vert<1$ is necessary and standard when dealing with Cahn-Hilliard equations, since it simply excludes that the initial datum is a pure phase, i.e, $c_0\equiv 1$ or $c_0\equiv -1$. Indeed, in these cases, the phase separation phenomenon would not take place, due to the presence of one single substance.}
\end{remark}
The proof of existence follows from the uniform estimates obtained in the proof of Theorem \ref{theo} below. For completeness, we now include a proof of continuous dependence on the initial data, entailing uniqueness, which was omitted in \cite{CaeEll21}. 

\begin{proof}
(Stability of weak solutions to \eqref{phi}-\eqref{ini})
The density $\rho$ is determined by a single ordinary differential equation which can be solved explicitly, and it is thus unique. Now suppose $(c_1, w_1)$ and $(c_2,w_2)$ both solve \eqref{phi}, \eqref{mu6}, with initial conditions $c_1(0)=c_{0,1}$ and $c_2(0)=c_{0,2}$, such that $(c_{0,1})_{\Gamma_0}=(c_{0,2})_{\Gamma_0}$. Denote $\xi^c = c_1 - c_2$, $\xi^w = w_1 - w_2$. Subtracting the corresponding equations leads to
\begin{align}
{\color{black}m_\star( \rho \md \xi^c, \eta)} + \gint \rho \tgrad \xi^w \cdot \tgrad \eta &= 0\label{a1a} \\
\label{a1b}\gint \tgrad \xi^c \cdot \tgrad \eta + \gint \rho ( F'(c_1) - F'(c_2) ) \eta &= \gint \rho \xi^w \eta.
\end{align}
Using the ODE \eqref{ODE} for $\rho$ and \eqref{phi}, noting that $\xi^c(0)=c_{0,1}-c_{0,2}$, we obtain 
\begin{align*}
\dfrac{d}{dt} \gint \rho \xi^c = {\color{black}m_\star( \rho \md \xi^c, 1 )} = 0 \quad \Longrightarrow \quad \gint \rho \xi^c =\int_{\Gamma_0}c_{0,1}-c_{0,2} \ \equiv0.
\end{align*}
Therefore the weighted inverse Laplacian $\invlaprho \xi^c$ is well defined, and taking $\eta = \invlaprho \xi^c$ in \eqref{a1a} we get
\begin{align}
&{\color{black}m_\star( \rho \md \xi^c, \invlaprho \xi^c)} + \gint \rho \tgrad \xi^w \cdot \tgrad \invlaprho \xi^c=0.
\label{eq:1} 
\end{align} 
Now note:
\begin{align*}
{\color{black}m_\star( \rho \md \xi^c, \invlaprho \xi^c)}& = \dfrac{d}{dt} \left[ \gint \rho \xi^c \invlaprho \xi^c \right] - \gint \rho \xi^c \md \invlaprho \xi^c \\
&=\dfrac{d}{dt} \left[ \gint \rho |\tgrad \invlaprho\xi^c|^2 \right] - \gint \rho \tgrad \invlaprho \xi^c \cdot \tgrad \md \invlaprho \xi^c \\
&=\dfrac{d}{dt} \|\xi^c\|_{\rho, -1}^2  - \gint \rho \tgrad \invlaprho \xi^c \cdot \tgrad \md \invlaprho \xi^c
\end{align*} 
and 
\begin{align*}
\gint \rho \tgrad \xi^w \cdot \tgrad \invlaprho \xi^c =\gint \rho \xi^w \xi^c
\end{align*}
so that \eqref{eq:1} becomes
\begin{align}\label{eq:2}
\dfrac{d}{dt} \|\xi^c\|_{\rho, -1}^2 + \gint \rho \xi^w \xi^c&= \gint \rho \tgrad \invlaprho \xi^c \cdot \tgrad \md \invlaprho \xi^c.
\end{align}
We now also test \eqref{a1b} with $\eta=\xi^c$ to obtain
\begin{align}\label{eq:3}
\|\tgrad \xi^c\|^2 + \gint \rho (F'(c_1)-F'(c_2)) \xi^c = \gint \rho \xi^w \xi^c.
\end{align}
Due to the structure of the logarithmic potential $F$, being $F_{ln}$ strictly convex, we have the estimate
\begin{align*}
\gint \rho (F'(c_1)-F'(c_2)) \xi^c \geq - C \| {\color{black}\sqrt{\rho}}\xi^c \|{\color{black}^2}
\end{align*}
from where \eqref{eq:3} becomes
\begin{align}\label{eq:4}
\|\tgrad \xi^c\|^2  \leq \gint \rho \xi^w \xi^c + C \|{\color{black}\sqrt{\rho}}\xi^c\|^2.
\end{align}
Adding \eqref{eq:2} and \eqref{eq:4} the terms involving the product $\xi^w \xi^c$ cancel out and we are led to
\begin{align}\label{eq:5}
\dfrac{d}{dt} \|\xi^c\|_{\rho, -1}^2 + \|\tgrad \xi^c \|^2 &\leq C \|{\color{black}\sqrt{\rho}}\xi^c\|^2 + \gint \rho \tgrad \invlaprho \xi^c \cdot \tgrad \md \invlaprho \xi^c.
\end{align}
We now estimate the first term on the right-hand side as
\begin{align*}
C \|{\color{black}\sqrt{\rho}}\xi^c\|^2 =C \gint \rho |\xi^c|^2 &= C \gint \rho \tgrad \xi^c \cdot \tgrad \invlaprho \xi^c\leq \dfrac{1}{4} \|\tgrad \xi^c\|^2 + C \| \xi^c \|_{\rho, -1}^2.
\end{align*}
For the second term, we note that
\begin{align*}
\gint \rho \tgrad \invlaprho \xi^c \cdot \tgrad \md \invlaprho \xi^c &\leq \dfrac{1}{2}\dfrac{d}{dt} \gint \rho |\tgrad \invlaprho \xi^c|^2 + C \gint \rho |\tgrad \invlaprho \xi^c|^2\\ &\leq \dfrac{1}{2}\dfrac{d}{dt} \|\xi^c\|_{\rho, -1}^2 + C \|\xi^c\|_{\rho, -1}^2.
\end{align*}
All in all, we obtain from \eqref{eq:5} the estimate
\begin{align*}
\dfrac{d}{dt}\|\xi^c\|^2_{\rho, -1} + \|\tgrad \xi^c\|^2 \leq C \|\xi^c\|^2_{\rho, -1},
\end{align*}
and an application of Gronwall's Lemma implies the continuous dependence estimate stated in \eqref{continuous}. Then uniqueness follows by setting $c_{0,1}\equiv c_{0,2}$.
\end{proof}

\begin{remark}
	Notice that the regularity stated in Theorem \ref{existence2} can be slightly improved as in the case of the first model (see Remark \ref{LfourHtworeg}). In particular, since $c\in L^2_{H^2}$ solves the problem, for almost all $t\in [0,T]$,
	$$
	-\frac{1}{\rho}\Delta_{\Gamma}c(t)=w(t)-F^{\prime}(c(t))\in L^2({\Gamma(t)}),
	$$
	we are allowed to multiply by $-\Delta_{\Gamma}c\in L^2(\Gamma(t))$ for almost any $t\in[0,T]$. Recalling that $\varphi^\prime>0$, after an integration by parts, being $\Gamma(t)$ closed and $c\in L^\infty_{H^1}$, we obtain, by \eqref{bounds},
	\begin{align*}
	&C\norm{\Delta_{\Gamma}c}^2\leq m(\frac{1}{\rho}\Delta_{\Gamma}c,\Delta_{\Gamma} c)+\frac{\theta}{2}m(\varphi^\prime(c),\vert\tgrad c\vert^2)\leq \norm{\tgrad  w}\norm{\tgrad  c}+\Vert\tgrad c\Vert^2\leq C(1+\norm{\tgrad  w}),
	\end{align*}
	and knowing that $w\in L^2_{H^1}$, we infer $c\in L^4_{H^2}$.
\end{remark}

We make use of the following weighted $L^2$ and $H^1$ products, whose induced norms are  equivalent norms on $L^2(\Gamma(t))$ and $H^1(\Gamma(t))$, respectively:
\begin{align}
(f,g)_\rho:=m(\rho f, g), \qquad (f,g)_{1,\rho}:=(f,g)_\rho+\int_{\Gamma(t)}\rho\tgrad f \cdot \tgrad g.
\label{ind}
\end{align}
Notice that for $t=0$ these definitions coincide with the natural norms on these spaces, being $\rho(0,x)\equiv 1$ for any $x\in \Gamma_0$.
We now give an extension of Proposition \ref{prop:transport}, whose proof is shown in Appendix \ref{app:geom}:
\begin{prop}
\label{prop2}
	For $\eta,\phi\in H^1_{L^2}$, with $\tgrad \partial^\bullet \phi\in L^2_{L^2}$ and $\tgrad \partial^\bullet \eta\in L^2_{L^2}$, the following identity holds
	\begin{align}
	\nonumber\frac{d}{dt}(\tgrad \eta,\tgrad \phi)_\rho&=m(\partial^\bullet\rho\tgrad \eta,\tgrad \phi)+(\tgrad \partial^	\bullet\eta,\tgrad \phi)_\rho+(\tgrad \eta,\tgrad \partial^	\bullet\phi)_\rho+\int_{\Gamma(t)}\rho B(\mathbf{V})\tgrad \eta\cdot \tgrad \phi\\
	&=(\tgrad \partial^	\bullet\eta,\tgrad \phi)_\rho+(\tgrad \eta,\tgrad \partial^	\bullet\phi)_\rho+\int_{\Gamma(t)}\rho \widetilde{B}(\mathbf{V})\tgrad \eta\cdot \tgrad \phi,
	\label{der1}
	\end{align}
	for almost any $t\in [0,T]$, where $B$ is the vector field given in Section \ref{weakk} and $\widetilde{B}(\mathbf{V}):=-2\mathbf{D}(\mathbf{V})$.
\end{prop}
In conclusion, being $\partial^\bullet\rho=-\rho\tgrad \cdot\mathbf{V}$, we infer by  Proposition \ref{prop:transport}
\begin{align}
\frac{d}{dt}(\eta,\phi)_{\rho}=(\partial^\bullet\eta,\phi)_\rho+(\eta,\partial^\bullet\phi)_\rho, \quad \forall\,\eta,\phi\in H^1_{L^2},
\label{byparts}
\end{align}
retrieving the classical integration by parts formula as in the case of a fixed manifold. This is not surprising, in the sense that the term $\rho$ accounts for local stretching or compressing of the surfaces and somehow annihilates the effects of the evolving surface.
\subsection{Regularisation and strict separation property}
{\color{black}For this model we only need to assume $\mathbf{A_\Phi}$, without any extra regularity hypothesis.}
Our main result is the following 
\begin{theorem}
	\label{theo}
	{\color{black}Let the assumptions of Theorem \ref{existence2} hold.} Denote by $(c,w)$ the (unique) weak solution to \eqref{phi}-\eqref{ini}.
	\begin{itemize}
	    \item[(i)] There exists a constant $C=C(T, E_{CH}^\rho(u_0))>0$ such that, for almost any $t\in[0,T]$,
	\begin{align}
	t\Vert w\Vert_{H^1(\Gamma(t))}^2+\int_0^t s\Vert\partial^\bullet c\Vert^2_{H^1(\Gamma(s))}ds\leq C(T, E_{CH}^\rho(u_0)).
	\label{mu5bis}
	\end{align}
	\item[(ii)] For any $0<\tau\leq T$, there exist constants $C=C(T,\tau, E_{CH}^\rho(u_0))>0$ and $C_p=C_p(T,\tau,p, E_{CH}^\rho(u_0))>0$ such that, for almost any $t\in[\tau,T]$
	\begin{align}
	\Vert w\Vert_{H^1(\Gamma(t))}\leq C(T,\tau, E_{CH}^\rho(u_0)),
	\label{mu4bis}
	\end{align}
	$$
	\Vert\varphi(c)\Vert_{L^p(\Gamma(t))}+\Vert\varphi^\prime(c)\Vert_{L^p(\Gamma(t))}\leq C_p(T,\tau,p, E_{CH}^\rho(u_0)),\quad \forall p\in[2,\infty),
	$$
		$$
	\Vert c\Vert_{H^2(\Gamma(t))}\leq C(T,\tau, E_{CH}^\rho(u_0)).
	$$
	\item[(iii)] There exists $\xi=\xi(T,\tau, E_{CH}^\rho(u_0))>0$, such that
	\begin{align*}
	\Vert c\Vert_{L^\infty(\Gamma(t))}\leq 1-{\xi}, \quad \text{ for almost any }t\in[\tau,T].
	\end{align*}
	\end{itemize}
\end{theorem}
%\begin{remark}
%Assumption A.1) is necessary due to the approximation technique. Indeed, we need to argue by elliptic regularity to obtain that the approximated chemical potential belongs to $H^3(\Gamma(t))$.
%\end{remark}
\begin{remark}
{\color{black}For this second model the simple assumption  \textbf{A}$_\Phi$ is enough. Indeed, due to the presence of $\rho$ we do not need, e.g., \eqref{es2}. For this reason, it is our belief that the second model seems more natural in the context of evolving surfaces.}
\label{remm}
\end{remark}
\begin{remark}
	As in Remark \ref{sepp}, assuming the regularity stated in Lemma \ref{C1}, also in this case we actually obtain that, for any $\tau>0$, $c\in C^0_{C^0}(\tau,T)$ and thus		\begin{align*}
	\sup_{t\in[\tau,T]}\Vert c \Vert_{C^0(\Gamma(t))}\leq 1-{\xi}.
	\end{align*}
\end{remark}

\subsubsection{Galerkin approximation}
We need a slight revision of the Galerkin approximation scheme. Let us recall \cite[Sec.4.1]{CaeEll21} and consider a basis $\{\chi_j^0: j\in \N\}$ orthonormal in $L^2(\Gamma_0)$ and orthogonal in $H^1(\Gamma_0)$ consisting of smooth functions such that $\chi_1^0$ is constant (for example consider the eigenfunctions of the Laplace-Beltrami operator). We then transport this basis using the flow map. This gives $\{\chi_j^t:=\phi_t(\chi_j^0): j\in \N\}\subset H^1(\Gamma(t))$. Observe that this basis is still an orthonormal basis in $L^2(\Gamma(t))$ endowed with the norm induced by \eqref{ind}. Indeed, thanks to \eqref{rho}, we have
$$
(\chi_i^t,\chi_j^t)_\rho=m(\rho\chi_i^t,\chi_j^t)=\int_{\Gamma_0}\chi_i^0\chi_j^0=\delta_{ij}, \quad \forall\, i,j\in \N.
$$
Notice that the same basis is such that $\overline{\text{span}\{\chi_j^t: j\in \N\}}^{H^1(\Gamma(t))}=H^1(\Gamma(t))$, but we are not able to show any orthogonality property in this space, even endowing the space with the norm induced by \eqref{ind}. We can then introduce the finite dimensional spaces
$$
V_M(t)=\text{span}\{\chi_j^t: 1\leq j\leq M\}\subset H^1(\Gamma(t)),
$$
and, exploiting the orthogonality with respect to the equivalent $L^2$ inner product \eqref{ind}$_1$, we can give an explicit expression to the $L^2$-orthogonal projector $P_M^\rho(t): L^2(\Gamma(t))\to V_M(t)$ as
$$
P_M^\rho(t) f:=\sum_{j=1}^M(f,\chi_j^t)_\rho\chi_j^t=\sum_{j=1}^M\left(\int_{\Gamma_0}\widetilde{f}\chi_j^0\right)\chi_j^t=\phi_t(P_M^0\widetilde{f}),
$$
where $P_M^0=P^\rho_M(0)$ is the orthogonal projector over $V_M(0)$ (endowed with the canonical $L^2$-norm).
Clearly this implies $\Vert P_M^\rho(t) f\Vert_{\rho}\leq \Vert f\Vert_{\rho}$ uniformly in $M$, for any $t>0$.
Following \cite[Sec.7]{AlpCaeDjuEll21}, we now introduce the matrix
$$
\mathbf{A}_t^0:=(D\Phi_t^0)^TD\Phi_t^0+\nu_0\otimes\nu_0,
$$
where $\nu_0$ is the normal to $\Gamma_0$. This matrix is invertible in $\R^{3}$ thanks to the extension in the normal direction. We then have that (see \cite[(7.5)]{AlpCaeDjuEll21})
\begin{align}
\phi_{-t}(\tgrad h)=D\Phi_t^0(\mathbf{A}_t^0)^{-1}\nabla_{\Gamma_0}(\phi_{-t}(h))=D\Phi_t^0(\mathbf{A}_t^0)^{-1}\nabla_{\Gamma_0}\widetilde{h}.
\label{pullback}
\end{align}
{\color{black}Note also that, for any $t\in[0,T]$, 
\begin{align}
\Vert D\Phi_t^0(\mathbf{A}_t^0)^{-1}\Vert_{C^0(\Gamma(t))}\leq C,
    \label{phiA}
\end{align}
being, thanks to assumption \textbf{A}$_\Phi$, $\Phi_{(\cdot)}^0\in C^1([0,T];C^2(\R^3;\R^3))$.}
 Therefore, again by \eqref{rho}, the properties of the orthogonal basis $\{\chi_j^0: j\in \N\}$ in $H^1(\Gamma_0)$, and the compatibility of the space $(H^1(\Gamma(t)),\phi_t)_{t\in[0,T]}$,
\begin{align}
\nonumber\int_{\Gamma(t)}\rho\tgrad  w \cdot \tgrad P_M^\rho(t)f&=\int_{\Gamma_0}\nabla_{\Gamma_0} \widetilde{w}(D\Phi_t^0)^{T}(\mathbf{A}_t^0)^{-T}D\Phi_t^0(\mathbf{A}_t^0)^{-1}\nabla_{\Gamma_0}P_M^0\widetilde{f}\\
\nonumber
&\leq C\Vert\nabla_{\Gamma_0} \widetilde{w}\Vert_{L^2(\Gamma_0)}\Vert\nabla_{\Gamma_0}P_M^0\widetilde{f}\Vert_{L^2(\Gamma_0)}\\
\nonumber
&\leq C\Vert \nabla_{\Gamma_0} \widetilde{w}\Vert_{L^2(\Gamma_0)}\Vert \nabla_{\Gamma_0}\widetilde{f}\Vert_{L^2(\Gamma_0)}\\
&\leq C\Vert\tgrad  {w}\Vert\Vert\tgrad  {f}\Vert,
\label{grad}
\end{align}
%NB ok perchè a(u,v)=lambda v e a() è proprio il grad, quindi la norma sara sempre minore o uguale-> quindi anche rispetto ad a() la base è ortogonale
for any $w,f\in H^1(\Gamma(t))$. 
We can also prove that the time derivative $\partial^\bullet$ commutes with the projector $P_M^\rho(t)$. Indeed, since $\partial^\bullet\chi_j^t\equiv 0$ for any $j\in\N$, we have
\begin{align}\label{dt}
\begin{split}
\partial^\bullet(P_M^\rho(t)f)&=\sum_{i=1}^M\partial^\bullet (f,\chi_i^t)_\rho\chi_i^t\\&=\sum_{i=1}^M\left((\partial^\bullet f,\chi_i^t)_\rho+m(f\partial^\bullet\rho,\chi_i^t)+g(\rho f,\chi_i^t)\right)\chi_i^t\\&=\sum_{i=1}^M(\partial^\bullet f,\chi_i^t)_\rho\chi_i^t\\
&=P_M^\rho(t)\partial^\bullet f, \quad \forall f\in H^1_{L^2},
\end{split}
\end{align}
where we exploited the fact that $\partial^\bullet\rho=-\rho\tgrad \cdot \mathbf{V}$.

We then consider the Galerkin approximation with the approximated potential $F^\delta$ ($\varphi^\delta=(F_{ln}^\delta)^\prime$) of the original problem. More precisely, given the spaces $V_M$ as in {\color{black}\eqref{VM}}, for each $M\in \N$, find functions $c_\delta^M,w_\delta^M\in L^2_{V_M}$  with $\partial^\bullet c_\delta^M\in L^2_{V_M}$ such that, for any $\eta\in L^2_{V_M}$ and all $t\in[0,T]$,
\begin{align}
&(\partial^\bullet c_\delta^M, \eta)_\rho+(\tgrad w_\delta^M,\tgrad \eta)_\rho=0,\label{CH2bis}\\
&w_\delta^M=P_M^\rho(t)(-\frac{1}{\rho}\Delta_{\Gamma(t)} c_\delta^M+(F^\delta)^\prime(c_\delta^M)),\label{mu2bis}\\
&c_\delta^M(0)=P_M^0c_{0}, \quad\text{ a.e. on } \Gamma_0. \label{inigal}
\end{align}
Notice that this formulation, written with respect to the $\rho$ inner product, is very similar the classical weak formulation of CH equation in bounded domains of $\R^2$ (or $\R^3$).
For this problem we then have
\begin{prop}
	{\color{black}Let Assumption $\textbf{A}_\Phi$ hold.} Then there exists a unique local solution $(c_\delta^M,w_\delta^M)$ to \eqref{CH2bis}-\eqref{inigal}. In particular there exist functions $(c^M,w^M)$ satisfying \eqref{CH2bis}-\eqref{mu2bis} on an interval $[0,t^\star)$, $0\leq t^\star\leq T$, together with \eqref{inigal}. The functions are of the form (omitting for simplicity the dependence on $\delta$)
	$$
	c_\delta^M(t)=\sum_{i=1}^M c_i^M(t)\chi_i^t,\qquad w_\delta^M(t)=\sum_{i=1}^M w_i^M(t)\chi_i^t, \qquad t\in[0,t^\star),
	$$
	with $c_i^M\in C^2([0,t^\star))$ and $w_i^M\in C^2([0,t^\star))$, for every $i\in\{1,\ldots,M\}$.
	\label{prop_bis}
\end{prop}
\begin{proof}
    We consider the matrix form of the equations, where, as before, we set $\mathbf{c}^M(t)=(c_1^M(t),\ldots,c_M^M(t))$ and $\mathbf{w}^M(t)=(w_1^M(t),\ldots,w_M^M(t))$,  \begin{align*}
	&M^\rho(t)\dot{\mathbf{c}}^M(t)+A_S^\rho(t)\mathbf{w}^M(t)=0,\\&
	A_S(t)\mathbf{c}^M(t)+(\mathbf{F}^\delta_\rho)^\prime(\mathbf{{c}}^M(t))-M^\rho(t)\mathbf{w}^M(t)=0.
	\end{align*}
	Here
	\begin{align*}
	&(M^\rho(t))_{ij}=(\chi_i^t,\chi_j^t)_\rho=\delta_{ij}, \\&
	(A_S(t))_{ij}=a_S(t;\chi^t_i,\chi^t_j),\\& (A_S^\rho(t))_{ij}=(\tgrad \chi^t_i,\tgrad \chi^t_j)_\rho,\\&
	(\mathbf{F}^\delta_\rho)^\prime(\mathbf{c}^M(t))_j=((F^\delta)^\prime(c^M(t)),\chi_j^t)_\rho.
	\end{align*}
	We now observe that again these matrices are more regular. Indeed,  $M^\rho_{ij}$ is actually the identity matrix.
	Similarly, we get
	$$
	\frac{d}{dt}(A_S(t))_{ij}=b(\chi^t_i,\chi^t_j)\in C^0([0,T]),
	$$
	and the same goes for $A_S^\rho$, by \eqref{der1}.
	Recalling then that $(F^\delta)^\prime$ is $C^{1,1}(\R)$, the result follows from the general theory of ODEs.
\end{proof}

We establish some a priori estimates for the solutions of the Galerkin approximation.

\begin{prop}
For the approximating solution pair $(u_\delta^M, w_\delta^M)$ we have: 
\begin{align}\label{en}
\sup_{t\in[0,T]}\tilde{E}_{CH}^\rho(c_\delta^M)+\frac{1}{2}\int_0^T\int_{\Gamma(t)}\vert\tgrad  w_\delta^M\vert^2dt\leq C(T) \\
    \|\partial^\bullet c_\delta^M\|_{L^2_{H^{-1}}} \leq C(T), \quad \text{ and } \quad \|w_\delta^M\|_{L^2_{H^1}} \leq C_\delta,
\end{align}
where as before $\tilde{E}_{CH}^\rho:=E^\rho_{CH}+\tilde C$ for some $\tilde C>0$ chosen so that $\tilde{E}_{CH}^\rho\geq 0$.
\end{prop}

\begin{proof}
We consider the Galerkin approximation of Proposition \ref{prop_bis}. First notice that we have the conservation of the total mass, by choosing $\eta\equiv 1$:
\begin{align}
\frac{d}{dt}\int_{\Gamma(t)}\rho c_\delta^M =\int_{\Gamma(t)}\rho \partial^\bullet c_\delta^M\equiv 0.
\label{mass3}
\end{align}
Arguing as in \cite[(6.2.1)]{CaeEll21} we can obtain
\begin{align}
\frac{d}{dt}E_{CH}^\rho(c_\delta^M)+\int_{\Gamma(t)}\rho\vert\tgrad  w_\delta^M\vert^2dt=b(c_\delta^M,c_\delta^M).
\label{en1}
\end{align}
Observe also that, by the conservation of total mass, {\color{black} see \eqref{mass2} and \eqref{mass3}}, $$(\widetilde{c}_\delta^M)_{\Gamma_0}\equiv \dfrac{\int_{\Gamma(t)}\rho c_\delta^M}{\vert\Gamma_0\vert}\equiv (c_0)_{\Gamma_0}.$$ 
{\color{black}In the end we easily get} 
\begin{align}
\frac{d}{dt}\tilde{E}^\rho_{CH}(c_\delta^M)+\frac{1}{2}\int_{\Gamma(t)}\rho\vert\tgrad  w_\delta^M\vert^2dt\leq C\tilde{E}^\rho_{CH}(c_\delta^M),
\label{en3}
\end{align}
where $\tilde{E}^\rho_{CH}(c_\delta^M):={E}^\rho_{CH}(c_\delta^M)+C$, so that $\tilde{E}^\rho_{CH}(c_\delta^M)\geq0$. Therefore, an application of Gronwall's Lemma gives
\begin{align}
\sup_{t\in[0,T]}\tilde{E}_{CH}^\rho(c_\delta^M)+\frac{1}{2}\int_0^T\int_{\Gamma(t)}\vert\tgrad  w_\delta^M\vert^2dt\leq C(T)\tilde{E}^{\rho}_{CH}(P_M^0c_{0}),
\label{enalpha}
\end{align}	
and, recalling the properties of $P_M^0$ (see, e.g., \cite[Lemma 5.6]{CaeEll21}), we have
\begin{align}
\tilde{E}^{\rho}_{CH}(P_M^0c_{0})\leq CT.
\label{ena}
\end{align}
This clearly allows us to extend the maximal time from $t^\star$ to $T$.
Concerning the estimate of the mean value of $c$, since in this model this quantity is not conserved, being conserved the product $\rho c$, we can observe, as in \cite{CaeEll21}, that
\begin{align*}
\frac{d}{dt}\int_{\Gamma(t)}c_\delta^M=\int_{\Gamma(t)}\partial^\bullet c_\delta^M+\int_{\Gamma(t)}c_\delta^M\tgrad \cdot \mathbf{V}.
\end{align*}
If we now take $\eta=P_M^\rho(t)\frac{1}{\rho}$ in \eqref{CH2bis}, we have, being $\partial^\bullet c^M_\delta\in L^2_{V_M}$ and $P_M^\rho(t)$ self-adjoint,
\begin{align*}
&\left(\partial^\bullet c_\delta^M, P_M^\rho(t)\frac{1}{\rho}\right)_\rho=m(\partial^\bullet c_\delta^M, 1)=-\left(\tgrad w_\delta^M,\tgrad P_M^\rho(t)\frac{1}{\rho}\right)_\rho.
\end{align*}
Therefore, integrating by parts $\int_{\Gamma(t)}c_\delta^M\tgrad \cdot \mathbf{V}$, recalling that $\Gamma(t)$ is closed, we get
\begin{align*}
\frac{d}{dt}\int_{\Gamma(t)}c_\delta^M=-\left(\tgrad w_\delta^M,\tgrad P_M^\rho(t)\frac{1}{\rho}\right)_\rho-\int_{\Gamma(t)}\tgrad c_\delta^M\cdot \mathbf{V},
\end{align*}
so that
\begin{align*}
&\int_{\Gamma(t)}c_\delta^M=\int_{\Gamma_0}P^0_Mc_0-\int_0^t\left(\nabla_{\Gamma(s)}w_\delta^M,\nabla_{\Gamma(s)}P_M^\rho(s)\frac{1}{\rho}\right)_\rho ds-\int_0^t\int_{\Gamma(s)}\tgrad c_\delta^M\cdot \mathbf{V} ds.
\end{align*}
Thus, by \eqref{bounds}, {\color{black}\eqref{ena}}, the properties of $P_M^0$, and \eqref{grad}, we deduce
\begin{align*}
&\left\vert\int_{\Gamma(t)}c_\delta^M\right\vert\leq C+C\int_0^t\left\Vert\nabla_{\Gamma(s)}w^M_\delta\right\Vert\left\Vert\nabla_{\Gamma(s)}\frac{1}{\rho}\right\Vert ds+C(T)\\&\leq C\left(1+\int_0^t\Vert \nabla_{\Gamma}w_\delta^M\Vert ds\right)+C(T)\leq C(T).
\end{align*}
Combining this result with \eqref{enalpha} and Poincaré's inequality, we infer that
\begin{align}
\Vert c^M_\delta\Vert_{L^\infty_{H^1}}\leq C(T),
\label{en2}
\end{align}
independently of $M$ and $\delta$.
Observe now that, by \eqref{grad}, we can find a uniform estimate for $\partial^\bullet c^M_\delta$. Indeed we can write, for any $\eta \in L^2_{H^1}$, being $P^\rho_M(t)$ self-adjoint (with respect to \eqref{ind}$_1$) and by \eqref{grad}{ \color{black} and  \eqref{en2},}
\begin{align*}
{\color{black}m_\star(\partial^\bullet c^M_\delta, \eta)}&=(P_M^\rho(t)\partial^\bullet c^M_\delta,\eta)=\left(P_M^\rho(t)\partial^\bullet c^M_\delta,\frac{1}{\rho}\eta\right)_\rho\\&=\left(\partial^\bullet c^M_\delta,P_M^\rho(t){\color{black}\left(\frac{1}{\rho}\eta\right)}\right)_\rho\\&=-\left(\tgrad w_\delta^M,\tgrad P_M^\rho(t){\color{black}\left(\frac{1}{\rho}\eta\right)}\right)_\rho\\&\leq C\Vert\tgrad w_\delta^M\Vert\left\Vert\tgrad \frac{\eta}{\rho}\right\Vert+C\Vert c_\delta^M\Vert\left\Vert\tgrad \frac{\eta}{\rho}\right\Vert\\&\leq C{\color{black}\left(\Vert\tgrad w_\delta^M\Vert+1\right)\Vert\eta\Vert_{H^1(\Gamma(t))}},
\end{align*}
which implies, by \eqref{enalpha},
\begin{align}
\Vert\partial^\bullet c^M_\delta\Vert_{L^2_{H^{-1}}}\leq C(T),
\label{h-1}
\end{align}
independently of $M$ and $\delta$. We now need a control over the mean value of $w_\delta^M$. This can be obtained by testing \eqref{mu2bis} with $\eta=P_M^\rho(t)\frac{1}{\rho}$ being $w\in L^2_{V_M}$. On account of \eqref{grad}, we have
\begin{align*}
\left(w_\delta^M,\frac{1}{\rho}\right)_\rho=\left(w_\delta^M,P_M^\rho(t)\frac{1}{\rho}\right)_\rho&=\left(-\frac{1}{\rho}\Delta c_\delta^M, P_M^\rho(t)\frac{1}{\rho}\right)_\rho+\left(F^\prime_\delta(c_\delta^M),P_M^\rho(t)\frac{1}{\rho}\right)_\rho\\&=m\left(\tgrad  c_\delta^M, \tgrad P_M^\rho(t)\frac{1}{\rho}\right)+\left(F^\prime_\delta(c_\delta^M),P_M^\rho(t)\frac{1}{\rho}\right)_\rho\\&\leq C\Vert\tgrad c_\delta^M\Vert\left\Vert\tgrad \left(P^\rho_M(t)\frac{1}{\rho}\right)\right\Vert+C\Vert F^\prime_\delta(c_\delta^M)\Vert\left\Vert P_M^\rho(t)\frac{1}{\rho}\right\Vert\\&\leq C\Vert\tgrad c_\delta^M\Vert\left\Vert\tgrad \frac{1}{\rho}\right\Vert+C_\delta\Vert c_\delta^M\Vert\leq C_\delta,
\end{align*}
by \eqref{bounds} and \eqref{enalpha}. Here we have also applied  the fact that $\vert F^\prime_\delta(c_\delta^M)\vert\leq C_\delta\vert c_\delta^M\vert$, being ${\color{black}\vert F^{\prime\prime}_\delta\vert}\leq C_\delta$ and $F_\delta^\prime(0)=F^\prime(0)=0$. In the Galerkin scheme we are not able to retrieve a uniform-in-$\delta$ estimate for the mean value of $w_\delta^M$. Indeed, we should be able to control the $L^\infty(\Gamma(t))$ norm of $\rho P_M^\rho(t)\frac{1}{\rho}$ and then control $\int_{\Gamma(t)}\vert F^\prime_\delta(c_\delta^M)\vert$, but this does not seem feasible. Therefore, we will need to pass to the limit in $M$ first. Then this control will be obtained independently of $\delta$.
The above bound entails, thanks to Poincaré's inequality combined with \eqref{en},
\begin{align}
\Vert w_\delta^M\Vert_{L^2_{H^1}}\leq C_\delta.
\label{ww}
\end{align}
\end{proof}

\subsubsection{Proof of Theorem \ref{theo}}
\label{sep2}
\underline{Part (i).} We now need to find higher-order estimates. In particular, we set $\eta=\partial^\bullet w_\delta^M\in L^2_{V_M}$ in \eqref{CH2bis}, to get
\begin{align}
(\partial^\bullet c_M^\delta, \partial^\bullet w^M_\delta)_\rho+(\tgrad w_M^\delta,\tgrad \partial^\bullet w_M^\delta)_\rho=0.
\label{E1}
\end{align}
Observe that, by \eqref{der1},
\begin{align}
\frac{1}{2}\frac{d}{dt}(\tgrad w_\delta^M,\tgrad w_\delta^M)_\rho=(\tgrad w_M^\delta,\tgrad \partial^\bullet w_M^\delta)_\rho+\int_{\Gamma(t)}\rho\widetilde{B}(\mathbf{V})\tgrad w_\delta^M\cdot \tgrad w_\delta^M.
\label{E2}
\end{align}
On the other hand, by Proposition \ref{prop:transport}, we have, for any $\eta\in L^2_{V_M}$ such that $\partial^\bullet\eta\in L^2_{V_M}$,
\begin{align*}
\frac{d}{dt}m(\tgrad c_\delta^M,\tgrad \eta)&=m(\tgrad \partial^\bullet c_\delta^M,\tgrad \eta)+m(\tgrad  c_\delta^M,\tgrad \partial^\bullet\eta)+\int_{\Gamma(t)}{B}(\mathbf{V})\tgrad c_\delta^M\cdot\tgrad \eta
\end{align*}
and, by \eqref{byparts},
\begin{align*}
\frac{d}{dt}(w_\delta^M,\eta)_\rho=(\partial^\bullet w_\delta^M,\eta)_\rho+(w,\partial^\bullet\eta)_\rho.
\end{align*}
Furthermore, we have, using again {\color{black}\eqref{byparts}},
\begin{align*}
\frac{d}{dt}(F^\prime_\delta(c_\delta^M),\eta)_\rho=(F^{\prime\prime}_\delta(c_\delta^M),\partial^\bullet c_\delta^M\eta)_\rho+(F^\prime_\delta(c_\delta^M),\partial^\bullet \eta)_\rho.
\end{align*}
We now recall that 
\begin{align*}
\frac{d}{dt}(w_\delta^M,\eta)_\rho=\frac{d}{dt}m(\tgrad c_\delta^M,\tgrad \eta)+\frac{d}{dt}(F^\prime_\delta(c_\delta^M),\eta)_\rho,
\end{align*}
thus
\begin{align*}
(\partial^\bullet w_\delta^M,\eta)_\rho+(w,\partial^\bullet\eta)_\rho&=m(\tgrad \partial^\bullet c_\delta^M,\tgrad \eta)+m(\tgrad  c_\delta^M,\tgrad \partial^\bullet\eta)\\&+\int_{\Gamma(t)}{B}(\mathbf{V})\tgrad c_\delta^M\cdot\tgrad \eta\\&+(F^{\prime\prime}_\delta(c_\delta^M),\partial^\bullet c_\delta^M\eta)_\rho+(F^\prime_\delta(c_\delta^M),\partial^\bullet \eta)_\rho.
\end{align*}
By noticing that, being $\partial^\bullet\eta\in L^2_{V_M}$,
$$
(w,\partial^\bullet\eta)_\rho=m(\tgrad  c_\delta^M,\tgrad \partial^\bullet\eta)+(F^\prime_\delta(c_\delta^M),\partial^\bullet \eta)_\rho,
$$
we infer, choosing $\eta=\partial^\bullet c_\delta^M$ (which is sufficiently regular thanks to Proposition \ref{prop_bis}),
\begin{align*}
(\partial^\bullet c_\delta^M,\partial^\bullet w_\delta^M)_\rho=(F^{\prime\prime}_\delta(c_\delta^M)\partial^\bullet c_\delta^M,\partial^\bullet c_\delta^M)_\rho+\Vert\tgrad \partial^\bullet c_\delta^M\Vert^2+ b(\mathbf V; c_\delta^M, \md c_\delta^M) 
\end{align*}
From these results, together with \eqref{E1} and \eqref{E2}, we eventually get
\begin{align*}
&\frac{1}{2}\frac{d}{dt}(\tgrad w_\delta^M,\tgrad w_\delta^M)_\rho+(F^{\prime\prime}_\delta(c_\delta^M)\partial^\bullet c_\delta^M,\partial^\bullet c_\delta^M)_\rho+\Vert\tgrad \partial^\bullet c_\delta^M\Vert^2\\&=-\int_{\Gamma(t)}{B}(\mathbf{V})\tgrad c_\delta^M\cdot\tgrad \partial^\bullet c_\delta^M+\int_{\Gamma(t)}\rho\widetilde{B}(\mathbf{V})\tgrad w_\delta^M\cdot \tgrad w_\delta^M,
\end{align*}
and, recalling the definition of {\color{black}$F_\delta^{\prime\prime}$} and $\varphi_\delta^\prime\geq 0$, we infer
\begin{align}
\nonumber\frac{1}{2}\frac{d}{dt}(\tgrad w_\delta^M,\tgrad w_\delta^M)_\rho+\Vert\tgrad \partial^\bullet c_\delta^M\Vert^2&\leq \frac{1}{2}\frac{d}{dt}(\tgrad w_\delta^M,\tgrad w_\delta^M)_\rho+\frac{\theta}{2}(\varphi_\delta^\prime(c_\delta^M),\partial^\bullet c_\delta^M\partial^\bullet c_\delta^M)_\rho+\Vert\tgrad \partial^\bullet c_\delta^M\Vert^2\\&\nonumber=(\partial^\bullet c_\delta^M,\partial^\bullet c_\delta^M)_\rho-\int_{\Gamma(t)}{B}(\mathbf{V})\tgrad c_\delta^M\cdot\tgrad \partial^\bullet c_\delta^M\\
&\hskip 2cm +\int_{\Gamma(t)}\rho\widetilde{B}(\mathbf{V})\tgrad w_\delta^M\cdot \tgrad w_\delta^M.
\label{ineq}
\end{align}
By \eqref{bounds}, \eqref{enalpha}, interpolation and standard inequalities, we then have
\begin{align*}
&(\partial^\bullet c_\delta^M,\partial^\bullet c_\delta^M)_\rho-\int_{\Gamma(t)}{B}(\mathbf{V})\tgrad c_\delta^M\cdot\tgrad \partial^\bullet c_\delta^M+\int_{\Gamma(t)}\rho\widetilde{B}(\mathbf{V})\tgrad w_\delta^M\cdot \tgrad w_\delta^M\\&\leq C_1\Vert\partial^\bullet c^M_\delta\Vert^2+C+\frac{1}{4}\Vert\tgrad \partial^\bullet c_\delta^M\Vert^2+C\Vert\tgrad  w_\delta^M\Vert^2.
\end{align*}
{\color{black}Notice that we have only exploited assumption \textbf{A}$_\Phi$ (see also Remark \ref{remm}). Here $C_1>0$ is a positive constant independent of $\delta,M$.}
We now test \eqref{CH2bis} by $\eta=\partial^\bullet c_\delta^M$, obtaining, for $\kappa$ suitably small to be chosen later on,
$$
C_2\Vert \partial^\bullet c_\delta^M\Vert^2 \leq \kappa\Vert\tgrad \partial^\bullet c_\delta^M\Vert^2+C(1+\Vert\tgrad w^M_\delta\Vert^2),
$$
{\color{black} where $C_2>0$ is independent of $\delta,M$.}
Adding this inequality, multiplied by $\omega=2\frac{C_1}{C_2}$, and \eqref{ineq} together, choosing $\kappa=\frac{C_2}{8C_1}$, and recalling \eqref{bounds}, we find
 \begin{align}
 \frac{d}{dt}\mathcal{Q}_\rho+\frac{1}{2}\Vert\tgrad \partial^\bullet c_\delta^M\Vert^2+C_1\Vert \partial^\bullet c_\delta^M\Vert^2 \leq C(1+(\tgrad w_\delta^M,\tgrad w_\delta^M)_\rho),
\label{grad1}
\end{align}
where $C$ and $C_1$ in this estimate do not depend on $\delta$ (they rely only on \eqref{enalpha}), and
$$
\mathcal{Q}_\rho:=\frac{1}{2}(\tgrad w_\delta^M,\tgrad w_\delta^M)_\rho {\color{black}\geq0.}
$$
We then multiply \eqref{grad1} by $t$. This gives
	\begin{align}
	 &\frac{d}{dt}{\color{black}\left(t\mathcal{Q}_\rho\right)}+\frac{t}{2}\Vert\tgrad \partial^\bullet c_\delta^M\Vert^2+C_1t\Vert \partial^\bullet c_\delta^M\Vert^2 \leq C(t+t\mathcal{Q}_\rho(t))+\mathcal{Q}_\rho(t).
	\label{grad2}
	\end{align}
	Then, thanks to \eqref{enalpha}, we have $\mathcal{Q}_\rho\in L^1(0,T)$, so that we can apply Gronwall's Lemma and infer
	\begin{align}
	\Vert\sqrt{t}\tgrad w_\delta^M\Vert_{L^\infty_{L^2}}+\Vert\sqrt{t} \partial^\bullet u_\delta^M\Vert_{L^2_{H^1}}\leq C(T).
	\label{high}
	\end{align}
	It is crucial to stress again that also the above constant does not depend on $\delta$.
	
Having this uniform (in $M$) regularity, we can easily pass to the limit as $M\rightarrow\infty$ and obtain, by compactness arguments, the existence of a solution $(c_\delta,w_\delta)$ such that, for any $\eta\in L^2_{H^1}$,
\begin{align}
&(\partial^\bullet c_\delta, \eta)_\rho+(\tgrad w_\delta,\tgrad \eta)_\rho=0,\label{CH2bis2}\\&
(w_\delta,\eta)_\rho=m(\tgrad  c_\delta,\nabla_{\Gamma}\eta)+((F^\delta)^\prime(c_\delta),\eta)_\rho,\label{mu2bis2}
\end{align}
and $c_\delta(0)=c_{0}$ almost everywhere in $\Gamma_0$. In particular, we have the following convergences (see \eqref{enalpha}, \eqref{h-1}, and \eqref{ww})
\begin{align*}
&c_\delta^M\overset{\ast}{\rightharpoonup} c_\delta,\quad \text{ in } L^\infty_{H^1},\\&
\partial^\bullet c_\delta^M{\rightharpoonup} \partial^\bullet c_\delta,\quad \text{ in } L^2_{H^{-1}},\\&
w_\delta^M{\rightharpoonup} w_\delta,\quad \text{ in } L^2_{H^1},
\end{align*}
which also imply, by \eqref{high},
\begin{align*}
&\sqrt{t}\partial^\bullet c_\delta^M{\rightharpoonup} \sqrt{t}\partial^\bullet c_\delta,\quad \text{ in } L^2_{H^1},\\&
\sqrt{t}\tgrad w_\delta^M\overset{\ast}{\rightharpoonup} \sqrt{t}\tgrad w_\delta,\quad \text{ in } L^\infty_{L^2}.
\end{align*}
As a consequence, from \eqref{enalpha}, \eqref{en2}, \eqref{h-1}, and \eqref{high} we can obtain the following bounds, which are still independent of $\delta$,
\begin{align}
\Vert c_\delta\Vert_{L^\infty_{H^1}}+\Vert \tgrad w_\delta\Vert_{L^2_{L^2}}+\Vert \partial^\bullet c_\delta\Vert_{L^2_{H^{-1}}}+\Vert\sqrt{t}w_\delta\Vert_{L^\infty_{L^2}}+\Vert\sqrt{t} \partial^\bullet c_\delta\Vert_{L^2_{H^1}}\leq C(T).
\label{high2}
\end{align}
We are left to find an estimate for $(w_\delta)_{\Gamma_0}$ which is independent of $\delta$. Being now $\eta=\frac{1}{\rho}\in L^2_{H^1}$, we can use it as a test function in \eqref{CH2bis2}, to get
\begin{align*}
\int_{\Gamma(t)} w_\delta=\left(w_\delta,\frac{1}{\rho}\right)_\rho=\left(\nabla_\Gamma c_\delta, \nabla_\Gamma\frac{1}{\rho}\right)+\left(F_\delta^\prime(c_\delta),\frac{1}{\rho}\right)_\rho.
\end{align*}
At this point we can repeat word by word the proof in \cite[Sec.1.6]{CaeEll21}, exploiting the fact that $(\widetilde{c})_{\Gamma_0}=(c_0)_{\Gamma_0}$ and $\vert(c_0)_{\Gamma_0}\vert<1$, and obtaining
$$
\left\vert \int_{\Gamma(t)}w_\delta\right\vert\leq C(1+\Vert\tgrad w_\delta\Vert).
$$
This result finally allows, by Poincaré's inequality, to infer
\begin{align}
\Vert w_\delta\Vert_{L^2_{H^1}}+\Vert \sqrt{t}w_\delta\Vert_{L^\infty_{H^1}}\leq C(T).
\label{final}
\end{align}
Therefore, we can pass to the limit with respect to $\delta$. Again by standard compactness arguments, we obtain a solution $(c,w)$ to \eqref{phi}-\eqref{ini}. In particular, we can retrieve, exactly as done for the first model in the proof of \cite[Thm.5.14]{CaeEll21}, the bound $\vert c\vert< 1$ almost everywhere on $\Gamma(t)$, for almost any $t\in[0,T]$. Clearly, again by sequential lower semicontinuity, we have the following bounds on the final solution (which is also unique by Theorem \ref{existence2}):
\begin{align}
\Vert c\Vert_{L^\infty_{H^1}}+\Vert w\Vert_{L^2_{H^1}}+\Vert \partial^\bullet c\Vert_{L^2_{H^{-1}}}+\Vert\sqrt{t}w\Vert_{L^\infty_{H^1}}+\Vert\sqrt{t} \partial^\bullet c\Vert_{L^2_{H^1}}\leq C(T).
\label{high3}
\end{align}

\underline{Part (ii).} Concerning the strict separation property, we use an argument similar to the one used in the proof of Theorem \ref{main1}. Let us fix $\tau>0$. We have that $\Vert w\Vert_{H^1(\Gamma(t))}\leq C$ for almost any $t\geq\tau$. Then we set $c_k=h_k(c)$, where $h_k$ is defined in \eqref{hk}. Being $c$ in $L^\infty_{H^1}$, we have
$$
\nabla_{\Gamma} c_k=\chi_{[-1+\frac{1}{k},1-\frac{1}{k}]}(c)\nabla_{\Gamma} c.
$$
Accordingly, for any $k>1$ and $p\geq 2$, $f_k=\left\vert \frac{\theta}{2}\varphi (c_k)\right\vert^{p-2}\frac{\theta}{2}\varphi (c_k)$ is well defined and belongs to $L^\infty_{H^1}$ and satisfies
$$
\nabla_{\Gamma}\left(\left\vert \frac{\theta}{2}\varphi (c_k)\right\vert^{p-2}\frac{\theta}{2}\varphi (c_k)\right)=(p-1)\left\vert \frac{\theta}{2}\varphi (c_k)\right\vert^{p-2}\frac{\theta}{2}\varphi^\prime(c_k)\nabla_{\Gamma} c_k.
$$
If we now set $\eta=\left\vert \frac{\theta}{2}\varphi (c_k)\right\vert^{p-2}\frac{\theta}{2}\varphi(c_k)$ in \eqref{mu6} then we infer that
\begin{align*}
&(p-1)\int_{\Gamma(t)}\left\vert \frac{\theta}{2}\varphi (c_k)\right\vert^{p-2}\frac{\theta}{2}\varphi^\prime(c_k)\nabla_{\Gamma}c\cdot \nabla_{\Gamma}c_k+\int_{\Gamma(t)}\rho\left\vert \frac{\theta}{2}\varphi (c_k)\right\vert^{p-2}\frac{\theta}{2}\varphi (c_k)\frac{\theta}{2}\varphi (c)=\int_{\Gamma(t)} \rho\widehat{w}\left\vert \frac{\theta}{2}\varphi (c_k)\right\vert^{p-2}\frac{\theta}{2}\varphi (c_k),
\end{align*}
where $ \widehat{w}=w+c$. Being $F_{ln}$ strictly convex, the first term in the left-hand side is nonnegative. Since $\varphi$ is increasing we infer
\begin{align}
\varphi (c_k)^2\leq \varphi (c)\varphi(c_k), \quad \forall k>1.
\label{uk2}
\end{align}
Regarding the right-hand side, by the Sobolev embedding $H^1(\Gamma(t))\hookrightarrow L^p(\Gamma(t))$ and \eqref{bounds}, we easily get
\begin{align*}
&\int_{\Gamma(t)} \rho \widehat{w}\left\vert \frac{\theta}{2}\varphi (c_k)\right\vert^{p-2}\frac{\theta}{2}\varphi(c_k)\leq \frac{1}{2C_\rho}\left\Vert\frac{\theta}{2}\varphi (c_k)\right\Vert^p_{L^p(\Gamma(t))}+C\Vert \widehat{w}\Vert^p_{L^p(\Gamma(t))}\leq \frac{1}{2C_\rho}\left\Vert\frac{\theta}{2}\varphi (c_k)\right\Vert^p_{L^p(\Gamma(t))}+C_p\Vert \widehat{w}\Vert^p_{H^1(\Gamma(t))},
\end{align*}
with $C_p>0$ depending on $p$.
Now, collecting the above estimates, being ${c}\in L^\infty_{H^1}$ and $\Vert w\Vert_{H^1(\Gamma(t))}\leq C$ for almost any $t\geq\tau$, and recalling, by \eqref{bounds}, that $0<\frac{1}{C_\rho}\leq \rho$, we immediately deduce (see \eqref{conv})
\begin{align}
\esssup_{t\in[\tau,T]}\Vert\varphi(c_k)\Vert_{L^p(\Gamma(t))}\leq C_p(T,\tau,p),\quad \forall p\in[2,\infty).
\label{phip2}
\end{align}
Consider now $g_k=\frac{\theta}{2}\varphi(c_k)e^{L\frac{\theta}{2}\vert \varphi(c_k)\vert}$, for some arbitrary $L>0$, and observe that
\begin{align*}
\nabla_{\Gamma}\left(\frac{\theta}{2}\varphi(c_k)e^{L\frac{\theta}{2}\vert\varphi(c_k)\vert}\right)=\frac{\theta}{2}\varphi^\prime(c_k)\left(1+L\frac{\theta}{2}\left\vert\varphi(c_k)\right\vert\right)e^{L\frac{\theta}{2}\vert\varphi(c_k)\vert}\nabla_{\Gamma} c_k.
\end{align*}
Therefore, using \eqref{mu6} with $\eta=g_k$, we find
\begin{align*}
&\int_{\Gamma(t)}\nabla_{\Gamma}c\cdot \nabla_{\Gamma}c_k \frac{\theta}{2}\varphi^\prime(c_k)\left(1+L\frac{\theta}{2}\left\vert\varphi(c_k)\right\vert\right)e^{L\frac{\theta}{2}\vert\varphi(c_k)\vert}+\int_{\Gamma(t)}\rho\frac{\theta}{2}\varphi(c)\frac{\theta}{2}\varphi(c_k)e^{L\frac{\theta}{2}\vert\varphi(c_k)\vert}=\int_{\Gamma(t)} \rho\widehat{w}\frac{\theta}{2}\varphi(c_k)e^{L\frac{\theta}{2}\vert \varphi(c_k)\vert}.
\end{align*}
Observe that the first term on the left-hand side is nonnegative. Hence, exploiting again \eqref{uk2} and \eqref{bounds}, we obtain in the end
%the others factors are positive,
\begin{align*}
\frac{1}{C_\rho}\int_{\Gamma(t)}\left(\frac{\theta}{2}\varphi(c_k)\right)^2e^{L\frac{\theta}{2}\vert\varphi(c_k)\vert}\leq \int_{\Gamma(t)}\rho \widehat{w}\frac{\theta}{2}\varphi(c_k)e^{L\frac{\theta}{2}\vert \varphi(c_k)\vert}
\end{align*}
By Lemma \ref{gener}, with $\rho_\star=\frac{1}{C_\rho}$, we get
\begin{align}
\int_{\Gamma(t)}\rho\vert  \widehat{w}\vert \left\vert \frac{\theta}{2}\varphi(c_k)\right\vert e^{L\left\vert \frac{\theta}{2}\varphi(c_k)\right\vert}\leq \frac{1}{2C_\rho}\int_{\Gamma(t)}\left\vert \frac{\theta}{2}\varphi(c_k)\right\vert^2 e^{L\left\vert \frac{\theta}{2}\varphi(c_k)\right\vert}+\int_{\Gamma(t)}e^{N\rho\vert  \widehat{w}\vert},
\label{pp3}
\end{align}
implying
\begin{align}
\label{MTineq3}
\frac{1}{2C_\rho}\int_{\Gamma(t)}\left\vert \frac{\theta}{2}\varphi(c_k)\right\vert^2 e^{L\left\vert \frac{\theta}{2}\varphi(c_k)\right\vert}\leq \int_{\Gamma(t)}e^{\rho N\vert  \widehat{w}\vert},
\end{align}
for any $L>0$ and some $N=N(L,C_\rho)$.
Exploiting \eqref{bounds}, the control over $J_t^0$ and then applying Lemma \ref{Trudi} with $u=C_\rho N \widetilde{\widehat{w}}$ and the manifold $\mathcal{M}=\Gamma_0$ (with the corresponding metric) we deduce
\begin{align*}
&\int_{\Gamma(t)}e^{\rho N\vert \widehat{w}\vert}\leq\int_{\Gamma(t)}e^{C_\rho N\vert \widehat{w}\vert}=\int_{\Gamma_0}e^{C_\rho N\vert\widetilde{\widehat{w}} \vert}J_t^0d\Gamma_0\leq C\int_{\Gamma_0}e^{C_\rho N\vert\widetilde{\widehat{w}} \vert}d\Gamma_0\leq Ce^{CN^2\Vert \widetilde{\widehat{w}}\Vert^2_{H^1(\Gamma_0)}}\leq Ce^{CN^2\Vert {\widehat{w}}\Vert^2_{H^1(\Gamma(t))}},
\end{align*}
being $(H^1(\Gamma(t)),\phi_t)_{t\in[0,T]}$ a compatible space.
On account of \eqref{phiprime}, taking $L=pC$ in \eqref{MTineq3} and recalling that $\Vert \widehat{w}\Vert_{H^1(\Gamma(t))}\leq C$ for almost any $t\geq\tau$, we end up with
$$
\Vert\varphi^\prime(c_k)\Vert_{L^p(\Gamma(t))}\leq C_p(T,\tau,p),
$$
which implies (see proof of Theorem \ref{main1})
\begin{align}
\esssup_{t\in[\tau,T]}\Vert\varphi^\prime(c)\Vert_{L^p(\Gamma(t))}\leq C_p(T,\tau,p),\quad \forall p\in[2,\infty).
\label{phip3}
\end{align}
Therefore, thanks to elliptic regularity, being $c\in L^\infty_{H^1}$ we get (see also \eqref{bounds})
$$
\Vert c\Vert_{H^2(\Gamma(t))}\leq \left(C+\Vert\Delta_{\Gamma(t)}c\Vert\right)\leq C\left(1+\Vert w\Vert+\Vert\varphi(c)\Vert\right)\leq C(T,\tau),
$$
for almost any $t\in[\tau,T]$. \vskip 2mm

\underline{Part (iii).} If we now apply the chain rule to $\varphi(c)$ (again we can obtain this by a truncation argument) we obtain
\begin{align*}
\tgrad  \varphi(c)=\varphi'(c)\tgrad c, \quad\text{for a.a.}\ t\in [\tau,T].
\end{align*}
Then, for almost any $t\in[\tau,T]$, we have that
$$
\Vert \tgrad  \varphi(c)\Vert_{L^p(\Gamma(t))} \leq \Vert  \varphi^\prime(c) \Vert_{L^{2p}(\Gamma(t))}\Vert \tgrad c \Vert_{L^{2p}(\Gamma(t))}\leq C_p(T,\tau,p),
$$
by the Sobolev embedding $H^2(\Gamma(t))\hookrightarrow W^{1,q}(\Gamma(t))$ which holds for every $q\geq2$. Therefore we infer
$$
\Vert \varphi(c)\Vert_{W^{1,p}(\Gamma(t))}\leq C_p(T,\tau,p),\quad \forall\,p\geq2,
$$
so that, choosing, e.g. $p=3$,
$$
\Vert \varphi(c)\Vert_{ L^\infty(\Gamma(t))}\leq C(T,\tau),\quad \text{ for a.a. }t\in[\tau,T].
$$
Therefore, being $c(t)\in H^2(\Gamma(t))\hookrightarrow C^0(\Gamma(t))$ for almost any $t\in[0,T]$, we can find $\xi=\xi(T,\tau)>0$ such that
\begin{align*}
\Vert c\Vert_{L^\infty(\Gamma(t))}\leq 1-{\xi}, \quad\text{ for a.a. }t\in[\tau,T],
\end{align*}	
that is, the strict separation property holds. This concludes the proof.
%\end{proof}

\appendix
\section{Proofs of some technical results}
\label{app:geom}
In this Appendix, we present the proofs of some results which were only stated in the main body of the paper.
\subsection{Proof of Proposition \ref{propder}}
\label{proof1}
Even though it is a straightforward result, we give here a short proof of the formula \eqref{der}.
We consider only sufficiently smooth functions $\eta,\phi$, since the general result can be obtained by a density argument. It is enough to prove this relation locally. Let $\Omega \subset\R{\color{black}^2}$ be an open set and let $X = X(\theta, t)$, $\theta\in\Omega$, $X(\cdot, t) : \Omega\to U\cap\Gamma(t)$ be a local regular parametrization of the open  set $U \cap \Gamma(t)$ (w.r.t. the induced metric) of the surface $\Gamma(t)$ which evolves so that $X_t = \mathbf{V}(X(\theta, t), t)$. %NB la parametrizzazione forma poi una partizione dell'unità e l'integrale è la somma dei vari U_i
	The induced metric $(g_{ij}), i,j=1,2,$ is
	given by $g_{ij} = X_{\theta_i} \cdot X_{\theta_j}$ with $g = \text{det}(g_{ij} )$. Note that, as usual, $g^{ij}=(g_{ij})^{-1}$. Define then $\mathbf{f}=\mathbf{V}_a^\tau\eta$ and set
	$$
	F(\theta,t)=\mathbf{f}(X(\theta,t),t),\qquad\Phi(\theta,t)=\phi(X(\theta,t),t),\qquad  \mathcal{V}(\theta,t)=\mathbf{V}(X(\theta,t),t).
	$$
	For the sake of simplicity we will omit the dependence on $\theta,t$.
	We have
	\begin{align*}
	\frac{d}{dt}\int_{U\cap\Gamma(t)}\mathbf{f}\cdot \tgrad \phi&=\frac{d}{dt}\int_{\Omega}F_lg^{ij}\Phi_{\theta_j}X^l_{\theta_i}\sqrt{g}d\theta\\&=\int_{\Omega}F_{l,t}g^{ij}\Phi_{\theta_j}X^l_{\theta_i}\sqrt{g}d\theta+\int_{\Omega}F_lg^{ij}_t\Phi_{\theta_j}X^l_{\theta_i}\sqrt{g}d\theta\\&+\int_{\Omega}F_lg^{ij}\Phi_{\theta_j,t}X^l_{\theta_i}\sqrt{g}d\theta+\int_{\Omega}F_lg^{ij}\Phi_{\theta_j}\mathcal{V}^l_{\theta_i}\sqrt{g}d\theta\\&+\int_{\Omega}F_lg^{ij}\Phi_{\theta_j}X^l_{\theta_i}\partial_t\sqrt{g}d\theta.
	\end{align*}
	Recalling that $g^{ij}_t=-g^{ik}g^{jl}(\mathcal{V}_{\theta_k}\cdot X_{\theta_l}+X_{\theta_k}\cdot\mathcal{V}_{\theta_l})$
	and
	$$
	\partial_t \sqrt{g}=\sqrt{g}g^{ij}X_{\theta_i}\cdot\mathcal{V}_{\theta_j},
	$$
	we obtain
\begin{align*}
\int_{\Omega}F_lg^{ij}_t\Phi_{\theta_j}&X^l_{\theta_i}\sqrt{g}=-\int_{\Omega}F_l\Phi_{\theta_j}X^l_{\theta_i}g^{ik}g^{jl}\mathcal{V}_{\theta_k}^rX^r_{\theta_l}\sqrt{g} -\int_{\Omega}F_l\Phi_{\theta_j}X^l_{\theta_i}g^{ik}g^{jl}\mathcal{V}_{\theta_l}^rX^r_{\theta_k}\sqrt{g} \\
&=-\int_{\Omega}F_l(g^{ik}X^l_{\theta_i}\mathcal{V}_{\theta_k}^r)(g^{jl}X^r_{\theta_l}\Phi_{\theta_j})\sqrt{g} -\int_{\Omega}F_l(g^{ik}X^l_{\theta_i}X^r_{\theta_k})(g^{jl}\Phi_{\theta_j}\mathcal{V}_{\theta_l}^r)\sqrt{g} \\
&=-\int_{U\cap\Gamma(t)}\mathbf{f}_l\underline{D}^{\Gamma(t)}_l\mathbf{V}_r\underline{D}^{\Gamma(t)}_r\phi-\int_{U\cap\Gamma(t)}\mathbf{f}_l(\delta_{lr}-\nu_l\nu_r)\tgrad \phi\cdot \tgrad \mathbf{V}_r,
\end{align*}
exploiting $\underline{D}^{\Gamma(t)}_lx^r=\delta_{lr}-\nu_l\nu_r$ and
\begin{align}
g^{jl}\Phi_{\theta_j}\mathcal{V}_{\theta_l}^r=g^{jl}g_{jk}g^{km}\Phi_{\theta_m}\mathcal{V}_{\theta_l}^r=(g^{km}\Phi_{\theta_m}X_{\theta_k}^n)(g^{jl}\mathcal{V}_{\theta_l}^rX_{\theta_j}^n).
\label{trick}
\end{align}
Then, on account of $g^{ij}X_{\theta_i}\cdot\mathcal{V}_{\theta_j}=(\nabla_{\Gamma}\cdot \mathbf{V})(X,\cdot)$, we obtain
$$
\int_{\Omega}F_lg^{ij}\Phi_{\theta_j}X^l_{\theta_i}\partial_t\sqrt{g} =\int_{U\cap\Gamma(t)}\mathbf{f}\cdot \tgrad \phi(\tgrad \cdot \mathbf{V}).
$$
Moreover, again by \eqref{trick}, we have
\begin{align}
\int_{\Omega}F_lg^{ij}\Phi_{\theta_j}\mathcal{V}^l_{\theta_i}\sqrt{g} =\int_{U\cap\Gamma(t)}\mathbf{f}_l\tgrad \phi\cdot \tgrad \mathbf{V}_l.
\label{deriv}
\end{align}
In conclusion, we obtain
\begin{align*}
&\frac{d}{dt}\int_{U\cap\Gamma(t)}\mathbf{f}\cdot \tgrad \phi=\int_{U\cap\Gamma(t)}\partial^\bullet\mathbf{f}\cdot \tgrad \phi+\int_{U\cap\Gamma(t)}\mathbf{f}\cdot \tgrad \partial^\bullet\phi\\&+\int_{U\cap\Gamma(t)}\mathbf{f}\cdot \tgrad \phi(\tgrad \cdot \mathbf{V})-\int_{U\cap\Gamma(t)}\mathbf{f}_l\underline{D}^{\Gamma(t)}_l\mathbf{V}_r\underline{D}^{\Gamma(t)}_r\phi,
\end{align*}
where we have exploited the fact that $\mathbf{f}\cdot \nu=0$, being a tangential vector. 
Indeed, setting $\mathbf{G}_r:=\tgrad \phi\cdot \tgrad \mathbf{V}_r$, then ${\mathbf{G}}_{\tau,l}:=(\delta_{lr}-\nu_l\nu_r)\mathbf{G}_r$
is the projection of $\mathbf{G}$ on the tangent space to $\Gamma(t)$. This entails that
$
\mathbf{f}\cdot \mathbf{G}=\mathbf{f}\cdot \mathbf{G}_\tau.
$
Therefore, we deduce
$$
-\int_{U\cap\Gamma(t)}\mathbf{f}_l(\delta_{lr}-\nu_l\nu_r)\tgrad \phi\cdot \tgrad \mathbf{V}_r=-\int_{U\cap\Gamma(t)}\mathbf{f}_l\tgrad \phi\cdot \tgrad \mathbf{V}_l,
$$
and this term simplifies with \eqref{deriv}. To conclude the proof it suffices to note that
$$
\partial^\bullet \mathbf{f}=\partial^\bullet\eta \mathbf{V}_a^\tau+\eta \partial^\bullet\mathbf{V}_a^\tau.
$$
\subsection{Proof of Proposition \ref{prop2}}
	Exploiting the same notation as in the proof of Proposition \ref{der} (Section \ref{proof1}) and considering, for simplicity, the case $\int_{\Gamma(t)}\rho\left\vert\tgrad {f}\right\vert^2$, for $f$ sufficiently regular, we can obtain
	$$
	\hat{\rho}\vert\nabla F\vert^2=\hat{\rho} g^{ij}F_{\theta_j}F_{\theta_i},
	$$
	with $\hat{\rho}(\theta,t)=\rho(X(\theta,t),t)$.
	Therefore, we find
		\begin{align*}
    \frac{d}{dt}\int_{U\cap\Gamma(t)}\rho\vert\tgrad {f}\vert^2 &=\frac{d}{dt}\int_{\Omega}\hat{\rho}F_lg^{ij}\Phi_{\theta_j}X^l_{\theta_i}\sqrt{g}d\theta\\
	&=\int_{\Omega}\hat{\rho}_{,t}F_lg^{ij}\Phi_{\theta_j}X^l_{\theta_i}\sqrt{g}d\theta+\int_{\Omega}\hat{\rho}(F_lg^{ij}\Phi_{\theta_j}X^l_{\theta_i}\sqrt{g})_{,t}d\theta\\
	&=\int_{U\cap\Gamma(t)}\partial^\bullet\rho\left\vert\tgrad {f}\right\vert^2+\int_{U\cap\Gamma(t)}\rho D_i\mathbf{V}_jD_i{f}D_j{f}\\
	&+2\int_{U\cap\Gamma(t)}\rho\tgrad {f}\cdot \tgrad \partial^\bullet{f}+\int_{U\cap\Gamma(t)}\rho\tgrad \cdot\mathbf{V}\vert\tgrad {f}\vert^2,
	\end{align*}
	where we have argued as in the proof of Proposition \ref{prop:transport} (see, e.g., \cite[Sec.5.1]{DziEll13-a}). In conclusion, to obtain the last identity in \eqref{der1}, it is enough to recall that $\partial^\bullet\rho=-\rho\tgrad \cdot\mathbf{V}$. The general case follows by polarization with respect to the inner product \eqref{ind}. The proof is ended.
\section{Two basic inequalities}\label{app:prelimtools}
\subsection{Moser-Trudinger inequality}
The Moser-Trudinger inequality for manifolds is given by (see \cite{Fon93})
\begin{lem}
	Let $(\mathcal{M},r)$ be a compact $n-dimensional$ Riemannian manifold, with $n\geq2$ and $r$ as a metric. Then there exists a constant $C$ depending only on $(\mathcal{M},r)$ and $\beta_0=\beta_0(n)>0$ such that
	$$
	\sup_{\int_\mathcal{M} udV=0,\ \int_{\mathcal{M}}\vert\nabla_{\mathcal{M}} u\vert^ndV\leq 1}\int_{\mathcal{M}}\text{e}^{\beta_0\vert u\vert^p}dV\leq C,
	$$
	where $p=\frac{n}{n-1}$.
	\label{trud}
\end{lem}
\begin{remark}
Notice that the constant $C$ depends not only on the volume of $\mathcal{M}$, $Vol(\mathcal{M})$, but also on its metric $r$. Therefore, in the case of $\mathcal{M}=\Gamma(t)$ it does not seem easy to find a constant $C$ independent of time.
\end{remark}
Adapting the proof proposed in \cite{NagSenYos97}, we can easily obtain the following
\begin{lem}
	Let $(\mathcal{M},r)$ be a compact $n-dimensional$ Riemannian manifold with metric $r$ and $n\geq2$. Let $u\in W^{1,n}(\mathcal{M})$. Then 
	$$
	\int_{\mathcal{M}}\text{e}^{\vert u\vert}dV\leq C_1\text{e}^{C_2\Vert u\Vert_{W^{1,n}(\mathcal{M})}^n},
	$$
	where the constant $C_1>0$ does not depend on $u$, but depends on $n$ and on $(\mathcal{M},r)$, whereas $C_2>0$ depends only on $n$.
	\label{Trudi}
\end{lem}
\begin{proof}
	Let us first consider $u\in W^{1,n}(\mathcal{M})$ with $\int_{\mathcal{M}} udV=0$. We then define $v=\frac{u}{\Vert \nabla u\Vert_{n}}$. Clearly $\Vert \nabla v\Vert_n=1$, therefore by Lemma \ref{trud} we get, for $p=\frac{n}{n-1}$,
	\begin{align}
	\int_{\mathcal{M}}\text{e}^{\beta_0\vert v\vert^p}dV\leq C.
	\label{p}
	\end{align}
	Now, recalling \eqref{p}, by Young's inequality we obtain
	\begin{align*}
	\int_{\mathcal{M}} e^{\vert u\vert}dV&=\int_{\mathcal{M}} e^{(p\beta_0)^{1/p}\vert v\vert(p\beta_0)^{-1/p}\Vert \nabla u\Vert_{n}}dV\leq \int_{\mathcal{M}} e^{\beta_0\vert v\vert^p+\frac{1}{n}(p\beta_0)^{-n/p}\Vert \nabla u\Vert_{n}^n}dV\leq C  e^{\frac{1}{\beta_n}\Vert \nabla u\Vert_{n}^n},
	\end{align*}
	having set $\beta_n=\left(\frac{1}{n}(p\beta_0)^{-n/p}\right)^{-1}=n\left(\frac{n\beta_0}{n-1}\right)^{(n-1)}$.
	To conclude the proof, let us fix $u\in W^{1,n}(\mathcal{M})$. Then, setting $w=u-(u)_{\mathcal{M}}$, so that $(w)_{\mathcal{M}}=0$, we can apply the result we have just proved, to infer
	$$
	\int_{\mathcal{M}} e^{\vert u\vert}dV\leq e^{\vert (u)_{\mathcal{M}}\vert}\int_{\mathcal{M}}e^{\vert w\vert}dV\leq C e^{\vert (u)_{\mathcal{M}}\vert+\frac{1}{\beta_n}\Vert \nabla u\Vert_{n}^n}\leq C_1e^{C_2\Vert  u\Vert_{W^{1,n}(\mathcal{M})}^n},
	$$
	concluding the proof.
\end{proof}

\subsection{Generalized Young's inequality}
We will use the following version of Young's inequality (see \cite[Appendix A]{ConGio20} and \cite{GioGraMir17}).
\begin{lem}
Let $L>0$ and $\rho_\star>0$ be given. Then, there exists $N=N(L,\rho_\star)>0$ such that
\begin{align*}
xye^{Ly}\leq e^{Nx}+\frac{\rho_\star}{2}y^2e^{Ly},\quad \forall x,y\geq 0.
\end{align*}
\label{gener}
	\end{lem}
\begin{proof}
We recall the generalized Young's inequality (see, e.g., \cite[Sec.8.2]{AdaFou03}): for any $a,b\geq0$,
$$
ab\leq \Phi(a)+\Psi(b),
$$
with, given $s\geq 0$,
$$
\Phi(s)=e^s-s-1,\qquad \Psi(s)=(1+s)\text{ln}(1+s)-s.
$$
Then, we choose $a=Nx$ and $b=N^{-1}ye^{Ly}$. We obtain, recalling that $\text{ln}(1+s)\leq s$ for any $s\geq0$,
\begin{align*}
xye^{Ly} &\leq e^{Nx}-Nx-1+(1+N^{-1}ye^{Ly})\text{ln}(1+N^{-1}ye^{Ly})-N^{-1}ye^{Ly}\\&=e^{Nx}-Nx-1+\text{ln}(1+N^{-1}ye^{Ly})-N^{-1}ye^{Ly}+N^{-1}ye^{Ly}\text{ln}(1+N^{-1}ye^{Ly})\\&\leq e^{Nx}+N^{-1}ye^{Ly}[\text{ln}(e^{Ly})+\text{ln}(e^{-Ly}+N^{-1}y)]\\&\leq e^{Nx}+N^{-1}ye^{Ly}[Ly+\text{ln}(1+N^{-1}y)]\\&\leq e^{Nx}+N^{-1}(L+N^{-1})y^2e^{Ly},
\end{align*}
and if we choose $N=N(L,\rho_\star)>\dfrac{L+\sqrt{L^2+2\rho_\star}}{\rho_\star}$ we finally obtain
$$
xye^{Ly}\leq e^{Nx}+\frac{\rho_\star}{2}y^2e^{Ly},
$$
concluding the proof.
	\end{proof}
\section{The embedding $\mathbb W^{\infty, 2}(H^2, H^1) \hookrightarrow C^0_{H^{3/2}}$}\label{app:embedding}
In this Appendix we aim to sketch a proof of the following result, which we make use of to obtain the extra regularity in Remark \ref{sepp}.

\begin{lem}
Assuming the following extra regularity for $\mathbf V$ and $\Phi$ (with respect to Assumption \textbf{A}$_{\Phi}$)
\begin{align*}
\mathbf V \in C^0( [0,T]; C^3(\R^{n+1}, \R^{n+1}) ) \quad \text{ and } \quad \Phi_0^{(\cdot)} \in C^1 ( [0,T]; C^3(\R^{n+1}, \R^{n+1}) ):
\end{align*}
\begin{itemize}
\item[(i)] The families $(H^2(\Gamma(t)), \phi_t)_{t}$, $(H^1(\Gamma(t)), \phi_t)_{t}$ and $(H^{3/2}(\Gamma(t)), \phi_t)_{t}$ are compatible.
\item[(ii)] The spaces $\W^{\infty,2}(H^2, H^1)$ and $\mathcal W^{\infty,2}(H^2(\Gamma_0), H^1(\Gamma_0))$ satisfy the evolving space equivalence.
\item[(iii)] We have the embedding $\W^{\infty,2}(H^2, H^1) \hookrightarrow C^0_{H^{3/2}}$. 
\end{itemize}
\label{C1}
\end{lem}

Before showing the proof we briefly recall some definitions. We denote
\begin{align*}
H^{3/2}(\Gamma(t)) = \left\{ u\in H^1(\Gamma(t)) \colon \gint\gint \dfrac{\left(\underline{D}^{\Gamma(t)}_i u(x) - \underline{D}^{\Gamma(t)}_i u(y)\right)^2 }{|x-y|^{n+1}} < \infty, \,\,  \forall i \right\}
\end{align*}
%with the inner product 
%\begin{align*}
%(u,v)_{H^{3/2}} &= (u,v)_{L^2} + (\tgrad u, \tgrad v)_{L^2} + \sum_{i=1}^3 \gint\gint \dfrac{\left(\partial^\Gamma_i u(x) - \partial^\Gamma_i u(y)\right) \left(\partial^\Gamma_i v(x) - \partial^\Gamma_i v(y)\right) }{|x-y|^{n+1}} \\
%&=(u,v)_{L^2} + (\tgrad u, \tgrad v)_{L^2} + \gint\gint \dfrac{\tgrad u(x) \cdot \tgrad v(x) - 2 \tgrad u(x) \cdot \tgrad v(y) + \tgrad u(y)\cdot \tgrad v(y) }{|x-y|^{n+1}}
%\end{align*}
with the norm
\begin{align*}
&\|u\|_{H^{3/2}}^2 = \|u\|_{L^2}^2 + \|\tgrad u\|_{L^2}^2 + \sum_{i=1}^{n+1} \gint\gint \dfrac{\left(\underline{D}_i u(x) - \underline{D}_i u(y)\right)^2 }{|x-y|^{n+1}} \\
&= \|u\|_{L^2}^2 + \|\tgrad u\|_{L^2}^2 + \gint\gint \dfrac{|\tgrad u(x)|^2 - 2 \tgrad u(x)\cdot\tgrad u(y) + |\tgrad u(y)|^2}{|x-y|^{n+1}}.
\end{align*}
Denote 
\begin{align*}
[u]_{3/2} := \gint\gint \dfrac{|\tgrad u(x)|^2 - 2 \tgrad u(x)\cdot\tgrad u(y) + |\tgrad u(y)|^2}{|x-y|^{n+1}}.
\end{align*}
%We will use the notation:
%\begin{align*}
%J^0_t := \mathrm{det} \, \mathbf{D}_\Gamma\Phi^0_t, \quad \mathbf{A}_t^0 :=
 %(\mathbf{D}_{\Gamma}\Phi_t^0)^{\T} \mathbf{D}_{\Gamma}\Phi_t^0 + \nu_0\otimes \nu_0, \quad  a_t^0 := \det \mathbf{A}_t^0.
%\end{align*}
%Recall also the following identities (see \cite[Section 7]{Ell}):
%    \begin{align}\label{eq:grad_pull}
 %       \tgrad \left(\phi_{-t} u\right) = (\mathbf{D}_{\Gamma}\Phi_t^0)^\T \phi_{-t} \left(\tgrad u\right) \quad \text{ and } \quad  \phi_{-t} \left(\tgrad u\right) &= \mathbf{D}_{\Gamma}\Phi_t^0 (\mathbf A_t^0)^{-1}\tgrad \left(\phi_{-t}  u\right)
  %  \end{align}
 %In the proof we will make use of the relation \eqref{pullback}.
%	Denoting for points $x,y\in \Gamma(t)$ the pullbacks $p = \phi_{-t}x$, $s = \phi_{-t} y$, we note that the assumptions on $\Phi_t^0$ imply that there exists $C_L>0$ such that 
%    \begin{align*}
%    | p - s | \geq C_L | x - y|.
%    \end{align*}
\begin{proof}

\underline{(i).} Compatibility of the first two pairs is established in \cite[Lemmas 7.2, 7.5]{AlpCaeDjuEll21}. For the third pair, we need to show:
\begin{itemize}
\item[a)] $\phi_t\colon H^{3/2}(\Gamma_0) \to H^{3/2}(\Gamma(t))$ and its inverse $\phi_{-t}\colon H^{3/2}(\Gamma(t)) \to H^{3/2}(\Gamma_0)$ are linear maps satisfying $\phi_0 = \text{Id}$ and are also bounded: there exists a constant $C_X>0$ s.t.
\begin{align*}%\label{eq:assOnEvolvingSpaces}
\begin{aligned}
\norm{\phi_t u}_{H^{3/2}} &\leq C_X\norm{u}_{H^{3/2}}, \quad \norm{\phi_{-t} u}_{H^{3/2}} &\leq C_X\norm{u}_{H^{3/2}}.
\end{aligned}
\end{align*}
\begin{small}
\begin{itemize}
\item[\underline{Proof.}] Linearity and the initial condition $\phi_0 = \text{Id}$ are immediate. Let $u\in H^{3/2}(\Gamma_0)$, then since $u\in H^1(\Gamma_0)$ we have from \cite[Lemma 7.2]{AlpCaeDjuEll21} that $\phi_t u \in H^1(\Gamma(t))$ as well as the bounds
\begin{align*}
\|\phi_t u\|_{L^2}^2 \leq C \|u\|_{L^2}^2 \quad \text{ and } \quad \|\phi_t u\|_{H^1}^2 \leq C \|u\|_{H^1}^2
\end{align*} 
where the constant $C$ depends only on an upper bound for the $C^1([0,T]\times\Gamma_0)$-norm of $\Phi^{(\cdot)}_0$. Now we aim to estimate 
\begin{align*}
[\phi_t u]_{3/2} = \gint\gint \dfrac{|\tgrad \phi_t u(x)|^2 - 2 \tgrad \phi_t u(x)\cdot\tgrad \phi_t u(y) + |\tgrad \phi_t u(y)|^2}{|x-y|^{n+1}}.
\end{align*}
We treat each term separately. Below $M>0$ is a general constant depending only on an upper bound on the $C^1([0,T]\times\Gamma_0)$-norm of $\Phi_0^{(\cdot)}$. We have by \eqref{pullback} %changing variables $x=\phi_t(p)$, $y=\phi_t(s)$,
\begin{align*}
\gint\gint \dfrac{|\tgrad \phi_t u(x)|^2}{|x-y|^{n+1}} &= \int_{\Gamma_0}\int_{\Gamma_0} \dfrac{|(\mathbf{D}_{\Gamma}\Phi_t^0)(p) (\mathbf A_t^0)^{-1}(p)\tzerograd u(p)|^2}{|\Phi_t^0(p) - \Phi_t^0(s)|^{n+1}} J^0_t (p)  J^0_t (s) \\
&\leq M \int_{\Gamma_0}\int_{\Gamma_0} \dfrac{|\tzerograd u (p)|^2}{|p-s|^{n+1}},
\end{align*}
in which we applied the bi-Lipschitz property of $\Phi_t^0$, which is ensured by its regularity, so that 
\begin{align}
\vert \Phi_t^0(p) - \Phi_t^0(s)\vert \geq C_L\vert p-s\vert,
\label{lip}
\end{align}
with $C_L>0$ a constant independent of time.
Similarly 
\begin{align*}
\gint\gint \dfrac{|\tgrad \phi_t u(y)|^2}{|x-y|^{n+1}} 
%&= \int_{\Gamma_0}\int_{\Gamma_0} \dfrac{|\mathbf{D}_{\Gamma}\Phi_t^0(s) (\mathbf A_t^0)^{-1}(s)\tgrad u(s)|^2}{|p - s|^{n+1}} J^0_t (p) J^0_t (s) \, \d\Gamma(p) \, \d\Gamma(s) \\
&\leq M \int_{\Gamma_0}\int_{\Gamma_0} \dfrac{|\tzerograd u (s)|^2}{|p-s|^{n+1}},
\end{align*}
and for the remaining term 
\begin{align*}
&\hskip -7mm \gint \gint \dfrac{\tgrad \phi_t u(x)\cdot\tgrad \phi_t u(y)}{|x-y|^{n+1}} \\
& \hskip -7mm = \int_{\Gamma_0}\int_{\Gamma_0} \dfrac{(\mathbf{D}_{\Gamma}\Phi_t^0)(p) (\mathbf A_t^0)^{-1}(p)\tzerograd u(p)\cdot (\mathbf{D}_{\Gamma}\Phi_t^0)(s) (\mathbf A_t^0)^{-1}(s)\tzerograd u(s)}{|\Phi_t^0(p) - \Phi_t^0(s)|^{n+1}} J^0_t (p) J^0_t (s) \\
&\hskip -7mm\leq M \int_{\Gamma_0}\int_{\Gamma_0} \dfrac{|\tzerograd u(p)| \, |\tzerograd u(s)|} {|p-s|^{\frac{n+1}{2}} \, |p-s|^{\frac{n+1}{2}}}   \\
&\hskip -7mm\leq \dfrac{M}{2} \int_{\Gamma_0}\int_{\Gamma_0} \dfrac{|\tzerograd u (p)|^2}{|p-s|^{n+1}} \, \d\Gamma(p) \, \d\Gamma(s) + \dfrac{M}{2} \int_{\Gamma_0}\int_{\Gamma_0} \dfrac{|\tzerograd u (s)|^2}{|p-s|^{n+1}} 
\end{align*}
Combining the above leads to $[\phi_t u]_{3/2} \leq M [u]_{3/2}$, proving that $\phi_t$ indeed maps $H^{3/2}(\Gamma_0)$ into $H^{3/2}(\Gamma(t))$ and is bounded. The calculations for $\phi_{-t}$ are analogous, since all the properties exploited are shared by $\Phi_0^t$ as well.
\end{itemize}
\end{small}

\item[b)] for all $u \in H^{3/2}(\Gamma_0)$, the map $t \mapsto \norm{\phi_t u}_{H^{3/2}}$ is measurable.
\begin{small}
\begin{itemize}
\item[\underline{Proof.}] It remains to show that $t\mapsto [\phi_t u]_{H^{3/2}}$ is measurable. From the calculations above it follows that the seminorm $[\phi_t u]_{H^{3/2}}$ is the sum of three components:
\begin{itemize}
\item[$\bullet$] The first term is the integral of 
\begin{align*}
g_1(t, p, s) := \dfrac{|(\mathbf{D}_{\Gamma}\Phi_t^0)(p) (\mathbf A_t^0)^{-1}(p)\tzerograd u(p)|^2}{|\Phi_t^0(p) - \Phi_t^0(s)|^{n+1}} J^0_t (p)  J^0_t (s)
\end{align*}
which is a continuous function of $t\in [0,T]$ for almost every $(p,s)\in \Gamma_0\times \Gamma_0$ and can be (uniformly in time) dominated, thanks to \eqref{lip}, as 
\begin{align*}
|g_1(t,p,s)| \leq M \dfrac{|\tzerograd u(p)|^2}{|p-s|^{n+1}} \in L^1(\Gamma_0\times \Gamma_0).
\end{align*}
Therefore, by Lebesgue's dominated convergence Theorem, $$t\mapsto \int_{\Gamma_0}\int_{\Gamma_0} g_1(t,p,s)$$ is continuous and thus measurable.
\item[$\bullet$] Similarly for the last term, the function
\begin{align*}
g_3(t, p, s) := \dfrac{|(\mathbf{D}_{\Gamma}\Phi_t^0)(p) (\mathbf A_t^0)^{-1}(p)\tzerograd u(s)|^2}{|\Phi_t^0(p) - \Phi_t^0(s)|^{n+1}} J^0_t (p)  J^0_t (s)
\end{align*}
is continuous for almost every $(p,s)\in \Gamma_0\times \Gamma_0$ and can be (uniformly in time) dominated, thanks again to \eqref{lip}, as 
\begin{align*}
|g_3(t,p,s)| \leq M \dfrac{|\tzerograd u(s)|^2}{|p-s|^{n+1}} \in L^1(\Gamma_0\times \Gamma_0),
\end{align*}
giving that $$t\mapsto \int_{\Gamma_0}\int_{\Gamma_0} g_3(t,p,s)$$ is continuous and then measurable.
\item[$\bullet$] Finally, the integrand in the middle term, say $g_2(t,p,s)$,
%\begin{align*
%g_2(t, p, s) := \dfrac{(\mathbf{D}_{\Gamma}\Phi_t^0)(p) (\mathbf A_t^0)^{-1}(p)\tzerograd u(p)\cdot (\mathbf{D}_{\Gamma}\Phi_t^0)(s) (\mathbf A_t^0)^{-1}(s)\tzerograd u(s)}{|\Phi_t^0(p) - \Phi_t^0(s)|^{n+1}} J^0_t (p) J^0_t (s) 
%\end{align*}
is also a continuous function of $t\in [0,T]$ for almost every $(p,s)\in \Gamma_0\times \Gamma_0$ and can now be (uniformly in time) dominated, again by \eqref{lip}, as 
\begin{align*}
|g_2(t,p,s)| \leq \dfrac{M}{2} \left(\dfrac{|\tzerograd u (p)|^2}{|p-s|^{n+1}}  +\dfrac{|\tzerograd u (s)|^2}{|p-s|^{n+1}} \right) \in L^1(\Gamma_0\times \Gamma_0),
\end{align*}
therefore $$t\mapsto \int_{\Gamma_0}\int_{\Gamma_0} g_2(t,p,s)$$ is  continuous and hence measurable.
\end{itemize}
\end{itemize}
\end{small}
This proves that also $(H^{3/2}(\Gamma(t)), \phi_t)_t$ is compatible.
\end{itemize}

\vskip 2mm

\underline{(ii).} The conditions verified in \cite[Proposition 7.7]{AlpCaeDjuEll21} show that the improved result in \cite[Theorem 5.10]{AlpCaeDjuEll21} holds true, giving the evolving space equivalence between the spaces $\W^{\infty,2}(H^2, H^1)$ and $\mathcal W^{\infty,2}(H^2(\Gamma_0), H^1(\Gamma_0))$.

\vskip 2mm

\underline{(iii).} This follows from the previous result together with the classical embedding $$\mathcal W^{\infty,2}(H^2(\Gamma_0), H^1(\Gamma_0))\hookrightarrow C^0( [0,T]; H^{3/2}(\Gamma_0) ).$$ Indeed, if $u\in W^{\infty,2}(H^2, H^1)$, then by the evolving space equivalence in (ii) we obtain 
\begin{align*}
\phi_{-(\cdot)} u \in \mathcal W^{\infty,2}(H^2(\Gamma_0), H^1(\Gamma_0)) \hookrightarrow C^0( [0,T]; H^{3/2}(\Gamma_0) ),
\end{align*}
and since $(H^{3/2}(\Gamma(t)), \phi_t)_t$ is compatible we deduce by definition $u  \in C^0_{H^{3/2}}$, which concludes the proof.
\end{proof}

\noindent
\textbf{Acknowledgements.}
{\color{black}The authors are grateful to the anonymous referee for the careful reading of the manuscript as well as for the valuable comments, particularly the second referee who pointed out the value of expanding on the modelling aspects. We would also like to thank Thomas Sales (University of Warwick) for detecting a mistake in one of our original proofs.} M. Grasselli and A. Poiatti have been partially funded by MIUR-PRIN Grant 2020F3NCPX "Mathematics for Industry 4.0 (Math4I4)". M. Grasselli and A. Poiatti are also members of Gruppo Nazionale per l'Analisi Matematica, la Probabilit\`{a} e le loro Applicazioni (GNAMPA), Istituto Nazionale di Alta Matematica (INdAM).

\bibliographystyle{abbrv}
\bibliography{main}
\end{document}